\documentclass[11pt]{article}
\usepackage{amsfonts}
\usepackage{color}
\usepackage{amsmath,amssymb,comment}
\usepackage[hidelinks]{hyperref}
\newtheorem{theo}{Theorem}[section]
\newtheorem{lem}[theo]{Lemma}

\newtheorem{prop}[theo]{Proposition}

\newtheorem{defi}[theo]{Definition}

\newcommand{\mysection}[1]{\section{#1} \setcounter{equation}{0}}
\newcommand{\proof}{{\sc Proof.} \quad}
\newcommand{\proofc}{{\sc Proof} \ }
\newcommand{\be}{\begin{equation} \label}
\newcommand{\ee}{\end{equation}}
\newcommand{\bea}{\begin{eqnarray}\label}
\newcommand{\eea}{\end{eqnarray}}
\newcommand{\bas}{\begin{eqnarray*}}
\newcommand{\eas}{\end{eqnarray*}}
\newcommand{\bit}{\begin{itemize}}
\newcommand{\eit}{\end{itemize}}
\newcommand{\qed}{\hfill$\Box$ \vskip.2cm}
\newcommand{\nn}{\nonumber}
\newcommand{\R}{\mathbb{R}}
\newcommand{\N}{\mathbb{N}}
\newcommand{\pO}{\partial\Omega}

\newcommand{\eps}{\varepsilon}

\newcommand{\supp}{{\rm supp} \, }

\newcommand{\wto}{\rightharpoonup}
\newcommand{\wsto}{\stackrel{\star}{\rightharpoonup}}
\newcommand{\hra}{\hookrightarrow}
\newcommand{\io}{\int_\Omega}
\newcommand{\na}{\nabla}
\newcommand{\Del}{\Delta}
\newcommand{\del}{\delta}
\newcommand{\al}{\alpha}
\newcommand{\vt}{\vartheta}
\newcommand{\lam}{\lambda}
\newcommand{\Lam}{\Lambda}
\newcommand{\sig}{\sigma}
\newcommand{\pa}{\partial}
\newcommand{\bom}{\overline{\Omega}}
\newcommand{\Om}{\Omega}

\newcommand{\ov}{\overline}

\newcommand{\wh}{\widehat}

\newcommand{\hs}{\hspace*}
\newcommand{\sm}{\setminus}

\newcommand{\vp}{\varphi}
\newcommand{\lbal}{\left\{ \begin{array}{l}}
\newcommand{\lball}{\left\{ \begin{array}{ll}}
\newcommand{\ear}{\end{array} \right.}

\newcommand{\abs}{\\[5pt]}

\newcommand{\adb}{\allowdisplaybreaks}

\newcommand{\tme}{T_{max,\eps}}

\newcommand{\F}{{\mathcal{F}}}

\newcommand{\ovv}{\ov{v}}
\newcommand{\ovk}{\ov{v}_k}

\newcommand{\ueps}{u_\eps}
\newcommand{\veps}{v_\eps}

\newcommand{\heps}{h_\eps}
\newcommand{\feps}{f_\eps}
\newcommand{\geps}{g_\eps}
\newcommand{\yeps}{y_\eps}
\newcommand{\zeps}{z_\eps}
\newcommand{\keps}{\kappa_\eps}
\newcommand{\Keps}{K_\eps}
\newcommand{\Teps}{\Theta_\eps}
\newcommand{\Tepst}{\Theta_{\eps t}}
\newcommand{\leps}{\ell_\eps}
\newcommand{\helleps}{\wh{\ell}_\eps}
\newcommand{\zd}{\zeta_\delta}
\newcommand{\Tinf}{\Theta_\infty}

\renewcommand{\div}{{\rm div} \,}
\newcommand{\nas}{\nabla^s}
\newcommand{\lan}{\langle}
\newcommand{\ran}{\rangle}
\newcommand{\D}{\mathbb{D}}
\newcommand{\C}{\mathbb{C}}
\newcommand{\E}{\mathbb{E}}
\newcommand{\B}{\mathbb{B}}
\newcommand{\A}{\mathbb{A}}
\newcommand{\M}{\mathcal{M}}
\newcommand{\kD}{k_{\D}}
\newcommand{\kC}{k_{\C}}

\newcommand{\xis}{\xi_\star}
\newcommand{\uK}{\underline{K}}
\newcommand{\epss}{\eps_\star}
\newcommand{\epsss}{\eps_{\star\star}}
\newcommand{\epssss}{\eps_{\star\star\star}}
\newcommand{\lleps}{\wh{\ell}_\eps}
\begin{document}
\adb
\title{Global solvability and stabilization\\
in multi-dimensional small-strain nonlinear thermoviscoelasticity}
\author{
Michael Winkler\footnote{Corresponding author. email: michael.winkler@math.uni-paderborn.de}\\
{\small Universit\"at Paderborn, Institut f\"ur Mathematik}\\
{\small 33098 Paderborn, Germany} }
\date{}
\maketitle
\begin{abstract}
\noindent 
  Despite considerable developments in the literature of the past decades, a standing open problem in the analysis
  of continuum mechanics appears to consist of determining how far the prototypical model for small-strain thermoviscoelastic
  evolution in Kelvin-Voigt materials with inertia, as given by
  \bas
	\lbal
	u_{tt} = \mu \Del u_t + (\lam+\mu)\na\na\cdot u_t
	+ \wh{\mu} \Del u + (\wh{\lam}+\wh{\mu}) \na\na\cdot u
	- B\na\Theta, \\[1mm]
	\kappa \Theta_t = D\Del\Theta + \mu |\na u_t|^2 + (\lam+\mu) |\div u_t|^2 - B\Theta \div u_t,
	\ear
	\qquad \qquad (\star)
  \eas
  is globally solvable in multi-dimensional settings and for initial data of arbitrary size.\abs
  The present manuscript addresses this in the context of an initial value problem in smoothly bounded $n$-dimensional
  domains with $n\ge 2$, posed under homogeneous boundary conditions of Dirichlet type for the displacement variable $u$,
  and of Neumann type for the temperature $\Theta$. 
  Within suitably generalized concepts of solvability, global existence of solutions is shown without any size restrictions
  on the data, and for a system actually more general than ($\star$) by, inter alia, allowing the heat capacity $\kappa$
  to depend on $\Theta$.
  Apart from that, results on large time behavior are derived which particularly assert stabilization of $\Theta$ toward
  a spatially homogeneous limit. \abs
  Besides on standard features related to energy conservation and entropy production, in its core parts the analysis relies
  on an evolution property of certain logarithmic refinements of classical entropy functionals, to the best of our knowledge
  undiscovered in precedent literature and possibly of independent interest.\abs
\noindent {\bf Key words:} thermoviscoelasticity; generalized solution; large time behavior\\
{\bf MSC 2020:} 74F05 (primary); 35D30, 35B40, 74A15, 80A17 (secondary)
\end{abstract}
%
%
%
%

\tableofcontents

\newpage
\section{Introduction}\label{intro}
Models for thermoviscoelastic evolution have been challenging mathematical research for decades. 
After the pioneering works \cite{dafermos_hsiao_smooth} and \cite{dafermos}
in this regard, refined knowledge on issues of 
global solvability (\cite{racke_zheng}, \cite{kim}, \cite{jiang_QAM1993}, \cite{watson}) 
and also of large-time behavior (\cite{racke_zheng}, \cite{hsiao_luo}) 
has led to a meanwhile quite comprehensive picture in spatially one-dimensional settings.\abs
The literature on higher-dimensional relatives seems considerably sparser, however, even for models which concentrate
on small-strain scenarios. 
For the classical model for nonlinear thermoviscoelastic dynamics at small strains in isotropic Kelvin-Voigt materials, for instance,
as given by (\cite{roubicek}, \cite{roubicek_SIMA}, \cite{pawlow_zajaczkowski_SIMA})
\be{00}
	\lbal
	u_{tt} = \mu \Del u_t + (\lam+\mu)\na\na\cdot u_t
	+ \wh{\mu} \Del u + (\wh{\lam}+\wh{\mu}) \na\na\cdot u
	- B\na\Theta, \\[1mm]
	\kappa \Theta_t = D\Del\Theta + \mu |\na u_t|^2 + (\lam+\mu) |\div u_t|^2 - B\Theta \div u_t
	\ear
\ee
with constant parameters $\mu,\lam,\wh{\mu},\wh{\lam},B,\kappa$ and $D$,
small-data global smooth solutions were constructed in \cite{shibata} and \cite{gaw_zaj_TMNA}, but
a theory of global solutions in the presense of large data seems yet lacking unless restrictions are imposed
by either resorting to radially symmetric frameworks (\cite{gawinecki}), or neglecting core mechanisms by
assuming that $\mu=\lam=0$ and $\wh{\lam}+\wh{\mu}=0$ (\cite{cieslak_CVPD}).
While various results are available 
for slightly modified variants, e.g.~taking certain enhancements of heat capacities at large temperatures or of
viscous dissipation at large strains into account, or assuming thermal dilation to undergo 
suitably strong saturation near sites of large temperatures
(\cite{pawlow_zajaczkowski_SIMA}, \cite{roubicek}, \cite{blanchard_guibe}, \cite{paw_zaj2017}, \cite{gaw_zaj_CPAA}), 
for the original system (\ref{00})
the knowledge beyond local-in-time solvability (\cite{bonetti_bonfanti}) seems comparatively thin yet.\abs
This situation differs considerably from that in related models that neglect inertial effects.
Even for large-strain versions of such systems that involve fourth-order mechanisms and couplings yet more strongly
nonlinear than those in (\ref{00}), namely, interesting recent developments
have not only asserted global solvability under mild assumptions on the respective ingredients (\cite{mielke_roubicek}),
but also provided far-reaching information on qualitative aspects such as positivity of temperature and convergence 
toward solutions of linearized systems in some limiting situations involving a parameter that measures deviation from
constant equilibria (\cite{friedrich2024}, \cite{friedrich2022}).
Results on further relatives, partially as well involving fourth-order expressions, 
alternatively including additional internal variables that describe plasticity, or addressing adhesive contact,
can be found in
\cite{roubicek_nodea2013}, \cite{roubicek_SIMA}, \cite{owczarek_wielgos}, \cite{bartels_roubicek},
\cite{paoli_petrov}, \cite{paoli_petrov_gamm}, \cite{roubicek_dcdss13}, 
\cite{rossi_roubicek}, \cite{pawlow2000} and \cite{guo_zhu}, for instance.\abs
{\bf Main results I: Existence of global large-data solutions.} \quad
A key challenge for any expedient analysis of (\ref{00}) seems to consist in an appropriate handling of the quadratic
dependence of the heat source therein on $\na u_t$, and the intention to suitably cope with this forms
a core intention of the present manuscript.
Under assumptions on the ingredients $\D,\C$ and $\B$ slightly more general than those underlying (\ref{00}),
this will be addressed in the framework of the problem
\be{0}
	\lball
	u_{tt} = \div (\D : \nas u_t) + \div ( \C : \nas u) - \div (\Theta \B) + f(x,t),
	\qquad & x\in\Om, \ t>0, \\[1mm]
	\kappa(\Theta) \Theta_t = D\Del \Theta + \lan \D : \nas u_t , \nas u_t \ran - \Theta \lan \B , \nas u_t \ran + g(x,t),
	\qquad & x\in\Om, \ t>0, \\[1mm]
	u=0, \quad \frac{\pa \Theta}{\pa\nu} =0,
	\qquad & x\in\pO, \ t>0, \\[1mm]
	u(x,0)=u_0(x), \quad u_t(x,0)=u_{0t}(x), \quad \Theta(x,0)=\Theta_0(x),
	\qquad & x\in\Om,
	\ear
\ee
posed in a smoothly bounded domain $\Om\subset\R^n$, with $D>0$, with tensors 
$\D\in \R^{n\times n\times n\times n}$,
$\C\in \R^{n\times n\times n\times n}$
and
$\B\in\R^{n\times n}$, with given forcing terms $f$ and $g$, with a possibly 
temperature-dependent heat capacity $\kappa:[0,\infty)\to [0,\infty)$, and with prescribed initial distributions
$u_0, u_{0t}$ and $\Theta_0$.
Here and below, the symmetric gradient of a vector field $\vp:\R^n\to\R^n$ is abbreviated by writing 
$\nas\vp=\frac{1}{2}(\na\vp+(\na\vp)^T)$, and the products between a tensor $\E\in \R^{n\times n\times n \times n}$ 
and a matrix $\A\in\R^{n\times n}$, and between matrices
$\A_1\in \R^{n\times n}$ and $\A_2\in\R^{n\times n}$, are denoted by $\E:\A$ and $\lan \A_1,\A_2\ran$, respectively.\abs
Two essentially well-known basic features of (\ref{0}), enjoyed by sufficiently regular solutions thereof, 
are expressed in the energy evolution law
\be{energy}
	\frac{d}{dt} \bigg\{
	\frac{1}{2}\io |u_t|^2
	+ \frac{1}{2} \io \lan \C:\nas u,\nas u\ran
	+ \io K(\Theta) \bigg\}
	= \io f\cdot u_t
	+ \io g,
	\qquad
	K(\Theta):=\int_0^\Theta \kappa(\sig) d\sig,
\ee
and the identity
\be{ent}
	\frac{d}{dt} \io \ell(\Theta)
	= D \io \frac{|\na\Theta|^2}{\Theta^2}
	+ \io \frac{\lan \D: \nas u_t,\nas u_t\ran}{\Theta}
	+ \io \frac{g}{\Theta},
	\qquad \ell(\Theta):=\int_1^\Theta \frac{\kappa(\sig)}{\sig} d\sig,
\ee
quantifying entropy production.
Indeed, properties of this type have served as fundamental ingredients both for existence theories and for qualitative studies
on (\ref{0}) and close relatives in one- and higher-dimensional settings (\cite{roubicek}, \cite{racke_zheng}, 
\cite{cieslak_CVPD}, \cite{cieslak_MAAN}).
In particular, in \cite{cieslak_CVPD} it has been seen that in the absence of viscosity effects and for constant heat capacities, 
that is, when in (\ref{0}) we have $\D=0$ and $\kappa\equiv const.$, the $L^1$ bound for $\Theta$
then gained from (\ref{energy}), in conjunction with regularity properties entailed 
by the presence of the diffusion-induced contribution to (\ref{ent}), facilitates an appropriate treatment of the 
nonlinearities related to thermal dilation, and thus the construction of some global solutions in three-dimensional domains.\abs
In the current more complex setting of the full model (\ref{0}) including viscous effects, however, the identities
(\ref{energy}) and (\ref{ent}) alone seem insufficient to provide suitable information on regularity of the 
corresponding contribution $\lan \D : \nas u_t , \nas u_t \ran$ to the heat generation process.
The core of our approach toward a theory of large-data solutions is now based on the observation that a certain logarithmically
enhanced modification of the entropy functional in (\ref{ent}) still satisfies a comparable inequality along trajectories,
slightly weakened in the sense of failing to express genuine nondecrease but yet sufficient to imply bounds for the associated
dissipation rate functionals at least on finite time intervals. 
Specifically, the cornerstone of our existence theory can be found in an inequality of the form
\bea{corner}
	\frac{d}{dt} \io \wh{\ell}(\Theta)
	\ge \frac{1}{C} \io \frac{\ln^2(\Theta+M) |\na\Theta|^2}{(\Theta+M)^2}
	+ \frac{1}{C} \io \frac{\ln^2(\Theta+M) |\nas u_t|^2}{\Theta+M} 
	- C \io |u_t|^2
	- C,
\eea
valid with some $M>1$ and $C>0$ and 
\be{whell}
	\wh{\ell}(\Theta):=\int_0^\Theta \frac{\ln^2(\sig+M) \kappa(\sig)}{\sig+M} d\sig
\ee
(see Lemma \ref{lem12}).
When combined with (\ref{energy}), this will imply an estimate of the form
\be{LlogL}
	\int_0^T \io |\nas u_t| \ln \big( |\nas u_t|+e\big)
	\le C(T),
	\qquad T>0,
\ee
(Lemma \ref{lem3}, Lemma \ref{lem6} and Lemma \ref{lem23}), 
particularly ensuring a uniform integrablity property of symmetric gradients
in suitably regularized variants of (\ref{0}) (cf.~(\ref{0eps}));
through Lions' lemma, an accordingly entailed weak $L^1$ compactness feature will allow for the identification of limits in 
corresponding viscosity-related 
contributions to certain integral identities for suitably renormalized versions of the temperature variable
(Lemma \ref{lem24} and Lemma \ref{lem244}).\abs
In the framework of an adequately generalized concept of solvability, to be specified in Definition \ref{dw} below,
this will lead to the first of our main results which 
asserts global existence of solutions to (\ref{0}), actually in domains of arbitrarily high dimension, 
within a large class of ingredients which particularly covers the scenario in (\ref{00}):
\begin{theo}\label{theo37}
  Let $n\ge 2$ and $\Om\subset\R^n$ be a bounded domain with smooth boundary,
  let $D>0$, and suppose that
  \be{DCB}
	\lbal
	\D=(\D_{\mathbf{ijkl}})_{(\mathbf{i}, \mathbf{j}, \mathbf{k},\mathbf{l}) 
		\in \{1,...,n\}^4} \in \R^{n\times n \times n \times n}, \\[1mm]
	\C=(\C_{\mathbf{ijkl}})_{(\mathbf{i}, \mathbf{j}, \mathbf{k},\mathbf{l}) 
		\in \{1,...,n\}^4} \in \R^{n\times n \times n \times n}
	\qquad \mbox{and} \\[1mm]
  	\B=(\B_{\mathbf{ij}})_{(\mathbf{i},\mathbf{j})\in \{1,...,n\}^2} \in \R^{n\times n}_{sym}
	\ear
  \ee
  are such that
  \be{DCsym}
	\D_{\mathbf{ijkl}}=\D_{\mathbf{klij}}=\D_{\mathbf{jikl}}
	\quad \mbox{and} \quad
	\C_{\mathbf{ijkl}}=\C_{\mathbf{klij}}=\C_{\mathbf{jikl}}
	\qquad \mbox{for all } (\mathbf{i}, \mathbf{j}, \mathbf{k},\mathbf{l}) \in\{1,...,n\}^4,
  \ee
  and that with some $\kD>0$ and $\kC>0$ we have
  \be{DC_lower}
	\lan \D:\A,\A\ran \ge \kD |\A|^2
	\quad \mbox{and} \quad
	\lan \C:\A,\A\ran \ge \kC |\A|^2
	\qquad \mbox{for all } \A\in\R^{n\times n}_{sym}.
  \ee
  Moreover, assume that
  \be{fg_reg}
	f\in L^1_{loc}([0,\infty);L^2(\Om;\R^n))
	\qquad \mbox{and} \qquad
	0 \le g\in L^1_{loc}(\bom\times [0,\infty)),
  \ee
  and suppose that
  \be{kappa_reg}
	\kappa\in C^0([0,\infty))
	\quad \mbox{ is such that $\kappa>0$ on } (0,\infty),
  \ee
  and that
  \be{131.1}
	\liminf_{\xi\to\infty} \kappa(\xi) >0.
  \ee
  Then whenever
  \be{init}
	\lbal
	u_0\in W_0^{1,2}(\Om;\R^n), \\[1mm]
	u_{0t} \in L^2(\Om;\R^n) 
	\qquad \mbox{and} \\[1mm]
	\Theta_0: \Om\to\R
	\quad \mbox{is such that} \quad 
	\Theta_0\not\equiv 0,
	\quad 
	K(\Theta_0)\in L^1(\Om)
 	\quad \mbox{and} \quad 
	\ell(\Theta_0) \in L^1(\Om),
	\ear
  \ee
  with 
  \be{K}
	K(\xi):=\int_0^\xi \kappa(\sig) d\sig,
	\qquad \xi\ge 0,
  \ee
  and
  \be{ell}
	\ell(\xi):= \int_1^\xi \frac{\kappa(\sig)}{\sig} d\sig,
	\qquad \xi>0,
  \ee
  there exist
  \be{reg37}
	\lbal
	u \in C^0([0,\infty);L^2(\Om;\R^n)) \cap L^\infty_{loc}([0,\infty);W_0^{1,2}(\Om;\R^n))
	\qquad \mbox{and} \\[1mm]
	\Theta\in L^\infty_{loc}([0,\infty);L^1(\Om))
	\ear
  \ee
  such that $\Theta>0$ a.e.~in $\Om\times (0,\infty)$, and that $(u,\Theta)$ 
  forms a generalized solution of (\ref{0}) in the sense of Definition \ref{dw}.
\end{theo}
A slight modification of our argument will additionally identify some situations in which the potentially nontrivial 
measure in the above result can actually be dropped. 
We emphasize that in spatially planar settings, 
the following statement in this regard asserts the existence of such integrable solutions 
actually for heat capacities which are even allowed to undergo some mildly fast decay at large temperatures,
as appearing to be of relevance in some materials which exhibit heat capacity anomalies
(\cite{drotziger}, \cite{griffel}, \cite{troyanchuk}).
In three-dimensional domains, unboundedness of $\kappa$ is required, but arbitrarily slow growth is sufficient 
for a corresponding conclusion; here and below, as usual we let $\xi_+:=\max\{\xi,0\}$ for $\xi\in\R$:
\begin{theo}\label{theo38}
  Let $n\ge 2$ and $\Om\subset\R^2$ be a bounded domain with smooth boundary,
  let $D>0$, and suppose that (\ref{DCB}), (\ref{DCsym}) and (\ref{DC_lower}) hold with some $\kD>0$ and $\kC>0$,
  and that $f,g$ and $\kappa$ satisfy (\ref{fg_reg}) and (\ref{kappa_reg}) as well as
  \be{38.1}
	\kappa(\xi) \cdot \ln^{(3-n)_+} \xi \to + \infty
	\qquad \mbox{as } \xi\to\infty.
  \ee
  Then for any choice of initial data which are such that (\ref{init}) is valid with $K$ as in (\ref{K}),
  one can find a global integrable solution $(u,\Theta)$
  of (\ref{0}), in the sense of Definition \ref{dw} ii), such that $\Theta>0$ a.e.~in $\Om\times (0,\infty)$.
  In addition,
  \be{reg388}
	\Theta \in L^\infty_{loc}([0,\infty);L^1(\Om))
	\qquad
	\mbox{if} \quad
	\liminf_{\xi\to\infty} \kappa(\xi) >0.
  \ee
\end{theo}
{\bf Remark.} \quad
  i) \ To put these results in perspective, 
  we once more emphasize that in the prototypical case when $\kappa\equiv 1$,
  available literature does not provide any result on global existence of large-data solutions. 
  Theorem \ref{theo37} and the two-dimensional part of Theorem \ref{theo38} particularly do not only 
  cover this scenario as a special case, but go considerably beyond by addressing situations
  where in accordance with the Debye law the heat capacity undergoes some cubic-like degeneracies at low temperatures
  (\cite{debye}).
  More generally, we recall that
  in the presense of heat capacities growing at large temperatures according to a power-type law of the form 
  $\kappa(\xi)\simeq \xi^\omega$, the mildest conditions on $\omega$ appearing in precedent results on global large-data solvability
  in boundary value problems for the PDE system in (\ref{0}) seem to require that
  \bas
	\omega>\lball
	0
	\qquad & \mbox{if } n=2, \\[1mm]
	\frac{1}{2}
	\qquad & \mbox{if } n\ge 3.
	\ear
  \eas
  In fact, in \cite{roubicek} some global very weak solutions to (\ref{0}) 
  have been constructed under the assumptions that $n\ge 2$ and that
  $\kappa$ suitably generalizes the non-singular choice $\kappa(\xi)=\kappa_0(\xi+1)^\omega$, $\xi\ge 0$, 
  with some $\omega>\frac{n-2}{2}$.
  More regular solutions were found for exactly power-type heat capacities of the singular form $\kappa(\xi)=\xi^\omega$
  when $\omega\in (\frac{1}{2},1]$ in \cite{gaw_zaj_CPAA} in two- and in \cite{pawlow_zajaczkowski_SIMA} and \cite{paw_zaj2017} 
  in three-dimensional domains.
  Certain weak-renormalized solvability in a related boundary value problem of fully Dirichlet type has been asserted for
  $n\ge 2$ and arbitrary $\kappa\in C^0(\R)$ fulfilling $\int_0^\xi \kappa(\sig)d\sig \ge C\xi^{\omega+1}$ with some
  $\omega>1$ in \cite[Theorem 3]{blanchard_guibe}.\abs
  ii) \ We furthermore highlight that both Theorem \ref{theo37} and Theorem \ref{theo38} yield solutions featuring 
  a.e.~strictly positive temperature distributions without presupposing the corresponding initial data 
  to be uniformly positive in $\Om$, such as required in \cite{pawlow_zajaczkowski_SIMA}, 
  \cite{paw_zaj2017} and \cite{friedrich2024},
  for instance. 
  In fact, through the assumption that $\ell(\Theta_0)$ be integrable the hypotheses in (\ref{init}) allow for some considerably
  strong zeroes of $\Theta_0$, and in the case when $\int_0^1 \frac{\kappa(\sig)}{\sig} d\sig<\infty$ even admit
  widely arbitrary behavior of $\Theta_0$ in $\{\Theta_0<1\}$.\abs
{\bf Main results II: Large time behavior.} \quad
The solutions found above may fail to enjoy regularity properties significantly beyond those in (\ref{reg37}) and
(\ref{reg388}); in particular, in view of Definition \ref{dw} i) the statement of Theorem \ref{theo37} potentially involves
nontrivial defect measures in the temperature-related contribution to the mechanical part in (\ref{0}).
Remarkably, under a mild additional assumption on decay of the external sources 
$f$ and $g$ it is nevertheless possible to characterize the large time behavior of the solutions constructed before.
Unlike in most previous parts, our analysis in this direction now makes substantial use of the entropy production property
predicted by (\ref{ent}), through both the monotonicity feature and the dissipation mechanism expressed therein.\abs
Specifically, in Section \ref{sect_theta} we shall use an approximate counterpart of (\ref{ent}) as a starting point for a
qualitative analysis of the heat subsystem in (\ref{0}), in line with limited information on smoothness operating at markedly  
low levels of regularity. 
As a result, we shall see that the temperature approaches a homogeneous profile in the large time limit:
\begin{theo}\label{theo55}
  Let $n\ge 2$ and $\Om\subset\R^2$ be a bounded domain with smooth boundary,
  let $D>0$, and suppose that (\ref{DCB}), (\ref{DCsym}), (\ref{DC_lower}) and (\ref{kappa_reg}) hold with some $\kD>0$ and $\kC>0$,
  and that $f$ and $g$, besides fulfilling (\ref{fg_reg}), are such that 
  \be{fg_decay}
	\int_0^\infty \|f(\cdot,t)\|_{L^2(\Om)} dt < \infty
	\qquad \mbox{and} \qquad
	\int_0^\infty \|g(\cdot,t)\|_{L^1(\Om)} dt < \infty.
  \ee
  i) \ If $\kappa$ is such that (\ref{kappa_reg}) and (\ref{131.1}) hold, then there exists $\Tinf>0$ such that for the generalized
  solution $(u,\Theta)$ found in Theorem \ref{theo37} we have
  \be{55.1}
	\int_t^{t+1} \io |\Theta-\Tinf|^p \to 0
	\quad \mbox{for all } p\in (0,1)
	\qquad \mbox{as } t\to\infty.
  \ee
  ii) \ If $\kappa$ satisfies (\ref{kappa_reg}) and (\ref{38.1}), 
  then the integrable solution $(u,\Theta)$ of (\ref{0}) from Theorem \ref{theo38} has the property that with some $\Tinf>0$,
  \be{55.2}
	\int_t^{t+1} \io |\Theta-\Tinf| \to 0
	\qquad \mbox{as } t\to\infty.
  \ee
\end{theo}
In Section \ref{sect_ut} we shall see that a further consequence of (\ref{ent}), 
as well available without imposing assumptions beyond those from
Theorem \ref{theo37}, Theorem \ref{theo38} and (\ref{fg_decay}), asserts some relaxation also in the mechanical part:
\begin{prop}\label{prop56}
  Let $\Om\subset\R^n$ be a smoothly bounded domain in $\R^n$ with $n\ge 2$,
  and assume that $D>0$, and that with some $\kD>0$ and $\kC>0$,
  (\ref{DCB}), (\ref{DCsym}), (\ref{DC_lower}) and (\ref{fg_reg}) as well as
  (\ref{fg_decay}) and (\ref{init}) hold.
  Then whenever $\kappa$ is such that (\ref{kappa_reg}) and either (\ref{131.1}) or (\ref{38.1}) is valid,
  the solution of (\ref{0}) obtained in Theorem \ref{theo37} and Theorem \ref{theo38}, respectively, satisfies
  \be{56.1}
	\int_t^{t+1} \io |u_t| \to 0
	\qquad \mbox{as } t\to\infty.
  \ee
\end{prop}
Section \ref{sect_u}, finally, will rely on Theorem \ref{theo55} and Proposition \ref{prop56} in revealing
large time decay also of the displacement variable itself, provided that $\kappa$ satisfies the assumptions of
Theorem \ref{theo38}.
\begin{theo}\label{theo60}
  Let $n\ge 2$ and $\Om\subset\R^n$ be a bounded domain with smooth boundary,
  let $D>0$, suppose that the assumptions in (\ref{DCB}), (\ref{DCsym}) and (\ref{DC_lower}) are met with some $\kD>0$ and $\kC>0$,
  that $\kappa$ is such that (\ref{kappa_reg}) and (\ref{38.1}) is valid, and that
  $f$ and $g$ satisfy (\ref{fg_reg}) and (\ref{fg_decay}).
  Then whenever (\ref{init}) holds, the integrable solution of (\ref{0}) from Theorem \ref{theo38} has the property that
  \be{60.1}
	u(\cdot,t) \to 0
	\quad \mbox{in } L^2(\Om;\R^n)
	\qquad \mbox{as } t\to\infty.
  \ee
\end{theo}
{\bf Remark.} \quad
  Despite the poor regularity information available, Theorem \ref{theo55}, Proposition \ref{prop56} and Theorem \ref{theo60}
  strongly indicate that the behavior in (\ref{0}) significantly differs from that in the viscosity-free relative
  obtained on letting $\D=0$. 
  As recently observed in \cite{cieslak_fuest_lankeit}, namely, the latter thermoelastic system admits solutions 
  asymptotically retaining a mechanical part which persistently oscillates without damping.
\mysection{Generalized solutions and integrable solutions}
One key toward providing accessibility of any solution theory to the tenuous regularity information implied by (\ref{energy}) and
(\ref{corner}) consists in the design of an appropriately generalized solution concept.
The following proposes a notion of solvability that combines renormalization and inclusion of defect measures 
(\cite{villani}, \cite{diperna_lions}, \cite{feireisl}, \cite{zhigun}) 
with the idea to replace a parabolic equation by a corresponding inequality,
supplemented by a suitable one-sided control of an integrated quantity (see (\ref{wF}) and (\ref{wt}), as well as
\cite{win_SIMA2015} and \cite{claes_lankeit_win} for some precedents).
Here and below, we let $\M(\bom)$ denote the Banach space of signed Radon measures on $\bom$, and let $\M_+(\bom)$ represent the subset thereof consisting of all nonnegative Radon measures.
We also let $L^\infty_{w-\star}((0,\infty);\M_+(\bom))$ be the Banach space consisting of all weak-$\star$ measurable functions 
$\mu:(0,\infty)\to\M_+(\bom)$, equipped with the norm given by 
$\|\mu\|_{L^\infty_{w-\star}((0,\infty);\M(\bom))}:=\rm{ess} \sup_{t>0} \|\mu(t)\|_{\M(\bom)}$ for 
$\mu\in L^\infty_{w-\star}((0,\infty);\M(\bom))$, and consider the closed subset
$L^\infty_{w-\star}((0,\infty);\M_+(\bom))$ thereof which precisely contains all 
$\mu\in L^\infty_{w-\star}((0,\infty);\M(\bom))$ satisfying $\mu(t)\in \M_+(\bom)$ for a.e.~$t>0$.
\begin{defi}\label{dw}
  Let $\Om\subset\R^n$ be a bounded domain with smooth boundary, 
  let $D>0$, and suppose that
  $\D=(\D_{\mathbf{ijkl}})_{(\mathbf{i}, \mathbf{j}, \mathbf{k},\mathbf{l}) 
		\in \{1,...,n\}^4} \in \R^{n\times n \times n \times n}$,
  $\C=(\C_{\mathbf{ijkl}})_{(\mathbf{i}, \mathbf{j}, \mathbf{k},\mathbf{l}) 
		\in \{1,...,n\}^4} \in \R^{n\times n \times n \times n}$,
  $\B=(\B_{\mathbf{ij}})_{(\mathbf{i},\mathbf{j}) \in \{1,...,n\}^2} \in \R^{n\times n}$,
  $\kappa\in C^0([0,\infty))$,
  $f\in L^1_{loc}([0,\infty);L^2(\Om;\R^n))$ and
  $0\le g\in L^1_{loc}(\bom\times [0,\infty))$
  as well as
  $u_0\in W_0^{1,2}(\Om;\R^n), u_{0t} \in L^2(\Om;\R^n)$ and $\Theta_0\in L^1(\Om)$
  are such that $\Theta_0\ge 0$ a.e.~in $\Om$ and $K(\Theta_0)\in L^1(\Om)$, where 
  $K(\xi):=\int_0^\xi \kappa(\sig)d\sig$, $\xi\ge 0$.\abs
  i) \ By a {\em generalized solution} of (\ref{0}) we mean a pair of functions
  \be{w1}
	\lbal
	u\in L^\infty_{loc}([0,\infty);W_0^{1,2}(\Om;\R^n))
	\qquad \mbox{and} \\[1mm]
	\Theta \in L^1_{loc}(\bom\times [0,\infty))
	\ear
  \ee
  such that
  \be{w23}
	u_t \in L^\infty_{loc}([0,\infty);L^2(\Om;\R^n))
	\qquad \mbox{and} \qquad
	\nas u_t \in L^1_{loc}(\bom\times [0,\infty);\R^{n\times n}),
  \ee
  that $\Theta\ge 0$ a.e.~in $\Om\times (0,\infty)$ and
  \be{w44}
	K(\Theta) \in L^\infty_{loc}([0,\infty);L^1(\Om))
  \ee
  as well as
  \be{w5}
	\na\phi_0(\Theta) \in L^2_{loc}(\bom\times [0,\infty);\R^n)
	\quad \mbox{for all $\phi_0\in C^1([0,\infty))$ such that $\supp \phi_0$ is bounded,}
  \ee
  that with some 
  \be{w6}
	\mu \in L^\infty_{w-\star,loc}([0,\infty);\M_+(\bom))
  \ee
  and some $a>0$ we have
  \bea{wu}
	& & \hs{-20mm}
	\int_0^\infty \io u\cdot\vp_{tt} 
	+ \io u_0 \cdot \vp_t(\cdot,0)
	- \io u_{0t} \cdot \vp(\cdot,0) \nn\\
	&=& - \int_0^\infty \io \lan \D:\nas u_t, \na\vp\ran
	- \int_0^\infty \io \lan \C:\nas u,\na\vp\ran \nn\\
	& & + \int_0^\infty \io  \Theta \lan \B,\na\vp\ran
	+ \int_0^\infty \int_{\bom} \lan \B,\na\vp\ran d\mu
	+ \int_0^\infty \io f\cdot\vp
  \eea
  for all $\vp\in C_0^\infty(\Om\times [0,\infty);\R^n)$ as well as
  \be{wF}
	\F(t) 
	+ a \int_{\bom} d\mu(t)
	\le \F_0
	+ \int_0^t \io f\cdot u_t
	+ \int_0^t \io g
	\qquad \mbox{for a.e.~} t>0
  \ee
  with
  \be{F}
	\F(t):=\frac{1}{2}\io |u_t(\cdot,t)|^2
	+ \frac{1}{2} \io \lan \C:\nas u(\cdot,t), \nas u(\cdot,t)\ran
	+ \io K(\Theta(\cdot,t)),
	\qquad t>0,
  \ee
  and
  \be{F0}
	\F_0:= \frac{1}{2} \io |u_{0t}|^2
	+ \frac{1}{2} \io \lan \C:\nas u_0, \nas u_0\ran
	+ \io K(\Theta_0),
  \ee
  and that for each nondecreasing concave $\phi\in C^\infty([0,\infty))$ fulfilling $\phi'\in C_0^\infty([0,\infty))$, 
  the inequality
  \bea{wt}
	& & \hs{-20mm}
	- \int_0^\infty \io K^{(\phi)}(\Theta) \vp_t
	- \io K^{(\phi)}(\Theta_0) \vp(\cdot,0) \nn\\
	&\ge& - D \int_0^\infty \io \phi''(\Theta) |\na\Theta|^2 \vp
	- D \int_0^\infty \io \phi'(\Theta) \na\Theta\cdot\na\vp
	+ \int_0^\infty \io \phi'(\Theta) \lan \D:\nas u_t,\nas u_t\ran \vp \nn\\
	& & - \int_0^\infty \io \Theta \phi'(\Theta) \lan \B,\nas u_t\ran \vp
	+ \int_0^\infty \io \phi'(\Theta) g\vp
  \eea
  holds for each nonnegative $\vp\in C_0^\infty(\bom\times [0,\infty))$, where
  \be{Kphi}
	K^{(\phi)}(\xi):=\int_0^\xi \kappa(\sig) \phi'(\sig) d\sig.
  \ee
  ii) \ A generalized solution $(u,\Theta)$ of (\ref{0}) will be called an {\em integrable solution of (\ref{0})}
  if in the above we may choose $\mu\equiv 0$.
\end{defi}
{\bf Remark.} \quad
  i) \ Let $\phi\in C^\infty([0,\infty))$ be such that $\phi'\in C_0^\infty([0,\infty))$, $\phi'\ge 0$ and $\phi'' \le 0$.
  Then interpreting $-\phi''(\Theta) |\na\Theta|^2 = |\na\phi_0(\Theta)|^2$ with $\phi_0(\xi):=\int_0^\xi \sqrt{-\phi''(\sig)} d\sig$,
  $\xi\ge 0$, since $\phi_0\in C^1([0,\infty))$ and $\phi_0'\equiv 0$ in $[0,\infty)\sm \supp \phi''$ we see that the first summand
  on the right-hand side of (\ref{wt}) is well-defined for each $\Theta$ fulfilling (\ref{w1}) and (\ref{w5}).
  Together with a similar argument applied to the second integral on the right of (\ref{wt}), and with a straightforward
  combination of (\ref{w1}) with (\ref{w44}) and (\ref{w23}), this shows that (\ref{wt}) indeed is meaningful under the
  hypotheses of Definition \ref{dw}.\abs
  ii) \ As to be made more precise in Proposition \ref{prop_dw} in the appendix, the above solution concept
  is indeed consistent with that of classical solvability in the sense that any sufficiently regular generalized
  solution in fact satisfies (\ref{0}) in the pointwise sense.
\mysection{Energy control in regularized problems}
In slight contrast to most related precedents such as those in \cite{roubicek}, \cite{mielke_roubicek} and \cite{blanchard_guibe},
our construction of solutions will be based on a high-order parabolic approximation of the mechanical part in (\ref{0}).
The following elementary preparation regularizes the ingredients therein;
here and below, given $f$, $g$ and $\kappa$ satisfying (\ref{fg_reg}) and (\ref{kappa_reg}) we let
the functions $K$ and $\ell$ be as defined in (\ref{K}) and (\ref{ell}).
\begin{lem}\label{lem99}
  Assume (\ref{fg_reg}), (\ref{kappa_reg}) and (\ref{init}).
  Then there exist
  \be{keps}
	(\keps)_{\eps\in (0,1)} \subset C^\infty([0,\infty))
	\quad \mbox{with} \quad
	\keps(\xi) \ge \eps
	\mbox{ and } 
	|\keps(\xi)-\kappa(\xi)| \le \frac{\eps}{(\xi+1)^2}
	\mbox{ for all $\xi\ge 0$ and } \eps\in (0,1),
  \ee
  as well as 
  \be{fge}
	\lbal
	(\feps)_{\eps\in (0,1)} \subset C_0^\infty(\Om\times (0,\infty);\R^n)
	\qquad \mbox{and} \\[1mm]
	(\geps)_{\eps\in (0,1)} \subset C_0^\infty(\Om\times (0,\infty))
	\mbox{ with $\geps\ge 0$ for all } \eps\in (0,1),
	\ear
  \ee
  such that
  \be{fgec}
	\feps\to f \mbox{ in } L^1_{loc}([0,\infty);L^2(\Om;\R^n))
	\quad \mbox{and} \quad
	\geps\to g \mbox{ in } L^1_{loc}(\bom\times [0,\infty))
	\qquad \mbox{as } \eps\searrow 0,
  \ee
  and that if (\ref{fg_decay}) holds, then moreover
  \be{fge_decay}
	\sup_{\eps\in (0,1)} \bigg\{ 
	\int_0^\infty \|\feps(\cdot,t)\|_{L^2(\Om)} dt
	+ \int_0^\infty \|\geps(\cdot,t)\|_{L^1(\Om)} dt 
	\bigg\} < \infty.
  \ee
  Apart from that, one can find
  \be{ie}
	\lbal
	(v_{0\eps})_{\eps\in (0,1)} \subset C_0^\infty(\Om;\R^n), \\[1mm]
	(u_{0\eps})_{\eps\in (0,1)} \subset C_0^\infty(\Om;\R^n) 
	\qquad \mbox{and } \\[1mm]
	(\Theta_{0\eps})_{\eps\in (0,1)} \subset C^\infty(\bom)
	\quad \mbox{with $\na\Theta_{0\eps} \in C_0^\infty(\Om;\R^n)$ and $\Theta_{0\eps} \ge \eps$
	for all $\eps\in (0,1)$,} 
	\ear
  \ee
  such that as $\eps\searrow 0$ we have
  \be{iec}
	v_{0\eps} \to u_{0t} \mbox{ in } L^2(\Om), 
	\quad
	u_{0\eps} \to u_0 \mbox{ in } W^{1,2}(\Om)
	\quad \mbox{and} \quad
	\Theta_{0\eps} \to \Theta_0
	\mbox{ a.e.~in $\Om$}
  \ee
  as well as
  \be{iec2}
	K(\Theta_{0\eps}) \to K(\Theta_0) 
	\quad \mbox{and} \quad
	\ell(\Theta_{0\eps}) \to \ell(\Theta_0)
	\qquad \mbox{in $L^1(\Om)$,}
  \ee
  and that for each $\eps_0\in (0,1)$, the sets $(\keps)_{\eps>\eps_0}$, $(\feps)_{\eps>\eps_0}$, $(\geps)_{\eps>\eps_0}$,
  $(v_{0\eps})_{\eps>\eps_0}$, $(u_{0\eps})_{\eps>\eps_0}$ and $(\Theta_{0\eps})_{\eps>\eps_0}$ are all finite.
\end{lem}
\proof
  Except for those concerning the $\Theta_{0\eps}$, all statements can be verified on the basis of straightforward approximation
  procedures involving standard mollifiers.
  To suitably approximate the initial temperature distribution, we let
  \be{99.9}
	\Lam(\xi):=\int_1^\xi \kappa(\sig) \cdot \max \Big\{ 1, \frac{1}{\sig} \Big\} d\sig,
	\qquad \xi\ge 0,
  \ee
  and then see that $\Lam\in C^1((0,\infty))$ with $\Lam'>0$.
  Writing $z:=\Lam(\Theta_0)$, in view of the inclusions $\ell(\Theta_0)\in L^1(\Om)$ and $K(\Theta_0)\in L^1(\Om)$ as well as
  the inequality $\Theta_0\ge 0$ a.e.~in $\Om$, 
  as required by (\ref{init}), we thus infer that $z\in L^1(\Om)$ with $z\ge \Lam(0)$ a.e.~in $\Om$,
  whence by $\R\cup\{-\infty\}$-valued continuity of $\Lam$ at the origin, again using mollification we find 
  $(\zeps)_{\eps\in (0,1)} \subset C^\infty(\bom)$ such  
  that $(\na \zeps)_{\eps\in (0,1)} \subset C_0^\infty(\Om;\R^n)$,
  that $\zeps\ge \Lam(2\eps)$ in $\Om$ for all $\eps\in (0,1)$, and that $\zeps\to z$ in $L^1(\Om)$ as $\eps\searrow 0$.
  Translated to $\wh{\Theta}_{0\eps}:=\Lam^{-1}(\zeps)$, $\eps\in (0,1)$, this shows that 
  $(\wh{\Theta}_{0\eps})_{\eps\in (0,1)} \subset C^1(\bom)$ is such that $\supp \na \wh{\Theta}_{0\eps} \subset\Om$
  and 
  \be{99.91}
	\wh{\Theta}_{0\eps} \ge 2\eps
	\qquad \mbox{in } \Om
  \ee
  for all $\eps\in (0,1)$, and that as $\eps\searrow 0$ we have
  $\Lam(\wh{\Theta}_{0\eps}) \to \Lam(\Theta_0)$ in $L^1(\Om)$ and $\wh{\Theta}_{0\eps} \to \Theta_0$ a.e.~in $\Om$.
  Since $\Lam(\xi)=\ell(\xi)$ for all $\xi\in (0,1)$ and $K(\xi)-\Lam(\xi)=c_1:=\int_0^1 \kappa(\sig)d\sig$ for all $\xi\ge 1$,
  and since thus the inequalities $K\ge \ell \ge 0$ on $[1,\infty)$ and $0\le K\le c_1$ on $(0,1)$ imply that
  \bas
	|\ell(\wh{\Theta}_{0\eps})|
	\le |\Lam(\wh{\Theta}_{0\eps})| + c_1
	\quad \mbox{and} \quad
	|K(\wh{\Theta}_{0\eps})| \le |\Lam(\wh{\Theta}_{0\eps})| + c_1
	\quad \mbox{in } \Om
	\qquad \mbox{for all } \eps\in (0,1),
  \eas
  using the dominated convergence theorem we find that 
  \be{99.92}
	\ell(\wh{\Theta}_{0\eps}) \to \ell(\Theta_0)
	\quad \mbox{and} \quad
	K(\wh{\Theta}_{0\eps}) \to K(\Theta_0)
	\quad \mbox{in } L^1(\Om)
	\qquad \mbox{as } \eps\searrow 0.
  \ee
  For fixed $\eps\in (0,1)$, we can next rely on the inclusion $\wh{\Theta}_{0\eps}\in C^1(\bom)$ in choosing
  $\Theta_{0\eps} \in C^\infty(\bom)$ such that $\na \Theta_{0\eps} \in C_0^\infty(\Om;\R^n)$ as well as
  \be{99.93}
	|\Theta_{0\eps} - \wh{\Theta}_{0\eps}| \le \eps,
	\quad
	|\ell(\Theta_{0\eps}) - \ell(\wh{\Theta}_{0\eps})| \le \eps
	\quad \mbox{and} \quad
	|K(\Theta_{0\eps}) - K(\wh{\Theta}_{0\eps})| \le \eps
  \ee
  in $\Om$. Then (\ref{99.91}) and (\ref{99.93}) assert that $\Theta_{0\eps} \ge \eps$ in $\Om$, while (\ref{99.92}) 
  together with (\ref{99.93}) imply (\ref{iec2}). 
  Finiteness of $(\Theta_{0\eps})_{\eps>\eps_0}$ for each $\eps_0\in (0,1)$, finally, can trivially be achieved by 
  straightforward modification.
\qed
These choices particularly enable us to conveniently approximate also the function in (\ref{K}):
\begin{lem}\label{lem_Keps}
  Assume (\ref{kappa_reg}), and for $\eps\in (0,1)$ let
  \be{Keps}
	\Keps(\xi):=\int_0^\xi \keps(\sig) d\sig,
	\qquad \xi\ge 0.
  \ee
  Then
  \be{K1}
	|\Keps(\xi)-K(\xi)| \le \eps
	\qquad \mbox{for all } \xi\ge 0.
  \ee
\end{lem}
\proof
  Since $\int_0^\infty \frac{d\sig}{(\sig+1)^2}=1$, from (\ref{keps}) it follows that, indeed,
  \bas
	|\Keps(\xi)-K(\xi)| 
	= \bigg| \int_0^\xi (\keps(\sig)-\kappa(\sig))d\sig\bigg|
	\le \int_0^\xi \frac{\eps}{(\sig+1)^2} d\sig \le \eps
  \eas
  for all $\xi\ge 0$.
\qed
Fixing any integer $m\ge 1$ such that
\be{m}
	m\ge \frac{n+4}{2},
\ee
for $\eps\in (0,1)$ we now consider the parabolic-ODE-parabolic regularization of (\ref{0}) given by
\be{0eps}
	\hs{-6mm}
	\lball
	v_{\eps t} + \eps \Del^{2m} \veps = \div (\D:\nas\veps) + \div (\C:\nas\ueps) - \div (\Teps\B) + \feps(x,t),
	\quad & x\in\Om, \ t>0, \\[1mm]
	u_{\eps t} = \veps, 
	\quad & x\in\Om, \ t>0, \\[1mm]
	\keps(\Teps) \Theta_{\eps t} = D\Del\Teps + \lan \D:\nas\veps,\nas\veps\ran - \Teps \lan \B,\nas\veps\ran + \geps(x,t),
	\quad & x\in\Om, \ t>0, \\[1mm]
	\Del^k \veps = \frac{\pa \Del^k \veps}{\pa\nu}=0, \ k\in \{0,...,m-1\}, 
		\quad \ueps=0, \quad \frac{\pa\Teps}{\pa\nu}=0,
	\quad & x\in\pO, \ t>0, \\[1mm]
	\veps(x,0)=v_{0\eps}(x), \quad \ueps(x,0)=u_{0\eps}(x), \quad \Teps(x,0)=\Theta_{0\eps}(x),
	\quad & x\in\Om.
	\ear
\ee
which can be seen to admit local classical solutions as well as a handy extensibility criterion. 
A technical but essentially straightforward derivation of the following statement in this regard will be postponed to
an appendix below.
\begin{lem}\label{lem1}
  Let $m\in\N$ satisfy $m>\frac{n+2}{4}$.
  For each $\eps\in (0,1)$ there exist $\tme\in (0,\infty]$ as well as
  \be{1.1}
	\lbal
	\veps\in C^0([0,\tme);W_0^{2m,2}(\Om;\R^n)) \cap C^\infty(\bom\times (0,\tme);\R^n), \\[1mm]
	\ueps\in C^1([0,\tme);W_0^{2m,2}(\Om;\R^n)) \cap C^\infty(\bom\times (0,\tme);\R^n)
	\qquad \mbox{and} \\[1mm]
	\Teps \in C^{2,1}(\bom\times [0,\tme)) \cap C^\infty(\bom\times (0,\tme))
	\ear
  \ee
  such that $\Teps>0$ in $\bom\times [0,\tme)$, that (\ref{0eps}) is solved in the classical sense in $\Om\times (0,\tme)$, and that
  \bea{ext}
	& & \hs{-20mm}
	\mbox{if $\tme<\infty$, \quad then \ as $t\nearrow\tme$ we have} \nn\\
	& & \|\veps(\cdot,t)\|_{W^{2m,2}(\Om)} + \|\ueps(\cdot,t)\|_{W^{2m,2}(\Om)} 
	+ \|\Teps(\cdot,t)\|_{W^{1,2}(\Om)} 
	+ \|\Teps(\cdot,t)\|_{L^\infty(\Om)} 
	\to \infty.
  \eea
\end{lem}
Thanks to the fact that the second equation in (\ref{0eps}) does not involve a parabolic regularization such as the first one,
the regularization made in (\ref{0eps}) is fully consistent with the property of essential energy nonincrease suggested
by (\ref{energy}):
\begin{lem}\label{lem3}
  Let $\eps\in (0,1)$.  Then
  \bea{3.1}
	& & \hs{-30mm}
	\frac{d}{dt} \bigg\{ \frac{1}{2} \io |\veps|^2 + \frac{1}{2} \io \lan \C:\nas\ueps,\nas\ueps\ran + \io \Keps(\Teps) \bigg\}
	+ \eps \io |\Del^m \veps|^2 \nn\\
	&\le& \io \feps\cdot\veps 
	+ \io \geps
	\qquad \mbox{for all } t\in (0,\tme).
  \eea
\end{lem}
\proof
  Since $\D_{\mathbf{ijkl}}=\D_{\mathbf{jikl}}$ and $\C_{\mathbf{ijkl}}=\C_{\mathbf{jikl}}$ 
  for $(\mathbf{i},\mathbf{j},\mathbf{k},\mathbf{l})\in\{1,...,n\}^4$, and since $\B_{\mathbf{ij}}=\B_{\mathbf{ji}}$
  for $(\mathbf{i},\mathbf{j})\in\{1,...,n\}^2$, 
  it readily follows that for any $\vp\in C^1(\Om;\R^n)$ and $\psi\in C^1(\Om;\R^n)$ we have
  \bas
	\lan \D:\nas\vp,\na\psi\ran=\lan \D:\nas\vp,\nas\psi\ran
	\qquad \mbox{and} \qquad
	\lan \C:\nas\vp,\na\psi\ran=\lan \C:\nas\vp,\nas\psi\ran
  \eas
  as well as $\lan \B,\na\vp\ran=\lan\B,\nas \vp\ran$ in $\Om$. 
  On testing the first equation in (\ref{0eps}) by $\veps$, we thus obtain that
  for all $t\in (0,\tme)$,
  \bea{3.2}
	& & \hs{-20mm}
	\frac{1}{2} \frac{d}{dt} \io |\veps|^2
	+ \eps \io |\Del^m \veps|^2 \nn\\
	&=& - \io \lan\D:\nas\veps,\na\veps\ran
	- \io \lan\C:\nas\ueps,\nas\veps\ran 
	+ \io \Teps \lan \B,\na\veps\ran
	+ \io \feps\cdot\veps \nn\\
	&=& - \io \lan\D:\nas\veps,\nas\veps\ran
	- \io \lan\C:\nas\ueps,\nas\veps\ran
	+ \io \Teps \lan \B,\nas\veps\ran
	+ \io \feps\cdot\veps,
  \eea
  where in view of the identity $\veps=u_{\eps t}$,
  \be{3.3}
	\hs{-4mm}
	- \io \lan\C:\nas\ueps,\nas\veps\ran
	= - \frac{1}{2} \frac{d}{dt} \io \lan\C:\nas\ueps,\nas\ueps\ran
	\qquad \mbox{for all } t\in (0,\tme),
  \ee
  because the symmetry assumption $\C_{\mathbf{ijkl}}=\C_{\mathbf{klij}}$ 
  for $(\mathbf{i},\mathbf{j},\mathbf{k},\mathbf{l})\in\{1,...,n\}^4$ warrants that whenever $T>0$ and
  $\vp\in C^1(\Om\times (0,T);\R^{n\times n}_{sym})$,
  \bas
	\pa_t \lan\C:\vp,\vp\ran
	= 2\lan\C:\vp,\vp_t\ran
	\qquad \mbox{in } \Om\times (0,T).
  \eas
  Apart from that, in view of (\ref{Keps}) the third equation in (\ref{0eps}) implies that
  \bas
	\frac{d}{dt} \io \Keps(\Teps) = \io \lan\D:\nas\veps,\nas\veps\ran
	- \io \Teps \lan\B,\nas\veps\ran
	+ \io \geps
	\qquad \mbox{for all } t\in (0,\tme),
  \eas
  which when added to (\ref{3.2}) yields (\ref{3.1}).
\qed
Appropriately exploiting this will be facilitated by the following elementary tool.
\begin{lem}\label{lem5}
  If $K>0, T\in (0,\infty]$ and $\beta\in (0,1]$, and if $y\in C^0([0,T)) \cap C^1((0,T))$ and $h\in C^0((0,T))$ are nonnegative
  and such that
  \be{5.1}
	y(0) \le K
	\qquad \mbox{and} \qquad
	\int_0^T h(t) dt \le K
  \ee
  as well as
  \be{5.2}
	y'(t) \le h(t) y^\beta(t) + h(t)
	\qquad \mbox{for all } t\in (0,T),
  \ee
  then
  \be{5.3}
	y(t) \le 3K e^K
	\qquad \mbox{for all } t\in (0,T).
  \ee
\end{lem}
\proof
  If $T, y$ and $h$ have the assumed properties, then by Young's inequality, (\ref{5.2}) implies that
  $y'(t) \le h(t) y(t) + 2h(t)$
  for all $t\in (0,T)$, so that by a Gr\"onwall lemma and (\ref{5.1}), indeed,
  \bas
	y(t)
	\le y(0) e^{\int_0^t h(s) ds}
	+ \int_0^t e^{\int_s^t h(\sig) d\sig} \cdot 2h(s) ds
	\le K e^K + e^K \cdot 2K
  \eas
  for all $t\in (0,T)$.
\qed
In fact, Lemma \ref{lem3} can thereby be seen to entail some basic estimates which in their essential parts coincide with
what has been predicted by (\ref{energy}) already.
\begin{lem}\label{lem6}
  There exists $(C(T))_{T>0} \subset (0,\infty)$ such that whenever $T>0$,
  \be{6.1}
	\io |\veps(\cdot,t)|^2 \le C(T)
	\qquad \mbox{for all $t\in (0,T)\cap (0,\tme)$ and } \eps\in (0,1)
  \ee
  and
  \be{6.2}
	\io |\na\ueps(\cdot,t)|^2 \le C(T)
	\qquad \mbox{for all $t\in (0,T)\cap (0,\tme)$ and } \eps\in (0,1)
  \ee
  and
  \be{6.3}
	\io \Keps(\Teps(\cdot,t)) \le C(T)
	\qquad \mbox{for all $t\in (0,T)\cap (0,\tme)$ and } \eps\in (0,1)
  \ee
  as well as
  \be{6.4}
	\eps \int_0^t \|\veps(\cdot,s)\|_{W^{2m,2}(\Om)}^2 \le C(T)
	\qquad \mbox{for all $t\in (0,T)\cap (0,\tme)$ and } \eps\in (0,1),
  \ee
  and that, moreover,
  \be{6.5}
	\sup_{T>0} C(T)<\infty
	\quad \mbox{if (\ref{fg_decay}) holds.}
  \ee
\end{lem}
\proof
  According to (\ref{DC_lower}) and Korn's inequality, there exists $c_1>0$ such that
  \be{6.6}
	\frac{1}{2} \io \lan\C:\nas\vp,\nas\vp\ran 
	\ge c_1 \io |\na\vp|^2
	\qquad \mbox{for all } \vp\in W_0^{1,2}(\Om;\R^n),
  \ee
  which in particular asserts nonnegativity of
  \be{6.7}
	\yeps(t):=\frac{1}{2} \io |\veps(\cdot,t)|^2
	+ \frac{1}{2} \io \lan\C:\nas\ueps(\cdot,t),\nas\ueps(\cdot,t)\ran
	+ \io \Keps(\Teps(\cdot,t)),
	\qquad t\in [0,\tme), \ \eps\in (0,1).
  \ee
  Fixing $\eps\in (0,1)$ now, from Lemma \ref{lem3} and the Cauchy-Schwarz inequality we obtain that
  \bea{6.8}
	\yeps'(t) + \eps\io |\Del^m \veps|^2
	&\le& \io \feps\cdot\veps + \io \geps \nn\\
	&\le& \|\feps\|_{L^2(\Om)} \|\veps\|_{L^2(\Om)} + \|\geps\|_{L^1(\Om)} \nn\\
	&\le& \sqrt{2} \|\feps\|_{L^2(\Om)} \sqrt{\yeps(t)} + \|\geps\|_{L^1(\Om)}
	\qquad \mbox{for all $t\in (0,\tme)$,}
  \eea
  whence an application of Lemma \ref{lem5} shows that for each $T>0$,
  \be{6.9}
	\yeps(t) \le c_2(T) := 3c_3(T) e^{c_3(T)}
	\qquad \mbox{for all $t\in (0,T) \cap (0,\tme)$,}
  \ee
  where
  \be{6.10}
	c_3(T):=\max \bigg\{ c_4 \, , \, \sqrt{2} \sup_{\eps'\in (0,1)} \int_0^T \|f_{\eps'}(\cdot,t)\|_{L^2(\Om)} \, , \, 
		\sup_{\eps'\in (0,1)} \int_0^T \|g_{\eps'}(\cdot,t)\|_{L^1(\Om)} dt \bigg\}
  \ee
  with $c_4:=\sup_{\eps'\in (0,1)} y_{\eps'}(0)$ being finite according to (\ref{iec}) and (\ref{iec2}).
  By integration in (\ref{6.8}), we thereupon see that
  for all $t\in (0,T)\cap (0,\tme)$,
  \bas
	\eps \int_0^t \io |\Del^m \veps|^2
	\le c_4 + \sqrt{2c_2(T)} \int_0^T \|\feps(\cdot,s)\|_{L^2(\Om)} ds
	+ \int_0^T \|\geps(\cdot,s)\|_{L^1(\Om)} ds.
  \eas
  Thanks to the fact that $\|\Del^m (\cdot)\|_{L^2(\Om)}$ defines a norm equaivalent to $\|\cdot\|_{W^{2m,2}(\Om)}$
  in $W_0^{2m,2}(\Om;\R^n)$ due to elliptic regularity theory (\cite{friedman}),
  together with (\ref{6.9}) and (\ref{6.6}) this yields (\ref{6.1})-(\ref{6.4}) with some $(C(T))_{T>0}$ that satisfies
  (\ref{6.5}) due to (\ref{6.10}) and (\ref{fge_decay}).
\qed
\subsection{Global existence in the approximate problems}
Our first two applications of Lemma \ref{lem6} will operate at levels of fixed $\eps\in (0,1)$ in order to make sure
that each of the solutions obtained in Lemma \ref{lem1} in fact exists globally. 
A combination of a comparison argument with a basic variational reasoning yields the first step into this direction:
\begin{lem}\label{lem7}
  If $\tme<\infty$ for some $\eps\in (0,1)$, then there exists $C(\eps)>0$ such that
  \be{7.1}
	\|\Teps(\cdot,t)\|_{L^\infty(\Om)} \le C(\eps)
	\qquad \mbox{for all } t\in (0,\tme)
  \ee
  and
  \be{7.2}
	\int_0^t \io |\na\Teps|^2 \le C(\eps)
	\qquad \mbox{for all } t\in (0,\tme).
  \ee
\end{lem}
\proof
  From Lemma \ref{lem6} and the presupposed finiteness of $\tme$ we obtain $c_1(\eps)>0$ such that
  $\int_0^t \big\{ \|\veps(\cdot,s)\|_{W^{2m,2}(\Om)} + \|\veps(\cdot,s)\|_{W^{2m,2}(\Om)}^2 \big\} ds \le c_1(\eps)$  
  for all $t\in (0,\tme)$. Again relying on the fact that our assumption $m>\frac{n+2}{4}$ ensures continuity of the embedding
  $W^{2m,2}(\Om)\hra W^{1,\infty}(\Om)$, in view of (\ref{fge}) we thus infer the existence of $c_2(\eps)>0$ and
  $c_3(\eps)>0$ such that $h_{1\eps}:=-\lan\B,\nas\veps\ran$ and $h_{2\eps}:=\lan\D:\nas\veps,\nas\veps\ran + \geps$ satisfy
  \be{7.3}
	\int_0^t \|h_{1\eps}(\cdot,s)\|_{L^\infty(\Om)} ds \le c_2(\eps)
	\quad \mbox{and} \quad
	\int_0^t \|h_{2\eps}(\cdot,s)\|_{L^\infty(\Om)} ds \le c_3(\eps)
	\qquad \mbox{for all } t\in (0,\tme).
  \ee
  We now let $c_4(\eps):=\|\Theta_{0\eps}\|_{L^\infty(\Om)}$ and
  \bas
	\ov{\Theta}(x,t)
	&:=&
	c_4(\eps) \exp \bigg\{ \frac{1}{\eps} \int_0^t \|h_{1\eps}(\cdot,s)\|_{L^\infty(\Om)} ds \bigg\} \nn\\
	& & + \int_0^t \exp \bigg\{ \frac{1}{\eps} \int_s^t \|h_{1\eps}(\cdot,\sig)\|_{L^\infty(\Om)} d\sig \bigg\} 
		\cdot \|h_{2\eps}(\cdot,s)\|_{L^\infty(\Om)} ds,
	\qquad t\in [0,\tme),
  \eas
  and then see that since $\keps\ge\eps$ by (\ref{keps}),
  \bas
	\keps(\Teps) \ov{\Theta}_t
	- D\Del \ov{\Theta} - h_{1\eps} \ov{\Theta} - h_{2\eps}
	&=& \keps(\Teps) \cdot \Big\{ \frac{1}{\eps} \|h_{1\eps}\|_{L^\infty(\Om)} \ov{\Theta} 
		+ \frac{1}{\eps} \|h_{2\eps}\|_{L^\infty(\Om)} \Big\} 
	- h_{1\eps} \ov{\Theta} - h_{2\eps} \\
	&\ge& \|h_{1\eps}\|_{L^\infty(\Om)} \ov{\Theta} 
		+ \|h_{2\eps}\|_{L^\infty(\Om)} 
	- h_{1\eps} \ov{\Theta} - h_{2\eps} \\[2mm]
	&\ge& 0
	\qquad \mbox{in } \Om\times (0,\tme).
  \eas
  As clearly $\ov{\Theta}(\cdot,0)=c_4(\eps)\ge\Theta_{0\eps}$ in $\Om$ and $\frac{\pa\ov{\Theta}}{\pa\nu}=0$ on $\pO\times (0,\tme)$,
  on the basis of the identity
  \be{7.33}
	\keps(\Teps) \Theta_{\eps t} = D\Del\Teps + h_{1\eps}(x,t)\Teps + h_{2\eps}(x,t),
	\qquad x\in\Om, \ t\in (0,\tme),
  \ee
  a comparison principle guarantees that $\Teps\le\ov{\Theta}$ in $\Om\times (0,\tme)$ and hence, by (\ref{7.3})
  and the nonnegativity of $\Teps$,
  \be{7.4}
	\|\Teps(\cdot,t)\|_{L^\infty(\Om)}
	\le c_5(\eps) := c_4(\eps) e^\frac{c_2(\eps)}{\eps} + e^\frac{c_2(\eps)}{\eps} c_3(\eps)
	\qquad \mbox{for all } t\in (0,\tme).
  \ee
  In particular, this implies that if we multiply (\ref{7.33}) by $\Teps$ and integrate by parts, then writing 
  $\wh{K}_\eps(\xi):=\int_0^\xi \sig\keps(\sig)d\sig$, $\xi\ge 0$, for all $t\in (0,\tme)$ we have
  \bas
	\frac{d}{dt} \io \wh{K}_\eps(\Teps)
	+ D \io |\na\Teps|^2
	= \io h_{1\eps} \Teps^2
	+ \io h_{2\eps} \Teps
	\le c_5^2(\eps) |\Om| \cdot \|h_{1\eps}\|_{L^\infty(\Om)}
	+ c_5(\eps) |\Om| \cdot \|h_{2\eps}\|_{L^\infty(\Om)}
  \eas
  and thus
  \bas
	& & \hs{-20mm}
	\io \wh{K}_\eps(\Teps(\cdot,t))
	+ D \int_0^t \io |\na\Teps|^2 \\
	&\le& \io \wh{K}_\eps(\Theta_{0\eps})
	+ c_5^2(\eps) |\Om| \int_0^t \|h_{1\eps}(\cdot,s)\|_{L^\infty(\Om)} ds
	+ c_5(\eps) |\Om| \int_0^t \|h_{2\eps}(\cdot,s)\|_{L^\infty(\Om)} ds \\
	&\le& \io \wh{K}_\eps(\Theta_{0\eps})
	+ c_5^2(\eps) |\Om| \cdot c_2(\eps)
	+ c_5(\eps) |\Om| \cdot c_3(\eps)
	\qquad \mbox{for all } t\in (0,\tme),
  \eas
  once more because of (\ref{7.3}). By nonnegativity and continuity of $\wh{K}_\eps$, 
  this establishes (\ref{7.2}), while (\ref{7.1}) has been achieved in (\ref{7.4}).
\qed
Another standard testing procedure turns the gradient estimate in (\ref{7.2}) into some possibly $\eps$-dependent but time-independent
control of the displacement variable:
\begin{lem}\label{lem8}
  Suppose that $\eps\in (0,1)$ is such that $\tme<\infty$. Then there exists $C(\eps)>0$ such that
  \be{8.1}
	\|\veps(\cdot,t)\|_{W^{2m,2}(\Om)} \le C(\eps)
	\qquad \mbox{for all } t\in (0,\tme)
  \ee
  and
  \be{8.2}
	\|\ueps(\cdot,t)\|_{W^{2m,2}(\Om)} \le C(\eps)
	\qquad \mbox{for all } t\in (0,\tme).
  \ee
\end{lem}
\proof
  As $m\ge 1$, finiteness of $\tme$ ensures that due to Lemma \ref{lem6} we have
  \be{8.3}
	\int_0^t \|\veps(\cdot,s)\|_{W^{2,2}(\Om)} ds
	\le c_1(\eps)
	\qquad \mbox{for all } t\in (0,\tme)
  \ee
  with some $c_1(\eps)>0$, so that since
  \bea{8.4}
	\hs{-6mm}
	\|\ueps(\cdot,t)\|_{W^{2m,2}(\Om)}
	&=& \bigg\| u_{0\eps} + \int_0^t \veps(\cdot,s) ds \bigg\|_{W^{2m,2}(\Om)} \nn\\
	&\le& \|u_{0\eps}\|_{W^{2m,2}(\Om)}
	+ \tme \int_0^t \|\veps(\cdot,s)\|_{W^{2m,2}(\Om)}
	\qquad \mbox{for all } t\in (0,\tme),
  \eea
  by using Lemma \ref{lem7} and (\ref{fge}) we find $c_2(\eps)>0$ such that for
  $\heps:=\div (\D:\nas\veps) + \div (\C:\nas\ueps) - \div (\Teps\B) + \feps$ we have
  \be{8.5}
	\int_0^t \io |\heps|^2 \le c_2(\eps)
	\qquad \mbox{for all } t\in (0,\tme).
  \ee
  Now testing the identity $v_{\eps t} + \Del^{2m} \veps = \heps$ by $\Del^{2m} \veps$ and employing Young's inequality yields
  \bas
	\frac{1}{2} \frac{d}{dt} \io |\Del^m \veps|^2
	+ \eps\io |\Del^{2m}\veps|^2
	= \io \heps\cdot \Del^{2m} \veps
	\le \eps \io |\Del^{2m}\veps|^2
	+ \frac{1}{4\eps} \io |\heps|^2
	\quad \mbox{for all } t\in (0,\tme),
  \eas
  so that (\ref{8.5}) warrants that
  \bas
	\io |\Del^m \veps|^2
	\le \io |\Del^m v_{0\eps}|^2
	+ \frac{1}{2\eps} \int_0^t \io |\heps|^2
	\le \io |\Del^m v_{0\eps}|^2
	+ \frac{c_2(\eps)}{2\eps}
	\qquad \mbox{for all } t\in (0,\tme).
  \eas
  Again thanks to elliptic regularity theory, this confirms (\ref{8.1}) and thereby, through (\ref{8.4}), also entails (\ref{8.2}).
\qed
This improves our knowledge on regularity of the temperature gradient:
\begin{lem}\label{lem87}
  If $\eps\in (0,1)$ is such that $\tme<\infty$, then there exists $C(\eps)>0$ such that
  \be{87.1}
	\io |\na\Teps(\cdot,t)|^2 \le C(\eps)
	\qquad \mbox{for all } t\in (0,\tme).
  \ee
\end{lem}
\proof
  Again since $m\ge 1$, from Lemma \ref{lem8} and Lemma \ref{lem7} we obtain $c_1(\eps)>0$ such that 
  $h_\eps:=\lan \D:\nas\veps,\nas\veps\ran - \Teps \lan\B,\nas\veps\ran+\geps$ satisfies $\io \heps^2 \le c_1(\eps)$
  for all $t\in (0,\tme)$.
  Using that $\keps\ge \eps$, from the third equation in (\ref{0eps}) and Young's inequality we thus infer that
  \bas
	\eps \int_0^t \io \Tepst^2
	&\le& \int_0^t \io \keps(\Teps) \Tepst^2 
	= \int_0^t \io \Tepst \cdot \big\{ D\Del\Teps + \heps\big\} \\
	&=& - \frac{1}{2} \io |\na\Teps(\cdot,t)|^2 + \frac{1}{2} \io |\na\Theta_{0\eps}|^2
	+ \int_0^t \io \Tepst \heps \\
	&\le& - \frac{1}{2} \io |\na\Teps(\cdot,t)|^2 + \frac{1}{2} \io |\na\Theta_{0\eps}|^2
	+ \eps \int_0^t \io \Tepst^2
	+ \frac{c_1 \tme}{4\eps}
	\qquad \mbox{for all } t\in (0,\tme),
  \eas
  so that (\ref{87.1}) follows if we let $C(\eps):=\io |\na \Theta_{0\eps}|^2 + \frac{c_1 \tme}{2\eps}$.
\qed
In conclusion:
\begin{lem}\label{lem9}
  For each $\eps\in (0,1)$, the solution of (\ref{0eps}) is global in time; that is, in Lemma \ref{lem1} we have $\tme=\infty$.
\end{lem}
\proof
  In view of Lemma \ref{lem8}, Lemma \ref{lem7} and Lemma \ref{lem87}, 
  this is an evident consequence of the extensibility criterion (\ref{ext}).
\qed
\mysection{The test function $\frac{\ln^2(\Teps+M)}{\Teps+M}$ in the thermal part}
Now the core of our existence theory consists in the following rigorous counterpart of (\ref{corner}).
\begin{lem}\label{lem12}
  There exists $(C(T))_{T>0}\subset (0,\infty)$ such that
  \be{12.1}
	\int_t^{t+1} \io \frac{\ln^2(\Teps+e)|\na\Teps|^2}{(\Teps+1)^2} \le C(T)
	\qquad \mbox{for all $t>0$ and } \eps\in (0,1),
  \ee
  that
  \be{12.2}
	\int_t^{t+1} \io \frac{\ln^2(\Teps+e)|\nas\veps|^2}{\Teps+1} \le C(T)
	\qquad \mbox{for all $t>0$ and } \eps\in (0,1),
  \ee
  and that
  \be{12.3}
	\sup_{T>0} C(T)<\infty
	\quad \mbox{if (\ref{fg_decay}) holds.}
  \ee
\end{lem}
\proof
  We fix any $M\ge e^4$, and for $\eps\in (0,1)$ we let
  \bas
	\helleps(\xi):=\int_0^\xi \frac{\ln^2(\sig+M)\keps(\sig)}{\sig+M} d\sig,
	\qquad \xi\ge 0,
  \eas
  noting that
  \be{12.33}
	\helleps'(\xi) = \frac{\ln^2(\xi+M)\keps(\xi)}{\xi+M}
	\qquad \mbox{for all } \xi\ge 0,
  \ee
  and that since 
  \be{12.4}
	\ln \sig \le \frac{2}{e} \sqrt{\sig}
	\qquad \mbox{for all } \sig\ge 1,
  \ee
  we have
  \be{12.5}
	0 \le \helleps(\xi) \le \frac{4}{e^2} \int_0^\xi \keps(\sig) d\sig
	= \frac{4}{e^2} \Keps(\xi)
	\qquad \mbox{for all } \xi\ge 0
  \ee
  according to (\ref{Keps}).\abs
  Now testing the third equation in (\ref{0eps}) against $\frac{\ln^2(\Teps+M)}{\Teps+M}$, in line with (\ref{12.33}) we obtain the
  identity
  \bea{12.6}
	\hs{-8mm}
	\frac{d}{dt} \io \helleps(\Teps)
	&=& \io \frac{\ln^2(\Teps+M)}{\Teps+M} \cdot \keps(\Teps) \Theta_{\eps t} \nn\\
	&=& \io \frac{\ln^2(\Teps+M)}{\Teps+M} \cdot \Big\{ D\Del\Teps + \lan\D:\nas\veps,\nas\veps\ran 
		- \Teps\lan\B,\nas\veps\ran + \geps \Big\} \nn\\
	&=& D \io \bigg\{ \frac{\ln^2(\Teps+M)}{(\Teps+M)^2} - \frac{2\ln(\Teps+M)}{(\Teps+M)^2} \bigg\} \cdot |\na\Teps|^2 \nn\\
	& & + \io \frac{\ln^2(\Teps+M)}{\Teps+M} \cdot \lan \D:\nas\veps,\nas\veps\ran \nn\\
	& & - \io \frac{\Teps\ln^2(\Teps+M)}{\Teps+M} \cdot \lan\B,\nas\veps\ran
	+ \io \frac{\ln^2(\Teps+M)}{\Teps+M} \cdot\geps
	\qquad \mbox{for all } t>0,
  \eea
  where (\ref{DC_lower}) and the nonnegativity of $\geps$ ensure that
  \be{12.7}
	\io \frac{\ln^2(\Teps+M)}{\Teps+M} \cdot \lan \D:\nas\veps,\nas\veps\ran
	\ge \kD \io \frac{\ln^2(\Teps+M) |\nas\veps|^2}{\Teps+M}
	\qquad \mbox{for all } t>0
  \ee
  and
  \be{12.8}
	\io \frac{\ln^2(\Teps+M)}{\Teps+M} \cdot\geps \ge 0
	\qquad \mbox{for all } t>0.
  \ee
  Moreover, our restriction on $M$ warrants that $\ln M\ge 4$ and hence  
  \bas
	\ln^2(\xi+M)-2\ln(\xi+M)
	&=& \ln^2(\xi+M) \cdot \Big\{ 1 - \frac{2}{\ln(\xi+M)} \Big\} \\
	&\ge& \ln^2(\xi+M) \cdot \Big\{ 1 - \frac{2}{\ln M} \Big\} \\
	&\ge& \frac{1}{2} \ln^2(\xi+M)
	\qquad \mbox{for all } \xi\ge 0,
  \eas
  so that
  \be{12.9}
	D \io \bigg\{ \frac{\ln^2(\Teps+M)}{(\Teps+M)^2} - \frac{2\ln(\Teps+M)}{(\Teps+M)^2} \bigg\} \cdot |\na\Teps|^2 
	\ge \frac{D}{2} \io \frac{\ln^2(\Teps+M) |\na\Teps|^2}{(\Teps+M)^2}
	\qquad \mbox{for all } t>0.
  \ee
  In the crucial second to last integral in (\ref{12.6}), finally, we first split according to
  \bea{12.10}
	& & \hs{-20mm}
	- \io \frac{\Teps\ln^2(\Teps+M)}{\Teps+M} \cdot \lan\B,\nas\veps\ran \nn\\
	&=& - \io \frac{(\Teps+M-M)\ln^2(\Teps+M)}{\Teps+M} \cdot \lan\B,\nas\veps\ran \nn\\
	&=& - \io \ln^2(\Teps+M) \cdot \lan\B,\nas\veps\ran
	+ M \io \frac{\ln^2(\Teps+M)}{\Teps+M} \cdot \lan\B,\nas\veps\ran
	\qquad \mbox{for all } t>0,
  \eea
  where by Young's inequality, and again by (\ref{12.4}),
  \bea{12.11}
	\hs{-8mm}
	\bigg| 
	M \io \frac{\ln^2(\Teps+M)}{\Teps+M} \cdot \lan\B,\nas\veps\ran
	\bigg|
	&\le& M |\B| \io \frac{\ln^2(\Teps+M)}{\Teps+M} \cdot |\nas\veps| \nn\\
	&\le& \frac{\kD}{2} \io \frac{\ln^2(\Teps+M) |\nas\veps|^2}{\Teps+M}
	+ \frac{M^2 |\B|^2}{2\kD} \io \frac{\ln^2(\Teps+M)}{\Teps+M} \nn\\
	&\le& \frac{\kD}{2} \io \frac{\ln^2(\Teps+M) |\nas\veps|^2}{\Teps+M}
	+ \frac{2 M^2 |\B|^2 |\Om|}{e^2 \kD}
	\quad \mbox{for all } t>0.
  \eea
  In the first summand on the right of (\ref{12.10}), however, we once more integrate by parts to verify that
  \bas
	& & \hs{-20mm}
	- \io \ln^2(\Teps+M) \cdot \lan\B,\nas\veps\ran \\
	&=& - \frac{1}{2} \sum_{i,j=1}^n \io \ln^2(\Teps+M) \B_{ij} \cdot \big\{ \pa_j (\veps)_i + \pa_i (\veps)_j \big\} \\
	&=& \sum_{i,j=1}^n \io \frac{\ln(\Teps+M)}{\Teps+M} \cdot \B_{ij} \pa_j \Teps (\veps)_i
	+ \sum_{i,j=1}^n \io \frac{\ln(\Teps+M)}{\Teps+M} \cdot \B_{ij} \pa_i \Teps (\veps)_j \\
	&=& \io \frac{\ln(\Teps+M)}{\Teps+M} (\B\cdot\na\Teps) \cdot\veps
	+ \io \frac{\ln(\Teps+M)}{\Teps+M} \na\Teps \cdot (\B\cdot\veps)
	\qquad \mbox{for all } t>0,
  \eas
  so that we may here estimate against the diffusion-induced contribution to (\ref{12.6}) by using Young's inequality according to
  \bas
	\bigg| 	- \io \ln^2(\Teps+M) \cdot \lan\B,\nas\veps\ran \bigg|
	&\le& 2|\B| \io \frac{\ln(\Teps+M)}{\Teps+M} |\na\Teps| \cdot |\veps| \nn\\
	&\le& \frac{D}{4} \io \frac{\ln^2(\Teps+M) |\na\Teps|^2}{(\Teps+M)^2}
	+ \frac{4|\B|^2}{D} \io |\veps|^2
	\qquad \mbox{for all } t>0.
  \eas
  Together with (\ref{12.11}) inserted into (\ref{12.10}), this shows that (\ref{12.6})-(\ref{12.9}) imply the inequality
  \bas
	\frac{d}{dt} \io \helleps(\Teps)
	&\ge& \frac{D}{4} \io \frac{\ln^2(\Teps+M) |\na\Teps|^2}{(\Teps+M)^2}
	+ \frac{\kD}{2} \io \frac{\ln^2(\Teps+M) |\nas\veps|^2}{\Teps+M} \\
	& & - c_1 \io |\veps|^2
	- c_2
	\qquad \mbox{for all } t>0
  \eas
  with $c_1:=\frac{4|\B|^2}{D}$ and $c_2:=\frac{2M^2 |\B|^2 |\Om|}{e^2 \kD}$, which upon a time integration entails that
  \bas
	& & \hs{-36mm}
	\frac{D}{4} \int_t^{t+1} \io \frac{\ln^2(\Teps+M) |\na\Teps|^2}{(\Teps+M)^2}
	+ \frac{\kD}{2} \int_t^{t+1} \io \frac{\ln^2(\Teps+M) |\nas\veps|^2}{\Teps+M} \nn\\
	&\le& \io \helleps\big(\Teps(\cdot,t+1)\big)
	- \io \helleps\big(\Teps(\cdot,t)\big) 
	+ c_1 \int_t^{t+1} \io |\veps|^2
	+ c_2 \\
	&\le& \frac{4}{e^2} \io \Keps\big(\Teps(\cdot,t+1)\big)
	+ c_1 \int_t^{t+1} \io |\veps|^2
	+ c_2 
	\qquad \mbox{for all } t>0
  \eas
  because of (\ref{12.5}).
  An application of Lemma \ref{lem6} therefore leads to the claimed conclusion, as clearly
  $\ln^2(\xi+M)\ge\ln^2(\xi+e)$ and $\xi+M=\frac{\xi+M}{\xi+1}\cdot (\xi+1) \le M\cdot (\xi+1)$ for all $\xi\ge 0$.
\qed
\mysection{$L^1$ boundedness and uniform integrability properties of $\Teps$}
In order to turn (\ref{12.2}) into an $L\log L$ bound for $\nas\veps$ in the style of that announced in (\ref{LlogL}),
but also to prepare an appropriate handling of the temperature-driven contribution to the first equation in (\ref{0eps}),
this section aims at deriving some $L^1$ bounds and equi-integrability features of $(\Teps)_{\eps\in (0,1)}$
in the settings either of Theorem \ref{theo37} or of Theorem \ref{theo38}.
\subsection{A time-independent $L^1$ bound for $\Teps$ when $\liminf_{\xi\to\infty} \kappa(\xi) >0$}
In the case when $\kappa$ is such that $\liminf_{\xi\to\infty} \kappa(\xi)>0$,
its primitive $K$ grows at least linearly, so that (\ref{energy}) suggests time-independent boundedness properties
of $\io \Theta$. 
In making this precise on the basis of Lemma \ref{lem6}, we will utilize the following decomposition of the function 
in (\ref{Keps}) that will also be used in the limit passage performed in Lemma \ref{lem35} later on.
\begin{lem}\label{lem131}
  Assume (\ref{131.1}).
  Then there exist $a>0$, $\xis>0$, $(\Keps^{(1)})_{\eps\in (0,1)} \subset C^0([0,\infty))$ and 
  $(\Keps^{(2)})_{\eps\in (0,1)} \subset C^0([0,\infty))$ such that
  \be{131.2}
	\Keps(\xi) = a\xi + \Keps^{(1)}(\xi) - \Keps^{(2)}(\xi)
	\qquad \mbox{for all $\xi\ge 0$ and } \eps\in (0,1),
  \ee
  that
  \be{131.3}
	0 \le \Keps^{(1)} \le \Keps
	\quad \mbox{and} \quad
	0 \le \Keps^{(2)} \le a\xis
	\qquad \mbox{on } [0,\infty),
  \ee
  and that
  \be{131.5}
	\Keps^{(1)} \to K^{(1)}
	\quad \mbox{and} \quad
	\Keps^{(2)} \to K^{(2)}
	\quad \mbox{in } C^0_{loc}([0,\infty))
	\qquad \mbox{as } \eps\searrow 0
  \ee
  with some $K^{(1)} \in C^0([0,\infty))$ and $K^{(2)} \in C^0([0,\infty))$.
\end{lem}
\proof
  Due to (\ref{131.1}) we can fix $c_1>0$ such that $\kappa(\xi)\ge c_1$ for all $\xi\ge 1$, and choosing $\xi_1\ge 1$
  such that $\frac{1}{(\xi_1+1)^2} \le \frac{c_1}{2}$, from (\ref{keps}) we obtain that
  \bas
	\keps(\xi) \ge \kappa(\xi) - \frac{1}{(\xi+1)^2} \ge c_1 - \frac{c_1}{2} = \frac{c_1}{2}
	\qquad \mbox{for all } \xi\ge\xi_1.
  \eas
  Therefore, writing
  \be{131.6}
	\Keps^{(1)}(\xi):=\Big( \Keps(\xi)-\frac{c_1\xi}{4}\Big)_+
	\quad \mbox{and} \quad
	\Keps^{(2)}(\xi):=\Big( \Keps(\xi)-\frac{c_1\xi}{4}\Big)_-
	\qquad \mbox{for $\xi\ge 0$ and } \eps\in (0,1),
  \ee
  we obtain nonnegative continuous functions on $[0,\infty)$ which for each 
  $\eps\in (0,1)$ satisfy $\Keps^{(1)}(\xi)=\Keps(\xi)-\frac{c_1\xi}{4}$
  and $\Keps^{(2)}(\xi)=0$ for all $\xi\ge 2\xi_1$, because for any such $\xi$ we have $\frac{\xi}{4}\ge \frac{\xi_1}{2}$ and hence
  \bas
	\Keps(\xi) - \frac{c_1 \xi}{4}
	&=& \int_0^\xi \keps(\sig) d\sig - \frac{c_1 \xi}{4}
	\ge \int_{\xi_1}^\xi \keps(\sig) d\sig - \frac{c_1 \xi}{4} \\
	&\ge& \frac{c_1}{2}(\xi-\xi_1) - \frac{c_1 \xi}{4}
	= \frac{c_1 \xi}{4} - \frac{c_1 \xi_1}{2}
	\ge 0.
  \eas
  The claim thus follows if we let $a:=\frac{c_1}{4}$ and $\xis:=2\xi_1$ as well as 
  $K^{(1)}(\xi):=(K(\xi)-\frac{c_1}{4}\xi)_+$ and $K^{(2)}(\xi):=(K(\xi)-\frac{c_1}{4}\xi)_-$ for $\xi\ge 0$,
  because $\Keps^{(2)}(\xi) \le \frac{c_1 \xi}{4} \le \frac{c_1}{4} \cdot 2\xi_1$
  for all $\xi\in [0,2\xi_1]$ and $\eps\in (0,1)$,
  and because $K_\eps \to K$ in $C^0_{loc}([0,\infty))$ as $\eps\searrow 0$ by (\ref{keps}), implying that (\ref{131.5})
  follows from (\ref{131.6}).
\qed
Indeed, (\ref{6.3}) can thereby be seen to entail the expected consequence:
\begin{lem}\label{lem13}
  If (\ref{131.1}) holds, 
  then there exists $(C(T))_{T>0} \subset (0,\infty)$ such that
  \be{13.2}
	\io \Teps(\cdot,t) \le C(T)
	\qquad \mbox{for all $t\in (0,T)$ and } \eps\in (0,1),
  \ee
  and that
  \be{13.3}
	\sup_{T>0} C(T)<\infty
	\quad \mbox{if (\ref{fg_decay}) holds.}
  \ee
\end{lem}
\proof
  We let $a, \xis, (\Keps^{(1)})_{\eps\in (0,1)}$ and $(\Keps^{(2)})_{\eps\in (0,1)}$ be as in Lemma \ref{lem131}, and then obtain
  from (\ref{131.2}) and (\ref{131.3}) that
  \bas
	a \io \Teps
	&=& \io \Keps(\Teps)
	- \io \Keps^{(1)}(\Teps)
	+ \io \Keps^{(2)}(\Teps) \\
	&\le& \io \Keps(\Teps)
	+ a\xis |\Om|
	\qquad \mbox{for all $t>0$ and } \eps\in (0,1).
  \eas
  The claim therefore directly results from Lemma \ref{lem6}.
\qed
\subsection{Uniform integrability properties of $\Teps$ when $\lim_{\xi\to\infty} \kappa(\xi) \ln^{(3-n)_+} \xi = +\infty$}
Next concerned with the scenario of Theorem \ref{theo38} in which (\ref{38.1}) is presupposed, 
we note that when $n\ge 3$ and thus $\kappa(\xi)$ is required to diverge to $+\infty$ as $\xi\to\infty$,
we may draw on the following simple observation on uniform integrability, actually valid irrespective of the size of $n$.
\begin{lem}\label{lem15}
  Suppose that
  \be{15.1}
	\kappa(\xi) \to + \infty
	\qquad \mbox{as } \xi\to\infty.
  \ee
  Then
  \be{15.2}
	\big(\Teps(\cdot,t)\big)_{t\in (0,T), \eps\in (0,1)}
	\mbox{ is uniformly integrable over $\Om$ for all } T>0,
  \ee
  and if moreover (\ref{fg_decay}) holds, then
  \be{15.3}
	\big(\Teps(\cdot,t)\big)_{t>0, \eps\in (0,1)}
	\mbox{ is uniformly integrable over $\Om$.}
  \ee
\end{lem}
\proof
  Since $\keps(\xi)\ge\kappa(\xi)-\frac{1}{(\xi+1)^2}$ for all $\xi\ge 0$ and $\eps\in (0,1)$ by (\ref{keps}), it follows that
  the functions in (\ref{Keps}) and (\ref{K}) satisfy
  \bas
	\Keps(\xi) \ge \int_0^\xi \kappa(\sig) d\sig - \int_0^\xi \frac{d\sig}{(\sig+1)^2}
	= K(\xi) - 1 + \frac{1}{\xi+1} \ge K(\xi)-1
	\qquad \mbox{for all $\xi\ge 0$ and } \eps\in (0,1)
  \eas
  and thus
  \bas
	\io K(\Teps) \le \io \Keps(\Teps) + |\Om|
	\qquad \mbox{for all $t>0$ and } \eps\in (0,1).
  \eas
  Since, on the other hand, (\ref{15.1}) implies that
  \bas
	\frac{K(\xi)}{\xi} 
	= \frac{1}{\xi} \int_0^\xi \kappa(\sig) d\sig
	\ge \frac{1}{\xi} \int_{\frac{\xi}{2}}^\xi \kappa(\sig) d\sig
	\ge \frac{1}{2} \inf_{\sig>\frac{\xi}{2}} \kappa(\sig)
	\ \to + \infty
	\qquad \mbox{as } \xi\to\infty,
  \eas
  in view of the de la Vall\'ee-Poussin theorem both properties in (\ref{15.2}) and (\ref{15.3}) result from (\ref{6.3}) 
  and (\ref{6.5}).
\qed
We next focus on the case when $n=2$, in which (\ref{38.1}) in fact allows for some $\kappa$ which are asymptotically singular
in the sense that $\kappa(\xi)\to 0$ as $\xi\to\infty$.
In such situations, (\ref{6.3}) apparently fails to provide $L^1$ estimates for the $\Teps$, so that then an additional use
of the diffusion-related regularity feature recorded in (\ref{12.1}) seems in order.
As to be seen in Lemma \ref{lem17}, in planar settings this indeed leads to some information at least in space-time $L^1$ settings
thanks to a corresponding Moser-Trudinger inequality.\abs
As a first preparation for this, the following statement on the existence of functions exhibiting 
arbitarily slow growth has taken from \cite[Lemma 2.1]{wang_win_JLMS}.
\begin{lem}\label{lem_JLMS}
  $\rho_0: [0,\infty)\to (0,\infty)$ be nondecreasing and such that $\rho_0(\xi)\to +\infty$ as $\xi\to\infty$.
  Then there exists a strictly increasing function $\rho\in C^0([0,\infty))$ such that $\rho(\xi) \to +\infty$ as 
  $\xi\to\infty$, that
  \bas
	0 < \rho(\xi) \le \rho_0(\xi)
	\qquad \mbox{for all } \xi\ge 0,
  \eas
  and that
  \bas
	\rho(\xi^2) \le 2\rho(\xi)
	\qquad \mbox{for all } \xi\ge 0.
  \eas
\end{lem}
Now if $\kappa$ does not decay too rapidly, then an elementary construction on the basis of this lemma may be used to find
a function which has two favorable growth properties, and which, in essence, is dominated by all of the $\Keps$ simultaneously:
\begin{lem}\label{lem16}
  Suppose that
  \be{16.1}
	\kappa(\xi) \cdot \ln \xi \to + \infty
	\qquad \mbox{as } \xi\to\infty.
  \ee
  Then there exist $C>0$ and $\uK\in C^1([0,\infty))$ such that $\uK>0$ on $[0,\infty)$ and
  \be{16.2}
	\uK(\xi) \le C \Keps(\xi) + C
	\qquad \mbox{for all $\xi\ge 0$ and } \eps\in (0,1),
  \ee
  that
  \be{16.3}
	0 \le \frac{(\xi+1) \uK'(\xi)}{\uK(\xi)+1} \le C
	\qquad \mbox{for all $\xi\ge 0$,}
  \ee
  and that
  \be{16.4}
	\frac{\uK(\xi) \ln \uK(\xi)}{\xi} \to + \infty
	\qquad \mbox{as } \xi\to\infty.
  \ee
\end{lem}
\proof
  Using that (\ref{16.1}) trivially implies that $(\xi+1)^2 \kappa(\xi) \to + \infty$ as $\xi\to\infty$, we pick $\xi_1>0$
  such that $\frac{1}{(\xi+1)^2} \le \frac{1}{2} \kappa(\xi)$ for all $\xi\ge\xi_1$, and hence infer from (\ref{keps}) that
  \be{16.44}
	\keps(\xi) \ge \kappa(\xi) - \frac{1}{(\xi+1)^2}
	\ge \frac{1}{2} \kappa(\xi)
	\qquad \mbox{for all $\xi\ge \xi_1$ and } \eps\in (0,1).
  \ee
  With this value of $\xi_1$ fixed henceforth, we let
  \be{16.5}
	\rho_0(\xi):=
	\lball
	\inf_{\sig\ge \xi_1} \big\{ \kappa(\sig) \cdot \ln (\sig+e) \big\},
	\qquad & \xi\in [0,\xi_1], \\[1mm]
	\inf_{\sig\ge \xi} \big\{ \kappa(\sig) \cdot \ln (\sig+e) \big\},
	\qquad & \xi>\xi_1,
	\ear
  \ee
  and thus obtain a positive and nondecreasing function $\rho_0$ on $[0,\infty)$	
  which due to (\ref{16.1}), now used in its full strength, satisfies $\rho_0(\xi)\to + \infty$ as $\xi\to\infty$.
  An application of Lemma \ref{lem_JLMS} therefore yields an increasing function $\rho\in C^0([0,\infty))$ which is such that
  \be{16.6}
	0 < \rho(\xi) \le \rho_0(\xi)
	\qquad \mbox{for all } \xi\ge 0,
  \ee
  that
  \be{16.66}
	\rho(\xi) \to + \infty
	\qquad \mbox{as } \xi\to\infty,
  \ee
  and that
  \be{16.7}
	\rho(\xi^2) \le 2\rho(\xi)
	\qquad \mbox{for all } \xi\ge 0.
  \ee
  We now let
  \be{16.8}
	\chi(\xi):=\frac{1}{\xi+1} \int_0^\xi \rho(\sig)d\sig,
	\qquad \xi\ge 0,
  \ee
  and define
  \be{16.9}
	\uK(\xi):=\chi(\xi) \cdot \frac{\xi+1}{\ln(\xi+e)},
	\qquad \xi\ge 0,
  \ee
  noting that $\chi$ and hence also $\uK$ belong to $C^1([0,\infty))$ by continuity of $\rho$, with
  \be{16.10}
	\chi'(\xi)
	= \frac{1}{\xi+1} \cdot \rho(\xi) - \frac{1}{(\xi+1)^2} \int_0^\xi \rho(\sig)d\sig
	\qquad \mbox{for all } \xi\ge 0.
  \ee
  As $\rho$ is nonnegative and nondecreasing, this implies that
  \bas
	\frac{1}{\xi+1} \cdot \rho(\xi)
	\ge \chi'(\xi)
	\ge \frac{1}{\xi+1} \cdot \rho(\xi) - \frac{1}{(\xi+1)^2} \cdot \xi\rho(\xi)
	\ge 0
	\qquad \mbox{for all } \xi\ge 0,
  \eas
  whence computing
  \bas
	\uK'(\xi) 
	= \chi'(\xi) \cdot \frac{\xi+1}{\ln(\xi+e)}
	+ \chi(\xi) \cdot \frac{\ln(\xi+e) - \frac{\xi+1}{\xi+e}}{\ln^2(\xi+e)}
	\qquad \mbox{for } \xi\ge 0,
  \eas
  we see that since $\frac{\xi+1}{\xi+e} \le 1 \le \ln(\xi+e)$ for all $\xi\ge 0$,
  \be{16.11}
	0 \le \uK'(\xi)
	\le \chi'(\xi) \cdot \frac{\xi+1}{\ln(\xi+e)}
	+ \frac{\chi(\xi)}{\ln(\xi+e)}
	\le \frac{\rho(\xi)+\chi(\xi)}{\ln(\xi+e)}
	\qquad \mbox{for all } \xi\ge 0.
  \ee
  We now observe that if we let $\xi_2:=\max\{4,2\xi_1\}$, then $(\frac{\xi}{2})^2 \ge \xi$ for all $\xi\ge \xi_2$, and that thus
  (\ref{16.7}) applies so as to ensure that since $\rho$ is nondecreasing,
  \be{16.111}
	\rho\Big(\frac{\xi}{2}\Big) \ge \frac{1}{2} \cdot \rho \Big( \Big(\frac{\xi}{2}\Big)^2\Big) \ge \frac{1}{2} \rho(\xi)
	\qquad \mbox{for all } \xi\ge \xi_2,
  \ee
  in line with (\ref{16.8}) meaning that
  \be{16.12}
	\chi(\xi)
	\ge \frac{1}{\xi+1} \int_{\frac{\xi}{2}}^\xi \rho\Big(\frac{\xi}{2}\Big) d\sig
	\ge \frac{1}{2(\xi+1)} \int_{\frac{\xi}{2}}^\xi \rho(\xi) d\sig
	= \frac{\xi}{4(\xi+1)} \rho(\xi) 
	\ge \frac{1}{5} \rho(\xi)
	\qquad \mbox{for all } \xi\ge \xi_2,
  \ee
  because $\frac{\xi}{\xi+1} \ge \frac{4}{5}$ for all $\xi\ge 4$.
  Therefore, (\ref{16.11}) together with (\ref{16.9}) implies that
  \be{16.13}
	0 \le \frac{(\xi+1)\uK'(\xi)}{\uK(\xi)+1}
	\le \frac{(\xi+1) \uK'(\xi)}{\uK(\xi)}
	\le \frac{(\xi+1)\cdot\frac{\rho(\xi)+\chi(\xi)}{\ln(\xi+e)}}{\chi(\xi)\cdot\frac{\xi+1}{\ln(\xi+e)}}
	= \frac{\rho(\xi)+\chi(\xi)}{\chi(\xi)}
	\le 6
	\qquad \mbox{for all } \xi\ge \xi_2,
  \ee
  while clearly
  \be{16.14}
	0 \le \frac{(\xi+1)\uK'(\xi)}{\uK(\xi)+1}
	\le (\xi+1)\uK'(\xi)
	\le (\xi+1) \cdot\frac{\rho(\xi)+\chi(\xi)}{\ln(\xi+e)}
	\le c_1:=(\xi_2+1)\cdot \big\{ \rho(\xi_2)+\chi(\xi_2)\big\}
	\ \mbox{for all } \xi\in [0,\xi_2),
  \ee
  so that (\ref{16.3}) holds whenever $C\ge\max\{6,c_1\}$.\abs
  Apart from that, (\ref{16.12}) along with (\ref{16.9}) implies that if we take $\xi_3\ge\xi_2$ such that in line with the
  unboundedness of $\rho$ we have $\frac{1}{5} \rho(\xi)\ge 2$ for all $\xi\ge \xi_3$, and that furthermore
  $\frac{\xi+1}{\ln(\xi+e)} \ge \frac{1}{2} \sqrt{\xi+e}$ for all $\xi\ge\xi_3$, then
  \bas
	\frac{\uK(\xi) \cdot \ln\uK(\xi)}{\xi}
	&\ge& \frac{\frac{1}{5} \rho(\xi) \cdot \frac{\xi+1}{\ln(\xi+e)} \cdot 
		\ln \Big\{ \frac{1}{5} \rho(\xi) \cdot \frac{\xi+1}{\ln(\xi+e)} \Big\}}{\xi} \\
	&\ge& \frac{\frac{1}{5} \rho(\xi) \cdot \frac{\xi+1}{\ln(\xi+e)} \cdot \ln \sqrt{\xi+e}}{\xi}
	= \frac{\xi+1}{10\xi} \rho(\xi)
	\ge \frac{1}{10} \rho(\xi)
	\qquad \mbox{for all } \xi\ge\xi_3,
  \eas
  so that (\ref{16.4}) results from (\ref{16.66}).\abs
  The estimate in (\ref{16.2}), finally, can be verified by combining (\ref{keps}) with (\ref{16.44}), (\ref{16.5}), (\ref{16.6})
  and (\ref{16.111}) 
  to see that
  \bas
	\Keps(\xi)
	&=& \int_0^\xi \keps(\sig) d\sig
	\ge \frac{1}{2} \int_{\xi_2}^\xi \kappa(\sig) d\sig
	= \frac{1}{2} \int_{\xi_2}^\xi \big\{ \kappa(\sig)\cdot\ln (\sig+e)\big\}\cdot\frac{1}{\ln(\sig+e)} d\sig \\
	&\ge& \frac{1}{2\ln(\xi+e)} \int_{\xi_2}^\xi \kappa(\sig)\cdot\ln (\sig+e) d\sig
	\ge \frac{1}{2\ln(\xi+e)} \int_{\xi_2}^\xi \rho_0(\sig) d\sig
	\ge \frac{1}{2\ln(\xi+e)} \int_{\xi_2}^\xi \rho(\sig) d\sig \\
	&\ge& \frac{1}{2\ln(\xi+e)} \int_{\frac{\xi}{2}}^\xi \rho\Big(\frac{\xi}{2}\Big) d\sig
	\ge \frac{\xi}{8\ln(\xi+e)} \rho(\xi)
	\qquad \mbox{for all $\xi\ge 2\xi_2$ and } \eps\in (0,1),
  \eas
  because $\xi_2\ge\xi_1$. Since $\chi(\xi)\le\frac{\xi}{\xi+1} \rho(\xi)$ for all $\xi\ge 0$ by (\ref{16.8}), namely,
  in view of (\ref{16.9}) this implies that $\uK(\xi) \le 8\Keps(\xi)$ for all $\xi\ge 2\xi_2$ and $\eps\in (0,1)$,
  and that thus (\ref{16.2}) holds whenever $C\ge\max\{8,\uK(2\xi_2)\}$.
\qed
As a last ingredient, let us recall from \cite[Lemma 2.3]{win_SIMA2020} the following consequence of the two-dimensional
Moser-Trudinger inequality.
\begin{lem}\label{lem444}
  Let $n=2$. Then for each $\eta>0$ there exists $C(\eta)>0$ with the property that
  whenever $\vp\in C^1(\bom)$ is nonnegative with $\vp\not\equiv 0$, we have
  \be{444.1}
	\io \vp \ln (\vp+1)
	\le \frac{1+\eta}{2\pi} \cdot \bigg\{ \io \vp \bigg\} \cdot \io\frac{|\nabla\vp|^2}{(\vp+1)^2}
	+  C(\eta) \cdot \bigg\{ \io \vp \bigg\}^3
	+ \Bigg\{ C(\eta) - \ln \bigg\{ \frac{1}{|\Omega|} \io \vp \bigg\} \Bigg\} \cdot \io \vp.
  \ee
\end{lem}
In fact, a combination of (\ref{12.1}) with (\ref{6.3}) can therefore be seen to entail some uniform integrability properties
of the $\Teps$ also for some decaying $\kappa$ when $n=2$:
\begin{lem}\label{lem17}
  Let $n=2$, and assume that $\kappa$ satisfies (\ref{16.1}).
  Then there exists $(C(T))_{T>0}\subset (0,\infty)$ such that
  \be{17.01}
	\int_t^{t+1} \io \Teps \le C(T)
	\qquad \mbox{for all $t\in (0,T)$ and } \eps\in (0,1),
  \ee
  and that
  \be{17.02}
	\sup_{T>0} C(T)<\infty
	\qquad \mbox{if (\ref{fg_decay}) holds.}
  \ee
  Moreover, 
  \be{17.1}
	(\Teps)_{\eps\in (0,1)}
	\mbox{ is uniformly integrable over $\Om\times (0,T)$ for all $T>0$,}
  \ee
  and if (\ref{fg_decay}) holds, then
  \be{17.2}
	\Big( \Teps(\cdot,\cdot+t_0) \Big)_{t_0\ge 0, \eps\in (0,1)}
	\mbox{ is uniformly integrable over $\Om\times (0,1)$.}
  \ee
\end{lem}
\proof
  Trivially estimating $\ln(\xi+e)\ge 1$ for $\xi\ge 1$ here, 
  from Lemma \ref{lem6} and Lemma \ref{lem12} we particularly infer the existence of $(c_i(T))_{T>0} \subset (0,\infty)$,
  $i\in\{1,2\}$, such that $\sup_{T>0} \big(c_1(T)+c_2(T)\big)<\infty$ if (\ref{fg_decay}) holds, and that for each $T>0$,
  \bas
	\io \Keps(\Teps) \le c_1(T)
	\qquad \mbox{for all $t\in (0,T+1)$ and } \eps\in (0,1)
  \eas
  as well as
  \bas
	\int_t^{t+1} \io \frac{|\na\Teps|^2}{(\Teps+1)^2} \le c_2(T)
	\qquad \mbox{for all $t\in (0,T)$ and } \eps\in (0,1).
  \eas
  We moreover combine Lemma \ref{lem444}, applied here to any $\eta>0$, with Young's inequality to find $c_3>0$ fulfilling
  \bas
	\io \vp \ln (\vp+1) \le c_3\cdot \bigg\{ \io \vp \bigg\} \cdot \io \frac{|\na\vp|^2}{(\vp+1)^2} 
	+ c_3 \cdot \bigg\{ \io \vp \bigg\}^3 + c_3
	\qquad \mbox{for all nonnegative } \vp\in C^1(\bom),
  \eas
  whereas Lemma \ref{lem16} yields $c_4>0$ and $c_5>0$ such that for the function $\uK$ constructed there we have
  \bas
	\uK(\xi) \le c_4\Keps(\xi)+c_4
	\quad \mbox{and} \quad
	0 \le \frac{(\xi+1)\uK'(\xi)}{\uK(\xi)+1} \le c_5
	\qquad \mbox{for all $\xi\ge 0$ and } \eps\in (0,1).
  \eas
  Given $T>0$, $t\in (0,T)$ and $\eps\in (0,1)$, we can therefore estimate
  \bas
	\int_t^{t+1} \io \uK(\Teps) \ln \big( \uK(\Teps)+1\big)
	&\le& c_3 \int_t^{t+1} \bigg\{ \io \uK(\Teps) \bigg\} \cdot \io \frac{\uK'^2(\Teps) |\na\Teps|^2}{(\uK(\Teps)+1)^2} \\
	& & + c_3 \int_t^{t+1} \bigg\{ \io \uK(\Teps)\bigg\}^3 + c_3 \\
	&\le& c_3 \int_t^{t+1} \bigg\{ c_4 \io \Keps(\Teps(\cdot,s)) + c_4 |\Om| \bigg\} 
		\cdot c_5^2 \io \frac{|\na\Teps(\cdot,s)|^2}{(\Teps(\cdot,s)+1)^2} ds \\
	& & + c_3 \int_t^{t+1} \bigg\{ c_4 \io \Keps(\Teps(\cdot,s)) + c_4 |\Om| \bigg\}^3 ds + c_3 \\
	&\le& c_3 c_4 \cdot (c_1(T)+|\Om|) c_5^2 c_2(T)
	+ c_3 c_4^3 \cdot (c_1(T)+|\Om|)^3 + c_3,
  \eas
  whence noting that according to (\ref{16.4}) we have $\frac{\uK(\xi) \ln (\uK(\xi)+1)}{\xi} \to + \infty$ as $\xi\to\infty$
  we do not only obtain (\ref{17.01})-(\ref{17.02}) as a particular consequence, but also infer (\ref{17.1})-(\ref{17.2})
  thanks to the de la Vall\'ee-Poussin theorem.
\qed
\mysection{An estimate for $\nas\veps$ in $L\log L$}
The following basic interpolation property will enable us to further develop (\ref{12.2}) by means of the $L^1$ bounds for
$\Teps$ obtained in the previous section.
\begin{lem}\label{lem231}
  If $\xi\ge e$ and $\eta\ge e^2$, then
  \be{231.1}
	\xi\ln\xi \le \frac{\xi^2 \ln^2 \eta}{\eta} + \eta.
  \ee
\end{lem}
\proof
  In the case when $\xi\ge e$ and $\eta\ge e^2$ are such that $\frac {\xi}{\ln\xi} \ge \frac{\eta}{\ln^2\eta}$, we have
  \bas
	\frac{\xi\ln\xi}{\frac{\xi^2 \ln^2\eta}{\eta}} 
	= \frac{\ln \xi}{\xi} \cdot \frac{\eta}{\ln^2\eta} \le 1,
  \eas
  so that (\ref{231.1}) trivially follows.
  If, conversely, $\xi\ge e$ and $\eta\ge e^2$ have the property that $\frac{\xi}{\ln\xi}<\frac{\eta}{\ln^2 \eta}$,
  then we may use that for 
  $\vp(\sig):=\frac{\sig}{\ln^2\sig}$, $\sig\ge e^2$, we have $\vp'(\sig)=\frac{\ln \sig -2}{\ln^3 \sig} \ge 0$ for all $\sig\ge e^2$.
  Therefore, namely, any such $\xi$ and $\eta$ must satisfy $\xi\ln\xi\le \eta$, for otherwise $\eta<\xi\ln\xi$ and hence,
  as $\ln\xi\ge 1$,
  \bas
	\frac{\eta}{\ln^2\eta}
	= \vp(\eta) 
	\le \vp(\xi\ln\xi)
	= \frac{\xi\ln\xi}{\ln^2(\xi\ln\xi)}
	\le \frac{\xi\ln\xi}{\ln^2\xi} = \frac{\xi}{\ln\xi},
  \eas
  which is absurd. Consequently, (\ref{231.1}) holds also in this case.
\qed
Indeed, in either of the scenarios in Theorem \ref{theo37} and Theorem \ref{theo38} we can hence establish 
the following statement of fundamental importance for our existence theory.
\begin{lem}\label{lem23}
  Suppose that either (\ref{131.1}) or (\ref{38.1}) holds.
  Then for every $T>0$ there exists $C(T)>0$ such that
  \be{23.1}
	\int_0^T \io |\nas\veps| \ln \big( |\nas\veps|+e \big)
	\le C(T)
	\qquad \mbox{for all } \eps\in (0,1);
  \ee
  in particular, 
  \be{23.2}
	(\nas\veps)_{\eps\in (0,1)}
	\mbox{ is uniformly integrable over } \Om\times (0,T).
  \ee
\end{lem}
\proof
  From Lemma \ref{lem231} we know that
  \bea{23.3}
	& & \hs{-16mm}
	\int_0^T \io \big( |\nas\veps|+e\big) \ln \big( |\nas\veps|+e\big) \nn\\
	&\le& \int_0^T \io \frac{(|\nas\veps|+e)^2 \ln^2(\Teps+e^2)}{\Teps+e^2}
	+ \int_0^T \io (\Teps+e^2)
	\qquad \mbox{for all $T>0$ and } \eps\in (0,1).
  \eea
  Since $\ln(\xi+e^2) \le \ln \big\{ (\xi+e)^2 \big\} = 2\ln (\xi+e)$ and $\ln (\xi+e) \le \xi+e$ for all $\xi\ge 0$, 
  due to Young's inequality we here have
  \bas
	\int_0^T \io \frac{(|\nas\veps|+e)^2 \ln^2(\Teps+e^2)}{\Teps+e^2}
	&\le& 8 \int_0^T \io \frac{\ln^2(\Teps+e) |\nas\veps|^2}{\Teps+1}
	+ 8e^2 \int_0^T \io \frac{\ln^2(\Teps+e)}{\Teps+e} \\
	&\le& 8 \int_0^T \io \frac{\ln^2(\Teps+e) |\nas\veps|^2}{\Teps+1}
	+ 8e^2 \int_0^T \io (\Teps+e),
  \eas
  whence (\ref{23.3}) implies that
  \bas
	\hs{-2mm}
	\int_0^T \io |\nas\veps| \ln \big( |\nas\veps|+e\big)
	&\le&
	\int_0^T \io \big( |\nas\veps|+e\big) \ln \big( |\nas\veps|+e\big) \\
	&\le& 8 \int_0^T \io \frac{\ln^2(\Teps+e) |\nas\veps|^2}{\Teps+1} \\
	& & + (8e^2+1) \int_0^T \io \Teps
	+ (8e^3+e^2) |\Om| T
	\quad \mbox{for all $T>0$ and } \eps\in (0,1).
  \eas
  In both cases under consideration, the claim therefore results from Lemma \ref{lem12} when combined with either Lemma \ref{lem13}
  or Lemma \ref{lem17}.
\qed

\mysection{First-order estimates for renormalized versions of $\Teps$}
The following two further consequences of Lemma \ref{lem12} on regularity of the temperature variable, 
mainly needed for the extraction of a corresponding 
pointwise a.e.~convergent subsequence, can be obtained in a rather straightforward manner.
\begin{lem}\label{lem20}
  If either (\ref{131.1}) or (\ref{38.1}) holds, 
  then given any $\phi\in C^\infty([0,\infty))$
  fulfilling $\phi'\in C_0^\infty([0,\infty))$, for each $T>0$ one can find $C(\phi,T)>0$ and $(C(\phi,T,p))_{p>n} \subset (0,\infty)$
  such that
  \be{20.1}
	\int_0^T \io \Big| \na \phi\big(\Keps(\Teps)\big)\Big|^2 \le C(\phi,T)
	\qquad \mbox{for all } \eps\in (0,1),
  \ee
  and that whenever $p>n$,
  \be{20.2}
	\int_0^T \Big\| \pa_t \phi\big(\Keps(\Teps(\cdot,t)+1)\big)\Big\|_{(W^{1,p}(\Om))^\star} dt \le C(\phi,T,p)
	\qquad \mbox{for all } \eps\in (0,1).
  \ee
\end{lem}
\proof
  Since $\int_2^\infty \frac{d\xi}{\ln\xi}=+\infty$, 
  in both considered cases we have $\int_0^\infty \kappa(\xi)d\xi= + \infty$.
  Thus, if we fix $c_1=c_1(\phi)>0$ such that $\phi'\equiv 0$ on $[c_1,\infty)$, then we can choose $\xi_1=\xi_1(\phi)>0$ in such 
  a way that $\int_0^{\xi_1} \kappa(\sig)d\sig \ge c_1+1$. As therefore
  \bas
	\Keps(\xi) \ge \int_0^\xi \kappa(\sig)d\sig - \int_0^\xi \frac{d\sig}{(\sig+1)^2}
	\ge c_1+1 - \int_0^\xi \frac{d\sig}{(\sig+1)^2} \ge c_1
	\qquad \mbox{for all $\xi\ge \xi_1$ and } \eps\in (0,1)
  \eas
  according to (\ref{keps}), writing $c_2\equiv c_2(\phi):=\|\kappa\|_{L^\infty((0,\xi_1))} +1$ and again relying on (\ref{keps})
  we obtain that
  \bea{20.4}
	& & \hs{-30mm}
	\big\{ (\xi+1)\keps(\xi) + (\xi+1) + \xi (\xi+1)^\frac{1}{2} \big\} \cdot \big| \phi'(\Keps(\xi))\big| \nn\\
	&\le& c_3\equiv c_3(\phi)
	:=\big\{ ((\xi_1+1)c_2 + (\xi_1+1) + \xi_1 (\xi_1+1)^\frac{1}{2} \big\} \cdot \|\phi'\|_{L^\infty((0,\infty))}
  \eea
  as well as
  \be{20.5}
	(\xi+1)^2 \cdot \big| \phi''(\Keps(\xi))\big| \cdot\keps(\xi)
	\le (\xi+1)^2 \cdot \big| \phi''(\Keps(\xi))\big| \cdot (\kappa(\xi)+1)
	\le c_4\equiv c_4(\phi):=(\xi_1+1)^2 \|\phi''\|_{L^\infty((0,\infty))} \cdot c_2
  \ee
  for all $\xi\ge 0$ and $\eps\in (0,1)$.
  Now for fixed $T>0$, (\ref{20.4}) guarantees that
  \bas
	\int_0^T \io \Big| \na\phi\big(\Keps(\Teps)\big)\Big|^2
	&=& \int_0^T \io \phi'^2\big(\Keps(\Teps)\big) \keps^2(\Teps) |\na\Teps|^2 \nn\\
	&\le& c_3^2 \int_0^T \io \frac{|\na\Teps|^2}{(\Teps+1)^2},
  \eas
  so that (\ref{20.1}) becomes a particular consequence of Lemma \ref{lem12}, as $\ln^2(\xi+e)\ge 1$ for all $\xi\ge 0$.\abs
  Furthermore, given $p>n$ we can pick $c_5=c_5(p)>0$ such that any $\vp\in C^1(\bom)$ fulfilling $\|\vp\|_{W^{1,p}(\Om)} \le 1$
  satisfies $\|\vp\|_{L^\infty(\Om)} + \|\na\vp\|_{L^2(\Om)} \le c_5$, whence from the third equation in (\ref{0eps}) and
  the Cauchy-Schwarz inequality we obtain that
  \bas
	\bigg| \io \pa_t \phi\big(\Keps(\Teps)\big) \vp \bigg|
	&=& \bigg| \io \phi'\big(\Keps(\Teps)\big) \keps(\Teps) \Theta_{\eps t} \vp \bigg| \\
	&=& \bigg| \io \phi'\big(\Keps(\Teps)\big) \cdot \Big\{
	D\Del\Teps + \lan\D:\nas\veps,\nas\veps\ran - \Teps \lan\B,\nas\veps\ran + \geps \Big\} \cdot \vp \bigg| \\
	&=& \bigg| - D \io \phi''\big(\Keps(\Teps)\big) \keps(\Teps) |\na\Teps|^2 \vp
	- D \io \phi'\big(\Keps(\Teps)\big) \na\Teps\cdot\na\vp \nn\\
	& & \hs{5mm}
	+ \io \phi'\big(\Keps(\Teps)\big) \lan\D:\nas\veps,\nas\veps\ran \vp
	- \io \Teps \phi'\big(\Keps(\Teps)\big) \lan\B,\nas\veps\ran \vp \\
	& & \hs{5mm}
	+ \io \phi'\big(\Keps(\Teps)\big) \geps\vp \bigg| \\
	&\le& c_5 D \io \big| \phi''\big(\Keps(\Teps)\big)\big| \keps(\Teps) |\na\Teps|^2
	+ c_5 D \cdot \bigg\{ \io \phi'^2\big(\Keps(\Teps)\big) |\na\Teps|^2 \bigg\}^\frac{1}{2} \\
	& & + c_5 \io \big|\phi'\big(\Keps(\Teps)\big)\big| \lan\D:\nas\veps,\nas\veps\ran
	+ c_5 |\B| \io \Teps \big|\phi'\big(\Keps(\Teps)\big)\big| \cdot |\nas\veps| \\
	& & + c_5 \io \big|\phi'\big(\Keps(\Teps)\big)\big| \geps
	\qquad \mbox{for all $t>0$ and $\eps\in (0,1)$.}
  \eas
  In view of (\ref{20.4}), (\ref{20.5}) and again the Cauchy-Schwarz inequality, we thus infer that
  for all $t>0$ and $\eps\in (0,1)$,
  \bas
	& & \hs{-20mm}
	\Big\| \pa_t \phi\big(\Keps(\Teps)\big) \Big\|_{(W^{1,p}(\Om))^\star} \\
	&\le& c_5 D \cdot c_4 \io \frac{|\na\Teps|^2}{(\Teps+1)^2}
	+ c_5 D \cdot \bigg\{ c_3^2 \io \frac{|\na\Teps|^2}{(\Teps+1)^2} \bigg\}^\frac{1}{2} \\
	& & + c_5 c_3 \io \frac{\lan\D:\nas\veps,\nas\veps\ran}{\Teps+1}
	+ c_5 |\B| \cdot \bigg\{ \io \frac{|\nas\veps|^2}{\Teps+1} \bigg\}^\frac{1}{2} \cdot \big\{ c_3^2 |\Om|\big\}^\frac{1}{2}
	+ c_5 c_3 \io \geps,
  \eas
  so that by Young's inequality,
  \bas
	& & \hs{-16mm}
	\int_0^T \Big\| \pa_t \phi\big(\Keps(\Teps(\cdot,t))\big) \Big\|_{(W^{1,p}(\Om))^\star} dt \\
	&\le& (c_4 c_5 D + c_3 c_5 D) \int_0^T \io \frac{|\na\Teps|^2}{(\Teps+1)^2}
	+ c_3 c_5 \int_0^T \io \frac{\lan\D:\nas\veps,\nas\veps\ran}{\Teps+1}
	+ c_5 |\B| \int_0^T \io \frac{|\nas\veps|^2}{\Teps+1} \\
	& & + c_3 c_5 \int_0^T \io \geps
	+ \frac{c_3 c_5 D T}{4}
	+ \frac{c_3^2 c_5 |\B| \cdot |\Om| \cdot T}{4}
	\qquad \mbox{for all } \eps\in (0,1).
  \eas
  Combining (\ref{fgec}) with Lemma \ref{lem12} therefore yields (\ref{20.2}).
\qed 
\mysection{Passing to the limit. Existence of solutions}
\subsection{Extracting subsequences and deriving (\ref{wu})}
A final preparation will be needed to conclude convergence of $\Teps(x,t)$ whenever $\Keps(\Teps(x,t))$ approximates
a finite number along a null sequence of parameters $\eps$:
\begin{lem}\label{lem22}
  Assume that $\int_0^\infty \kappa(\xi) d\xi= +\infty$,
  and suppose that $(\eps_j)_{j\in\N} \subset (0,1)$, $(\xi_j)_{j\in\N} \subset [0,\infty)$ and $z\ge 0$
  are such that $\eps_j\searrow 0$ and 
  \be{22.1}
	K_{\eps_j}(\xi_j)\to z
	\qquad \mbox{as } j\to\infty.
  \ee
  Then
  \be{22.2}
	\xi_j \to K^{-1}(z)
	\qquad \mbox{as } j\to\infty.
  \ee
\end{lem}
\proof
  If $(j_i)_{i\in\N} \subset\N$ was such that $j_i\to\infty$ as $i\to\infty$, but that with $\xi_0:=K^{-1}(z)$ and some
  $\del>0$ we had $\xi_{j_i} \ge \xi_0+\del$ for all $i\in\N$, then since $K_{\eps_{j_i}}'\ge 0$ for all $i\in\N$, recalling
  (\ref{K1}) we would see that
  \bas
	K_{\eps_{j_i}}(\xi_{j_i}) \ge K_{\eps_{j_i}}(\xi_0+\del) 
	\ge K(\xi_0+\del)-\eps_{j_i}
	= z + \big\{ K(\xi_0+\del)-z\big\} - \eps_{j_i}
	\qquad \mbox{for all } i\in\N.
  \eas
  As $K(\xi_0+\del)>K(\xi_0)=z$ by strict monotonicity of $K$, however, this would contradict (\ref{22.1}), whence, in fact,
  $\limsup_{j\to\infty} \xi_j\le K^{-1}(z)$.
  Along with an analogous argument asserting that $\liminf_{j\to\infty} \xi_j\ge K^{-1}(z)$, this confirms (\ref{22.2}).
\qed
We are now in the position to construct candidates for solutions to (\ref{0}) by means of suitable subsequence extraction.\abs
Our first statement in this regard addresses the situation from Theorem \ref{theo37}, and yields a triple $(u,\Theta,\mu)$
which complies with the respective regularity requirements in (\ref{reg37}) and Definition \ref{dw}, and which can already 
here be seen to satisfy (\ref{wu}): 
\begin{lem}\label{lem24}
  If (\ref{131.1}) holds, 
  then there exists $(\eps_j)_{j\in\N} \subset (0,1)$ such that $\eps_j\searrow 0$ as $j\to\infty$, and that with some 
  \be{24.1}
	\lbal
	u \in C^0([0,\infty);L^2(\Om;\R^n)) \cap L^\infty_{loc}([0,\infty);W_0^{1,2}(\Om;\R^n)), \\[1mm]
	\Theta\in L^\infty_{loc}([0,\infty);L^1(\Om))
	\qquad \mbox{and} \\[1mm]
	\mu \in L^\infty_{w-\star,loc}([0,\infty);\M_+(\bom))
	\ear
  \ee
  which are such 
  that $\Theta\ge 0$ a.e.~in $\Om\times (0,\infty)$, that (\ref{w23}), (\ref{w44}) and (\ref{w5}) hold,
  and that for each $\phi_0\in C^1([0,\infty))$ with $\phi_0'$ having compact support, and for arbitrary $T>0$, we have
  \begin{eqnarray}
	& & \veps \wto u_t
	\qquad \mbox{in } L^2(\Om\times (0,T);\R^n), 
	\label{24.4} \\
	& & \nas\veps\wto \nas u_t
	\qquad \mbox{in } L^1(\Om\times (0,T);\R^n), 
	\label{24.44} \\
	& & \ueps \to u
	\qquad \mbox{in } C^0([0,T];L^2(\Om;\R^n)),
	\label{24.45} \\
	& & \ueps \wto u
	\qquad \mbox{in } L^2((0,T);W^{1,2}(\Om;\R^n)),
	\label{24.5} \\
	& & \Teps \to \Theta
	\qquad \mbox{a.e.~in } \Om\times (0,T),
	\label{24.6} \\
	& & \Teps \wsto \Theta+\mu
	\qquad \mbox{in } L^\infty((0,T);\M_+(\bom))
	\qquad \mbox{and}
	\label{24.7} \\
	& & \na \phi_0(\Teps) \wto \na \phi_0(\Theta)
	\qquad \mbox{in } L^2(\Om\times (0,T);\R^n)
	\label{24.8} 
  \end{eqnarray}
  as $\eps=\eps_j\searrow 0$.
  These limit objects are such that (\ref{wu}) holds for each $\vp\in C_0^\infty(\Om\times [0,\infty);\R^n)$.
\end{lem}
\proof
  A combination of Lemma \ref{lem6} with Lemma \ref{lem13} and the fact that $u_{\eps t}=\veps$ for $\eps\in (0,1)$
  shows that for each $T>0$,
  \bas
	(\ueps)_{\eps\in (0,1)} 
	\mbox{ is bounded in } L^\infty((0,T);W_0^{1,2}(\Om;\R^n)),
  \eas
  that
  \bas
	(u_{\eps t})_{\eps\in (0,1)} \mbox{ and } 
	(\veps)_{\eps\in (0,1)} 
	\mbox{ are bounded in } L^\infty((0,T);L^2(\Om;\R^n)),
  \eas
  and that
  \be{24.87}
	(\Teps)_{\eps\in (0,1)} \mbox{ and } 
	\big(\Keps(\Teps)\big)_{\eps\in (0,1)}
	\mbox{ are bounded in } L^\infty((0,T);L^1(\Om)).
  \ee
  Furthermore, for $k\in\N$ fixing $\phi_k\in C^\infty([0,\infty))$ such that $\phi_k(\xi)=\xi$ for all $\xi\in [0,k]$ and
  $\phi_k'\equiv 0$ on $[k+1,\infty)$, from Lemma \ref{lem20} we obtain that
  \bas
	\big(\phi_k(\Keps(\Teps))\big)_{\eps\in (0,1)}
	\mbox{ is bounded in } L^2((0,T);W^{1,2}(\Om)),
  \eas
  and that, e.g.,
  \bas
	\big(\pa_t \phi_k(\Keps(\Teps))\big)_{\eps\in (0,1)}
	\mbox{ is bounded in } L^1\big((0,T);(W^{1,n+1}(\Om))^\star\big)
  \eas
  for all $T>0$. 
  Relying on the fact that $(\phi_k)_{k\in\N}$ is countable, a straightforward extraction procedure involving the Banach-Alaoglu
  theorem and an Aubin-Lions lemma thus yields $(\eps_j)_{j\in\N} \subset (0,1)$ and functions
  \be{24.9}
	\lbal
	v \in L^\infty_{loc}([0,\infty);L^2(\Om;\R^n)), \\[1mm]
	u \in C^0([0,\infty);L^2(\Om;\R^n)) \cap L^\infty_{loc}([0,\infty);W_0^{1,2}(\Om;\R^n)) \qquad \mbox{and} \\[1mm]
	z \in L^\infty_{loc}([0,\infty);L^1(\Om)),
	\ear
  \ee
  as well as some $\wh{\mu}\in L^\infty_{w-\star,loc}([0,\infty);\M(\bom))$, such that $z\ge 0$ a.e.~in $\Om\times (0,\infty)$
  and (\ref{w44}) holds, that $\eps_j\searrow 0$ as $j\to\infty$, and that $\veps\wto v$ in $L^2_{loc}(\bom\times [0,\infty);\R^n)$,
  \be{24.99}
	\Keps(\Teps) \to z	
	\qquad \mbox{a.e.~in $\Om\times (0,\infty)$}
  \ee
  and
  \be{24.98}
	\Teps \wsto \wh{\mu}
	\qquad \mbox{in } L^\infty((0,T);\M_+(\bom))
  \ee
  as well as (\ref{24.45}) and (\ref{24.5}) are valid as $\eps=\eps_j\searrow 0$.\abs
  Since $u_{\eps t}=\veps$ for all $\eps\in (0,1)$, these limits must satisfy $v=u_t$ a.e.~in $\Om\times (0,\infty)$,
  whence (\ref{24.4}) and, in view of (\ref{24.9}), also the first inclusion in (\ref{w23}) follow.
  Apart from that, as a consequence of Lemma \ref{lem23} and the Dunford-Pettis theorem,
  \bas
	(\nas\veps)_{\eps\in (0,1)} 
	\mbox{ is relatively compact with respect to the weak topology in } L^1(\Om\times (0,T);\R^{n\times n})
  \eas
  for all $T>0$, 
  so that in light of (\ref{24.4}) we obtain the second property in (\ref{w23}) and (\ref{24.44}) as $\eps=\eps_j\searrow 0$.\abs
  We next note that in view of Lemma \ref{lem22} it follows from (\ref{24.99}) that necessarily $\Teps \to K^{-1}(z)$
  a.e.~in $\Om\times (0,\infty)$ as $\eps=\eps_j\searrow 0$, due to (\ref{24.87})
  meaning that if we let $\Theta:=K^{-1}(z)$ and $\mu:=\Theta-\wh{\mu}$, then we obtain that, indeed,
  $0\le \Theta\in L^\infty_{loc}([0,\infty);L^1(\Om))$ and $\mu\in L^\infty_{w-\star,loc}([0,\infty);\M(\bom))$
  and that (\ref{w44}) is satisfied, and that (\ref{24.6}) and (\ref{24.7}) hold as $\eps=\eps_j\searrow 0$.
  The claimed nonnegativity of $\mu$ is entailed by Fatou's lemma: For $T>0$ and arbitrary nonnegative $\vp\in C^0(\bom\times [0,T])$,
  namely, on the one hand we have
  \bas
	\int_0^T \io \Teps \vp \to \int_0^T \io \Theta \vp + \int_0^T \int_{\bom} \vp d\mu
	\qquad \mbox{as } \eps=\eps_j\searrow 0
  \eas
  by (\ref{24.7}), while on the other hand,
  \bas
	\int_0^T \io \Theta\vp \le \liminf_{\eps=\eps_j\searrow 0} \int_0^T \io \Teps \vp
  \eas
  according to (\ref{24.6}). Combining these relations shows that $\int_0^T \int_{\bom} \vp d\mu \ge 0$ for any such $\vp$ and $T$,
  whence indeed $\mu\ge 0$.\abs
  Having (\ref{24.6}) at hand now enables us to verify (\ref{w5}) and (\ref{24.8}) as well:
  Given $T>0$ and $\phi_0\in C^1([0,\infty))$ such that $\supp\phi_0$ is bounded, once more resorting to Lemma \ref{lem12} we readily
  obtain that $(\na\phi_0(\Teps))_{\eps\in (0,1)}$ is bounded in $L^2(\Om\times (0,T);\R^n)$.
  Since (\ref{24.6}) and the Vitali convergence theorem assert that $\phi_0(\Teps)\to \phi_0(\Theta)$ in $L^2(\Om\times (0,T))$
  as $\eps=\eps_j\searrow 0$, it follows that whenever $(j_i)_{i\in\N}\subset\N$ and $\wh{z}\in L^2(\Om\times (0,T);\R^n)$ 
  are such that $j_i\to\infty$ and $\na\phi_0(\Theta_{\eps_{j_i}})\wto \wh{z}$ in $L^2(\Om\times (0,T))$ as $i\to\infty$,
  we must have $\wh{z}=\na\phi_0(\Theta)$, whence indeed both (\ref{w5}) and (\ref{24.8}) follow for any such $\phi_0$.\abs
  The identity in (\ref{wu}) now becomes a direct consequence of the approximation properties gathered so far:
  According to an integration by parts in (\ref{0eps}), for each $\vp\in C_0^\infty(\Om\times [0,\infty);\R^n)$ we have
  \bea{24.11}
	& & \hs{-30mm}
	\int_0^\infty \io \ueps\cdot\vp_{tt}
	+ \io u_{0\eps} \cdot\vp_t(\cdot,0)
	- \io v_{0\eps}\cdot\vp(\cdot,0) 
	+ \eps \int_0^\infty \io \veps\cdot\Del^{2m} \vp \nn\\
	&=& - \int_0^\infty \io \lan\D:\nas\veps,\na\vp\ran
	- \int_0^\infty \io \lan\C:\nas\ueps,\na\vp\ran \nn\\
	& & + \int_0^\infty \io \Teps \lan\B,\na\vp\ran
	+ \int_0^\infty \io \feps\cdot\vp
	\qquad \mbox{for all } \eps\in (0,1),
  \eea
  where due to (\ref{24.4}) and (\ref{24.5}),
  \bas
	\eps \int_0^\infty \io \veps\cdot\Del^{2m} \vp \to 0
  \eas
  as well as
  \bas
	\int_0^\infty \io \ueps\cdot\vp_{tt} \to \int_0^\infty \io u\cdot\vp_{tt}
	\qquad \mbox{and} \qquad
	- \int_0^\infty \io \lan\C:\nas\ueps,\na\vp\ran
	\to - \int_0^\infty \io \lan\nas \C:\nas u,\na\vp\ran
  \eas
  as $\eps=\eps_j\searrow 0$, while by (\ref{24.44}) and (\ref{24.7}),
  \bas
	- \int_0^\infty \io \lan\D:\nas\veps,\na\vp\ran 
	\to - \int_0^\infty \io \lan\D:\nas u_t,\na\vp\ran
  \eas
  and 
  \be{24.999}
	\int_0^\infty \io \Teps\lan\B,\na\vp\ran
	\to \int_0^\infty \io \Theta\lan\B,\na\vp\ran
	+ \int_0^\infty \int_{\bom} \lan\B,\na\vp\ran d\mu
  \ee
  as $\eps=\eps_j\searrow 0$. Since
  \bas
	\io u_{0\eps} \cdot \vp_t(\cdot,0)
	\to \io u_0\cdot\vp_t(\cdot,0),
	\
	\io v_{0\eps}\cdot\vp(\cdot,0)
	\to \io u_{0t} \cdot \vp(\cdot,0)
	\ \mbox{and} \
	\int_0^\infty \io \feps\cdot\vp
	\to \int_0^\infty \io f\cdot\vp
  \eas
  as $\eps=\eps_j\searrow 0$ by (\ref{iec}) and (\ref{fgec}), from (\ref{24.11}) we thus in fact obtain (\ref{wu}).
\qed
Instead of Lemma \ref{lem13} employing Lemma \ref{lem15} and Lemma \ref{lem17}, in quite a similar fashion we can derive
the following counterpart aiming at the framework of Theorem \ref{theo37}.
\begin{lem}\label{lem244}
  Assume (\ref{38.1}).		
  Then one can find $(\eps_j)_{j\in\N} \subset (0,1)$ as well as
  \be{244.1}
	\lbal
	u \in L^\infty_{loc}([0,\infty);W_0^{1,2}(\Om;\R^n))
	\qquad \mbox{and} \\[1mm]
	\Theta\in L^1_{loc}(\bom\times [0,\infty))
	\ear
  \ee
  such that $\Theta\ge 0$ a.e.~in $\Om\times (0,\infty)$, that (\ref{w23}), (\ref{w44}) and (\ref{w5}) are valid,
  that $\eps_j\searrow 0$ as $j\to\infty$, that for all $T>0$ and 
  any $\phi_0\in C^1([0,\infty))$ such that $\supp \phi_0'$ is bounded,
  \begin{eqnarray}
	& & \veps \wto u_t
	\qquad \mbox{in } L^2(\Om\times (0,T);\R^n), 
	\label{244.4} \\
	& & \nas\veps\wto \nas u_t
	\qquad \mbox{in } L^1(\Om\times (0,T);\R^n), 
	\label{244.44} \\
	& & \ueps \to u
	\qquad \mbox{in } C^0([0,T];L^2(\Om;\R^n)),
	\label{244.45} \\
	& & \ueps \wto u
	\qquad \mbox{in } L^2((0,T);W^{1,2}(\Om;\R^n)),
	\label{244.5} \\
	& & \Teps \to \Theta
	\qquad \mbox{in $L^1(\Om\times (0,T))$ and a.e.~in } \Om\times (0,\infty),
	\qquad \mbox{and}
	\label{244.6} \\
	& & \na \phi_0(\Teps) \wto \na \phi_0(\Theta)
	\qquad \mbox{in } L^2(\Om\times (0,T);\R^n)
	\label{244.8} 
  \end{eqnarray}
  as $\eps=\eps_j\searrow 0$,
  and that the identity in (\ref{wu}) holds with $\mu=0$ for each $\vp\in C_0^\infty(\Om\times [0,\infty);\R^n)$.
\end{lem}
\proof
  Lemma \ref{lem6} and Lemma \ref{lem20} still assert boundedness of $(\ueps)_{\eps\in (0,1)}$ 
  in $L^\infty((0,T);W_0^{1,2}(\Om;\R^n))$, of $(u_{\eps t})_{\eps\in (0,1)}$ and $(\veps)_{\eps\in (0,1)}$ in 
  $L^\infty((0,T);L^2(\Om;\R^n))$, 
  of $\big(\Keps(\Teps)\big)_{\eps\in (0,1)}$ in $L^\infty((0,T);L^1(\Om))$,
  of $\big(\phi(\Keps(\Teps))\big)_{\eps\in (0,1)}$ in $L^2((0,T);W^{1,2}(\Om))$ and of
  $\big(\pa_t \phi(\Keps(\Teps))\big)_{\eps\in (0,1)}$ in $L^1\big((0,T);(W^{1,n+1}(\Om))^\star\big)$
  for all $T>0$ and any $\phi\in C^\infty([0,\infty))$ fulfilling $\phi'\in C_0^\infty([0,\infty))$,
  while from Lemma \ref{lem23} we still know that $(\nas\veps)_{\eps\in (0,1)}$ is uniformly integrable over 
  $\Om\times (0,T)$ 
  for all $T>0$.
  A verbatim copy of the reasoning from Lemma \ref{lem24} thus yields $(\eps_j)_{j\in\N} \subset (0,1)$ such that
  $\eps_j\searrow 0$ as $j\to\infty$, and that with some
  $u\in C^0([0,\infty);L^2(\Om;\R^n)) \cap L^\infty_{loc}([0,\infty);W_0^{1,2}(\Om;\R^n))$ 
  and some measurable $\Theta: \Om\times (0,\infty)\to [0,\infty)$
  fulfilling (\ref{w23}) and (\ref{w44}), for each $T>0$ we have (\ref{244.4}), (\ref{244.44}) and (\ref{244.5}) and, moreover,
  \be{244.9}
	\Teps\to\Theta
	\qquad \mbox{a.e.~in $\Om\times (0,T)$}
  \ee
  as $\eps=\eps_j\searrow 0$.\abs
  The novelty now consists in relying, rather than on Lemma \ref{lem13}, on Lemma \ref{lem17} and Lemma \ref{lem15}
  here to see that thanks to
  our hypothesis (\ref{38.1}), $(\Teps)_{\eps\in (0,1)}$ is uniformly integrable over $\Om\times (0,T)$ for all $T>0$:
  Due to the Vitali convergence theorem, namely, (\ref{244.9}) therefore asserts that $\Teps\to\Theta$ in $L^1(\Om\times (0,T))$
  for all $T>0$ as $\eps=\eps_j\searrow 0$.
  This firstly enables us to derive the claims concerning (\ref{w5}) and (\ref{244.8}) in much the same manner as done
  in the corresponding part of Lemma \ref{lem24}; secondly, in the integral identity obtained upon testing the first 
  equation in (\ref{0eps}) by an arbitrary $\vp\in C_0^\infty(\Om\times [0,\infty);\R^n)$ (cf.(\ref{24.11}))
  this ensures that, unlike in (\ref{24.999}),
  \bas
	\int_0^\infty \io \Teps\lan\B,\na\vp\ran
	\to \int_0^\infty \io \Theta\lan\B,\na\vp\ran
  \eas
  as $\eps=\eps_j\searrow 0$,
  whence (\ref{wu}) actually holds with $\mu=0$ for any such $\vp$.
  Under the condition in (\ref{reg388}), finally, Lemma \ref{lem13} can still be applied here, so that 
  the additonal property in (\ref{reg388}) is a consequence of (\ref{244.9}) and Fatou's lemma.
\qed
\subsection{The inequality (\ref{wt})}
Our derivation of the weak parabolic inequality in (\ref{wt}) will rely on the following observation on weak lower semicontinuity.
\begin{lem}\label{lem32}
  Let $(\E_{\mathbf{ijkl}})_{(\mathbf{i},\mathbf{j},\mathbf{k},\mathbf{l})\in\{1,...,n\}^4} \in \R^{n\times n\times n \times n}$ 
  be such that 
  $\E_{\mathbf{ijkl}}=\E_{\mathbf{klij}}=\E_{\mathbf{jikl}}$ for all $(\mathbf{i},\mathbf{j},\mathbf{k},\mathbf{l})\in\{1,...,n\}^4$,
  and that
  \be{32.1}
	\lan\E:\A,\A\ran \ge k_{\E} |\A|^2
	\qquad \mbox{for all } \A\in \R^{n\times n}_{sym}
  \ee
  with some $k_{\E}>0$. Then the following holds:\\
  i) \ There exists $\sqrt{\E}:\R^{n\times n}_{sym} \to \R^{n\times n}_{sym}$ such that
  \be{32.2}
	\lan\sqrt{\E}:\A,\A'\ran
	= \lan \A, \sqrt{\E}:\A'\ran
	\qquad \mbox{for all $\A\in\R^{n\times n}_{sym}$ and } \A'\in \R^{n\times n}_{sym},
  \ee
  that
  \be{32.3}
	\sqrt{\E}:(\sqrt{\E}:\A)=\A
	\qquad \mbox{for all } \A\in \R^{n\times n}_{sym},
  \ee
  and that
  \be{32.4}
	\lan \sqrt{\E}:\A,\A\ran
	\ge \sqrt{k_{\E}} |\A|^2
	\qquad \mbox{for all } \A\in \R^{n\times n}_{sym}.
  \ee
  ii) \ If $T>0$, $(h_j)_{j\in\N} \subset L^\infty(\Om\times (0,T))$, $h\in L^\infty(\Om\times (0,T))$,
  $(\vp_j)_{j\in\N} \subset L^2(\Om\times (0,T);W^{1,2}(\Om;\R^n))$ and 
  $\vp \in L^2(\Om\times (0,T);W^{1,2}(\Om;\R^n))$ are such that
  \be{32.5}
	h_j \to h
	\qquad \mbox{a.e.~in } \Om\times (0,T)
  \ee
  and
  \be{32.6}
	\nas\vp_j \wto \nas \vp
	\qquad \mbox{in } L^1(\Om\times (0,T);\R^{n\times n})
  \ee
  as $j\to\infty$, and that $h_j\to 0$ a.e.~in $\Om\times (0,T)$ for all $j\in\N$,
  then
 \be{32.8}
	\int_0^T \io h \lan \E:\nas\vp,\nas\vp\ran
	\le \liminf_{j\to\infty} \int_0^T \io h_j \lan \E:\nas\vp_j,\nas\vp_j\ran.
  \ee
\end{lem}
\proof
  i) \ As is evident, the vector space $X:=\R^{n\times n}_{sym} \cong \R^d$, $d:=\frac{n(n+1)}{2}$, becomes a Hilbert space when 
  equipped with the norm that is induced by the scalar product $\lan\cdot,\cdot\ran: X\times X\to\R$.
  Since $\E_{\mathbf{ijkl}}=\E_{\mathbf{jikl}}$ for all $(\mathbf{i},\mathbf{j},\mathbf{k},\mathbf{l})\in\{1,...,n\}^4$, 
  it follows that via $X\ni \A\mapsto \E:\A$,
  $\E$ maps $X$ linearly into itself, and the additional assumption that $\E_{\mathbf{ijkl}}=\E_{\mathbf{klij}}$
  for all $(\mathbf{i},\mathbf{j},\mathbf{k},\mathbf{l})\in\{1,...,n\}^4$ asserts that this map actually is self-adjoint.
  According to (\ref{32.1}), $\E$ moreover is positive definite, whence if in a standard manner we let
  $\lam_1,...,\lam_d$ denote its positive eigenvalues and $\A_1,...,\A_d$ make up a corresponding orthonormal basis, and set
  $\sqrt{\E}:\A:=\sum_{\iota=1}^d a_\iota \sqrt{\lam_\iota} \A_\iota$ for $\A=\sum_{\iota=1}^d a_\iota \A_\iota$ with
  $a_\iota=\lan \A,\A_\iota\ran$, $\iota\in\{1,...,d\}$, then we obtain a linear $\sqrt{\E}:X\to X$
  which satisfies (\ref{32.2})-(\ref{32.4}).\abs
  ii) \ We may assume without loss of generality that with some $c_1>0$ we have
  \bas
	\int_0^T \io h_j \lan\E:\nas\vp_j,\nas\vp_j\ran \le c_1
	\qquad \mbox{for all } j\in\N,
  \eas
  and we then let $z_j:=\sqrt{h_j} \sqrt{\E}:\nas\vp_j$, $j\in\N$, where $\sqrt{\E}$ is as in i).
  Then in view of (\ref{32.3}) and (\ref{32.2}),
  \bas
	\int_0^T \io |z_j|^2
	= \int_0^T \io h_j \lan \sqrt{\E}:\nas\vp_j,\sqrt{\E}:\nas\vp_j\ran
	= \int_0^T \io h_j \lan\E:\nas\vp_j,\nas\vp_j\ran
	\le c_1
	\qquad \mbox{for all } j\in\N,
  \eas
  whence we can find $z\in L^2(\Om\times (0,T);\R^{n\times n})$ and a sequence $(j_i)_{i\in\N}\subset \N$ such that
  $j_i\to \infty$ as $i\to\infty$ and
  $z_{j_i}\wto z$ in $L^2(\Om\times (0,T);\R^{n\times n})$ as $i\to\infty$, which by a lower semicontinuity feature
  of $\|\cdot\|_{L^2(\Om\times (0,T);\R^{n\times n})}$ implies that
  \be{32.9}
	\int_0^T \io |z|^2 \le \liminf_{i\to\infty} \int_0^T \io |z_{j_i}|^2.
  \ee
  But according to the Lions lemma (\cite[Lemme 1.3]{lions}), from the fact that as $i\to\infty$ we have $\sqrt{h_{j_i}} \to \sqrt{h}$
  a.e.~in $\Om\times (0,T)$ and $\sqrt{\E}:\nas\vp_{j_i} \wto \sqrt{\E}:\nas\vp$ in $L^1(\Om\times (0,T);\R^{n\times n})$,
  as clearly entailed by (\ref{32.5}) and (\ref{32.6}), it follows that $z$ must coincide with $\sqrt{h}\sqrt{\E}:\nas\vp$
  a.e.~in $\Om\times (0,T)$, so that (\ref{32.8}) results from (\ref{32.9}) and, again, (\ref{32.2}) and (\ref{32.3}).
\qed
The convergence properties asserted by Lemma \ref{lem24} and Lemma \ref{lem244} can thus be seen to entail (\ref{wt}):
\begin{lem}\label{lem33}
  Suppose that either (\ref{131.1}) or (\ref{38.1}) holds.		
  Then for each $\phi\in C^\infty([0,\infty))$ and any $\vp\in C_0^\infty(\bom\times [0,\infty))$ 
  fulfilling $\phi'\in C_0^\infty([0,\infty))$, $\phi'\ge 0$, $\phi''\le 0$ and $\vp\ge 0$,
  the inequality in (\ref{wt}) holds.
\end{lem}
\proof
  For $\eps\in (0,1)$, we let $\Keps^{(\phi)}(\xi):=\int_0^\xi \keps(\sig)\phi'(\sig)d\sig$, $\xi\ge 0$, and observe that if we
  let $\xi_0>0$ be such that $\phi'\equiv 0$ on $[\xi_0,\infty)$, then writing $c_1:=\|\kappa\|_{L^\infty((0,\xi_0))}+1$ 
  and $c_2:=\|\phi'\|_{L^\infty((0,\infty))}$ we have
  \be{33.1}
	0 \le \Keps^{(\phi)}(\xi) 
	\le \int_0^\xi (\kappa(\sig)+1) \phi'(\sig) d\sig
	\le c_1 c_2 \xi_0
	\qquad \mbox{for all $\xi\ge 0$ and } \eps\in (0,1)
  \ee
  as well as
  \be{33.2}
	\big|\Keps^{(\phi)}(\xi)-K^{(\phi)}(\xi)\big|
	= \bigg| \int_0^\xi (\keps(\sig)-\kappa(\sig)) \phi'(\sig) d\sig \bigg|
	\le c_2 \xi_0 \eps
	\qquad \mbox{for all $\xi\ge 0$ and } \eps\in (0,1),
  \ee
  because $|\keps-\kappa| \le \eps \le 1$ on $[0,\infty)$ for all $\eps\in (0,1)$ by (\ref{keps}).
  In the identity
  \bea{33.3}
	& & \hs{-30mm}
	- D \int_0^\infty \io \phi''(\Teps) |\na\Teps|^2 \vp
	+ \int_0^\infty \io \phi'(\Teps) \lan \D:\nas\veps,\nas\veps\ran \vp
	+ \int_0^\infty \io \phi'(\Teps) \geps\vp \nn\\
	&=& - \int_0^\infty \io \Keps^{(\phi)}(\Teps) \vp_t
	- \io \Keps^{(\phi)}(\Theta_{0\eps})\vp(\cdot,0)  \nn\\
	& & + D \int_0^\infty \io \phi'(\Teps) \na\Teps\cdot\na\vp
	+ \int_0^\infty \io \Teps \phi'(\Teps) \lan\B,\nas\veps\ran \vp,
  \eea
  valid for each $\eps\in (0,1)$ according to (\ref{0eps}), we may thus use 
  that in both cases to be addressed, along the sequences $(\eps_j)_{j\in\N}$ respectively provided by Lemma \ref{lem24} and
  Lemma \ref{lem244} we have 
  \be{33.4}
	\Teps\to\Theta
	\quad \mbox{a.e.~in } \Om\times (0,\infty)
	\qquad \mbox{as } \eps=\eps_j\searrow 0.
  \ee
  Since thus $\Keps^{(\phi)}(\Teps) \to K^{(\phi)}(\Theta)$ a.e.~in $\Om\times (0,\infty)$ as $\eps=\eps_j\searrow 0$ 
  by (\ref{33.2}), namely, thanks to (\ref{33.1}) and the boundedness of $\supp\vp_t$ the dominated convergence theorem 
  asserts that
  \be{33.5}
	- \int_0^\infty \io \Keps^{(\phi)}(\Teps) \vp_t
	\to - \int_0^\infty \io K^{(\phi)}(\Theta) \vp_t
	\qquad \mbox{as } \eps=\eps_j\searrow 0,
  \ee
  while from (\ref{iec2}) and (\ref{33.2}) it similarly follows that
  \be{33.6}
	- \io \Keps^{(\phi)}(\Theta_{0\eps}) \vp(\cdot,0)
	\to - \io K^{(\phi)}(\Theta_0) \vp(\cdot,0)
	\qquad \mbox{as } \eps=\eps_j\searrow 0.
  \ee
  Next, an application of (\ref{24.8}) and (\ref{244.8}) to $\phi_0=\phi$ shows that in line with our notational convention
  from the remark following Definition \ref{dw},
  \be{33.7}
	D \int_0^\infty \io \phi'(\Teps) \na\Teps\cdot\na\vp
	= D \int_0^\infty \io \na \phi(\Teps) \cdot\na\vp
	\to D \int_0^\infty \io \na\phi(\Theta)\cdot\na\vp
	= D \int_0^\infty \io \phi'(\Theta) \na\Theta\cdot\na\vp
  \ee
  as $\eps=\eps_j\searrow 0$, and in treating the rightmost integral in (\ref{33.3}) we again rely on the boundedness of $\supp\phi'$
  when noting that $0\le\xi\mapsto \xi\phi'(\xi)$ is bounded.
  Thanks to Lemma \ref{lem23}, this implies uniform integrability of $(\Teps\phi'(\Teps)\lan\B,\nas\veps\ran)_{\eps\in (0,1)}$
  over $\Om\times (0,T)$ for all $T>0$, so that combining (\ref{33.4}) with 
  the fact that 
  \be{33.77}
	\nas\veps\wto \nas u_t
	\quad \mbox{in } L^1_{loc}(\bom\times [0,\infty);\R^{n\times n})
	\qquad \mbox{ as $\eps=\eps_j\searrow 0$}
  \ee
  by Lemma \ref{lem24} and Lemma \ref{lem244}, due to the Lions lemma we infer that
  \be{33.8}
	\int_0^\infty \io \Teps \phi'(\Teps) \lan\B,\nas\veps\ran \vp
	\to \int_0^\infty \io \Theta \phi'(\Theta) \lan\B,\nas u_t\ran \vp
	\qquad \mbox{as } \eps=\eps_j\searrow 0.
  \ee
  We now employ Lemma \ref{lem32} ii) to see that since $0\le \phi'(\Teps)\vp\to\phi'(\Theta)\vp$ a.e.~in $\Om\times (0,\infty)$
  as $\eps=\eps_j\searrow 0$ by (\ref{33.4}),
  \be{33.9}
	\int_0^\infty \io \phi'(\Theta) \lan \D:\nas u_t,\nas u_t \ran \vp
	\le \liminf_{\eps=\eps_j\searrow 0} \int_0^\infty \io \phi'(\Teps) \lan \D:\nas\veps,\nas\veps\ran \vp.
  \ee
  Furthermore, by nonnegativity of $-\phi''$ and $\vp$, and by lower semicontinuity of $L^2$ norms with respect to weak convergence,
  an application of (\ref{24.8}) and (\ref{244.8}) with $\phi_0(\xi)=\int_0^\xi \sqrt{-\phi''(\sig)} d\sig$, $\xi\ge 0$,
  shows that again interpreting as in the remark following Definition \ref{dw} we have 
  \bea{33.10}
	- D \int_0^\infty \io \phi''(\Theta) |\na\Theta|^2 \vp
	&=& D \int_0^\infty \io |\na\phi_0(\Theta)|^2 \vp \nn\\
	&\le& D \cdot \liminf_{\eps=\eps_j\searrow 0} \int_0^\infty \io |\na\phi_0(\Teps)|^2 \vp \nn\\
	&=& \liminf_{\eps=\eps_j\searrow 0} \bigg\{ 
	- D \int_0^\infty \io \phi''(\Teps) |\na\Teps|^2 \vp \bigg\}.
  \eea
  Since (\ref{33.4}) together with (\ref{fgec}) and the nonnegativity of $\phi'$, $\vp$ and the $\geps$ clearly implies that 
  \bas
	\int_0^\infty \io \phi'(\Theta) g \vp 
	\le \liminf_{\eps=\eps_j\searrow 0} \int_0^\infty \io \phi'(\Teps) \geps\vp,
  \eas
  inserting (\ref{33.5})-(\ref{33.10}) into (\ref{33.3}) and rearranging yields (\ref{wt}).
\qed
\subsection{The energy inequality (\ref{wF})}		
Next, the weak energy inequality (\ref{wF}) will be established on the basis of the following simple consequence of Lemma \ref{lem3}.
\begin{lem}\label{lem34}
  If $\zeta\in C_0^\infty([0,\infty))$ is such that $\zeta\ge 0$,
  then
  \bea{34.1}
	& & \hs{-30mm}
	- \frac{1}{2} \int_0^\infty \io \zeta'(t) |\veps|^2
	- \frac{1}{2} \int_0^\infty \io \zeta'(t) \lan\C:\nas\ueps,\nas\ueps\ran
	- \int_0^\infty \io \zeta'(t) \Keps(\Teps) \nn\\
	&\le& \zeta(0) \io \Big\{ \frac{1}{2} |v_{0\eps}|^2 + \frac{1}{2} \lan\C:\nas u_{0\eps},\nas u_{0\eps}\ran 
		+ \Keps(\Theta_{0\eps})	\Big\} \nn\\
	& & + \int_0^\infty \io \zeta(t) \feps\cdot\veps
	+ \int_0^\infty \io \zeta(t) \geps
	\qquad \mbox{for all } \eps\in (0,1).
  \eea
\end{lem}
\proof
  This immediately follows after multiplying (\ref{3.1}) by $\zeta\ge 0$ and integrating by parts over $(0,\infty)$, because
  \bas
	\eps \int_0^\infty \io \zeta(t) |\Del^m \veps|^2 \ge 0
	\qquad \mbox{for all } \eps\in (0,1)
  \eas
  according to the fact that $\zeta\ge 0$.
\qed
Now under the assumptions of Theorem \ref{theo37}, the approximation properties recorded in Lemma \ref{lem24} 
apply so as to derive (\ref{wF}) from this.
\begin{lem}\label{lem35}
  If (\ref{131.1}) is satisfied and $u,\Theta,\mu$ and $a$ are as in Lemma \ref{lem24}, then (\ref{wF}) holds.
\end{lem}
\proof
  Let us first make sure that whenever $\zeta\in C_0^\infty([0,\infty))$ is nonnegative and nonincreasing, we have
  \be{35.4}
	- \int_0^\infty \zeta'(t)\cdot\bigg\{ \F(t) + a\int_{\bom} d\mu(t)\bigg\} dt
	\le \zeta(0) \F_0 + \int_0^\infty \io \zeta f\cdot u_t + \int_0^\infty \io \zeta g.
  \ee
  To this end, we note that for any such $\zeta$, due to the nonnegativity of $-\zeta'$ the weak $L^2$-approximation properties in
  (\ref{24.4}) and (\ref{24.5}) are sufficient to ensure that on the left-hand side of (\ref{34.1}),
  \be{35.5}
	\liminf_{\eps=\eps_j\searrow 0} \bigg\{ - \frac{1}{2} \int_0^\infty \io \zeta' |\veps|^2 \bigg\}
	\ge - \frac{1}{2} \int_0^\infty \io \zeta' |u_t|^2
  \ee
  and
  \be{35.6}
	\liminf_{\eps=\eps_j\searrow 0} \bigg\{ - \frac{1}{2} \int_0^\infty \io \zeta' \lan\C:\nas\ueps,\nas\ueps\ran\bigg\}
	\ge - \frac{1}{2} \int_0^\infty \io \zeta' \lan\C:\nas u,\nas u\ran,
  \ee
  because (\ref{24.5}) entails that with $\sqrt{\C}$ taken from Lemma \ref{lem32} we have
  $\sqrt{-\zeta'} \sqrt{\C}:\nas\ueps \wto \sqrt{-\zeta'} \sqrt{\C}:\nas u$ in $L^2(\Om\times (0,\infty);\R^{n\times n})$
  as $\eps=\eps_j\searrow 0$.
  Moreover, (\ref{24.4}) together with (\ref{fgec}) and the boundedness of $\supp\zeta'$ warrants that on the right of (\ref{34.1}),
  \be{35.7}
	\int_0^\infty \io \zeta \feps\cdot\veps
	\to \int_0^\infty \io \zeta f\cdot u_t
	\quad \mbox{and} \quad
	\int_0^\infty \io \zeta\geps \to \int_0^\infty \io \zeta g
	\qquad \mbox{as $\eps=\eps_j\searrow 0$},
  \ee
  while by (\ref{iec}),
  \be{35.8}
	\frac{1}{2} \io |v_{0\eps}|^2
	\to \frac{1}{2} \io |u_{0t}|^2
	\quad \mbox{and} \quad
	\frac{1}{2} \io \lan\C:\nas u_{0\eps},\nas u_{0\eps}\ran
	\to \frac{1}{2} \io \lan\C:\nas u_0,\nas u_0\ran
	\qquad \mbox{as } \eps=\eps_j\searrow 0,
  \ee
  because
  \bas
	\bigg| \io \Keps(\Theta_{0\eps}) - \io K(\Theta_{0\eps}) \bigg| \le |\Om| \eps
	\qquad \mbox{for all } \eps\in (0,1)
  \eas
  due to Lemma \ref{lem_Keps}.\abs
  In the third summand on the left of (\ref{34.1}), we first rely on Lemma \ref{lem131} and for $i\in\{1,2\}$ take 
  $(\Keps^{(i)})_{\eps\in (0,1)}$ and $K^{(i)}$ as found there in rewriting
  \be{35.10}
	- \int_0^\infty \io \zeta' \Keps(\Teps)
	= - a \int_0^\infty \io \zeta' \Teps
	- \int_0^\infty \io \zeta' \Keps^{(1)}(\Teps)
	+ \int_0^\infty \io \zeta' \Keps^{(2)}(\Teps)
	\qquad \mbox{for } \eps\in (0,1),
  \ee
  and observe that here we may use (\ref{24.7}) to confirm that
  \be{35.11}
	- a \int_0^\infty \io \zeta' \Teps
	\to - a \int_0^\infty \io \zeta' \Theta
	- a \int_0^\infty \int_{\bom} \zeta' d\mu
	\qquad \mbox{as } \eps=\eps_j\searrow 0.
  \ee
  Furthermore, the uniform boundedness property of the $\Keps^{(2)}$ expressed in (\ref{131.3}) ensures that thanks to (\ref{24.6}),
  (\ref{131.5}) and the dominated convergence theorem,
  \be{35.12}
	\int_0^\infty \io \zeta' \Keps^{(2)}(\Teps)
	\to \int_0^\infty \io \zeta' K^{(2)}(\Theta)
	\qquad \mbox{as } \eps=\eps_j\searrow 0,
  \ee
  whereas by nonnegativity of $-\zeta'$ and of $\Keps^{(1)}$ for $\eps\in (0,1)$, a combination of (\ref{24.6}) with (\ref{131.5})
  and Fatou's lemma shows that
  \be{35.13}
	\liminf_{\eps=\eps_j\searrow 0} \bigg\{ - \int_0^\infty \io \zeta' \Keps^{(1)}(\Teps) \bigg\}
	\ge - \int_0^\infty \io \zeta' K^{(1)}(\Theta).
  \ee
  Since (\ref{131.2}) and (\ref{131.5}) clearly imply that $K(\xi)=a\xi+K^{(1)}(\xi)-K^{(2)}(\xi)$ for all $\xi\ge 0$,
  from (\ref{35.10})-(\ref{35.13}) we thus infer that
  \be{35.14}
	\liminf_{\eps=\eps_j\searrow 0} \bigg\{ - \int_0^\infty \io \zeta' \Keps(\Teps)\bigg\}
	\ge - \int_0^\infty \io \zeta' K(\Theta)
	- a \int_0^\infty \int_{\bom} \zeta' d\mu,
  \ee
  which along with (\ref{35.5})-(\ref{35.8}) warrants that in line with (\ref{F}) and (\ref{F0}), the inequality in (\ref{34.1})
  indeed entails (\ref{35.4}).\abs
  To derive (\ref{wF}) from this, finally, writing $h(t):=\F(t)+ a\int_{\bom} d\mu(t)$ for $t>0$ we observe that the regularity
  features in (\ref{24.1}) particularly guarantee that $h\in L^1_{loc}([0,\infty))$, whence we can find a null set 
  $N\subset (0,\infty)$ such that each $t_0\in (0,\infty)\sm N$ is a Lebesgue point of $h$.
  Fixing any such $t_0$, for $\del\in (0,1)$ we let
  \bas
	\zd(t):=
	\lball
	1, \qquad & t\in [0,t_0], \\[1mm]
	1-\frac{t-t_0}{\del},
	\qquad & t\in (t_0,t_0+\del), \\[1mm]
	0, \qquad & t\ge t_0+\del,
	\ear
  \eas
  and use a straightforward mollification procedure to find $(\zeta_{\delta k})_{k\in\N} \subset C_0^\infty([0,\infty))$
  such that $\zeta_{\del k}(0)=1, \zeta_{\del k}' \le 0$ on $[0,\infty)$ and $\zeta_{\del k} \equiv 0$ on $[t_0+2\del,\infty)$
  for all $k\in\N$, and that $\zeta_{\del k} \wsto \zd$ in $W^{1,\infty}((0,\infty))$ as $k\to\infty$.
  An application of (\ref{35.4}) to $\zeta=\zeta_{\del k}$ for $k\in\N$, followed by taking $k\to\infty$, therefore yields the
  inequality
  \bas
	\frac{1}{\del} \int_{t_0}^{t_0+\del} h(t) dt
	\le \F_0 + \int_0^{t_0+\del} \io \zd f\cdot u_t + \int_0^{t_0+\del} \zd g
	\qquad \mbox{for all } \del\in (0,1),
  \eas
  which in view of the Lebesgue point property of $t_0$ implies that
  \bas
	h(t_0) \le \F_0 + \int_0^{t_0} f\cdot u_t + \int_0^{t_0} g.
  \eas
  As $t_0\in (0,\infty)\sm N$ was arbitrary, in line with our definition of $h$ this establishes (\ref{wF}).
\qed
Under the assumption in (\ref{38.1}), a simplification of this reasoning yields the stronger conclusion 
compatible with the respective claim in Theorem \ref{theo38}.
\begin{lem}\label{lem36}
  Assume (\ref{38.1}),	
  and let $u$ and $\Theta$ be as found in Lemma \ref{lem244}. 
  Then (\ref{wF}) is satisfied with $a=0$ and $\mu=0$.
\end{lem}
\proof
  This can be seen by repeating the argument from Lemma \ref{lem35}, relying on the convergence properties
  asserted by Lemma \ref{lem244} now instead of those from Lemma \ref{lem24};
  actually, the reasoning can even be simplified here, as no dissipation of any expression in the style of the second
  summand on the left of (\ref{wF}) is striven for, meaning that the corresponding analogue of (\ref{35.14}) can be derived
  directly from (\ref{244.6}) through Fatou's lemma.
\qed
\mysection{Entropy and positivity. Proofs of Theorems \ref{theo37} and \ref{theo38}}
Only now we turn our attention to a rigorous counterpart of the entropy property in (\ref{ent}).
The following observation in this regard will not only be used to complete our existence theory by asserting 
the claims concerning positivity made in Theorem \ref{theo37} and Theorem \ref{theo38}, but it will also form
a starting point for our results on large time behavior.
\begin{lem}\label{lem11}
  For $\eps\in (0,1)$, let
  \be{leps}
	\leps(\xi):=\int_1^\xi \frac{\keps(\sig)}{\sig} d\sig,
	\qquad \xi>0.
  \ee
  Then for all $t_0\ge 0$ and $t>t_0$,
  \bea{11.6}
	\io \leps(\Teps(\cdot,t))
	\ge \io \leps(\Teps(\cdot,t_0)) 
	+ D \int_{t_0}^t \io \frac{|\na\Teps|^2}{\Teps^2}
	+ \kD \int_{t_0}^t \io \frac{|\nas\veps|^2}{\Teps}
	+ \int_{t_0}^t \io \frac{\geps}{\Teps}.
  \eea
  In particular,
  \be{11.1}
	\io \leps(\Teps(\cdot,t)) \ge \io\leps(\Teps(\cdot,t_0))
	\qquad \mbox{for all $t_0\ge 0, t>t_0$ and } \eps\in (0,1),
  \ee
  and there exists $(C(T))_{T>0}\subset (0,\infty)$ such that
  \be{11.2}
	\int_0^T \io \frac{|\na\Teps|^2}{\Teps^2} \le C(T)
	\qquad \mbox{for all $T>0$ and } \eps\in (0,1)
  \ee
  and
  \be{11.3}
	\int_0^T \io \frac{|\nas\veps|^2}{\Teps} \le C(T)
	\qquad \mbox{for all $T>0$ and } \eps\in (0,1),
  \ee
  and that
  \be{11.4}
	\sup_{T>0} C(T)<\infty
	\quad \mbox{if (\ref{fg_decay}) holds.}
  \ee
\end{lem}
\proof
  For $\eps\in (0,1)$, relying on the fact that $\leps'(\xi)=\frac{\keps(\xi)}{\xi}$ for $\xi>0$ and that $\Teps$ is positive
  in $\bom\times [0,\tme)$ by Lemma \ref{lem1}, we may integrate by parts using (\ref{0eps}) to see that
  \bea{11.5}
	\hs{-8mm}
	\frac{d}{dt} \io \leps(\Teps)
	&=& \io \frac{\keps(\Teps)}{\Teps} \Theta_{\eps t} \nn\\
	&=& \io \frac{1}{\Teps} \cdot \Big\{ D\Del\Teps + \lan \D:\nas\veps,\nas\veps\ran
	- \Teps \lan\B,\nas\veps\ran + \geps \Big\} \nn\\
	&=& D \io \frac{|\na\Teps|^2}{\Teps^2}
	+ \io \frac{\lan\D:\nas\veps,\nas\veps\ran}{\Teps}
	- \io \lan\B,\nas\veps\ran 
	+ \io \frac{\geps}{\Teps}
	\qquad \mbox{for all } t>0.
  \eea
  Here, another integration by parts shows that since $\veps=0$ on $\pO\times (0,\infty)$,
  \bas
	- \io \lan\B,\nas\veps\ran =0
	\qquad \mbox{for all } t>0,
  \eas
  while using (\ref{DC_lower}) we find that
  \bas
	\io \frac{\lan\D:\nas\veps,\nas\veps\ran}{\Teps}
	\ge \kD \io \frac{|\nas\veps|^2}{\Teps}
	\qquad \mbox{for all } t>0,
  \eas
  whence integrating (\ref{11.5}) in time yields (\ref{11.6}) and, consequently, also (\ref{11.1}).
  Moreover, from (\ref{leps}) and (\ref{Keps}) it follows that
  \bas
	0 \le \leps(\xi) \le \int_1^\xi \keps(\sig) d\sig \le \Keps(\xi)
	\qquad \mbox{for all } \xi\ge 1
  \eas
  and $\leps(\xi)\le 0$ for all $\xi\in (0,1)$, so that
  \bas
	\io \leps(\Teps(\cdot,t)) \le \io \Keps(\Teps(\cdot,t))
	\qquad \mbox{for all } t>0.
  \eas
  Apart from that, (\ref{keps}) ensures that
  \bas
	- \leps(\xi) = \int_\xi^1 \frac{\keps(\sig)}{\sig} d\sig
	\le \int_\xi^1 \frac{\kappa(\sig)+\eps}{\sig} d\sig
	= - \ell(\xi) + \eps\ln\frac{1}{\xi}
	\qquad \mbox{for all $\xi\in (0,1)$ and } \eps\in (0,1),
  \eas
  so that since $-\leps(\xi)\le 0$ for all $\xi\ge 1$ and $\eps\in (0,1)$, and since 
  Lemma \ref{lem99} particularly entails the existence of $c_1>0$ such that 
  $-\int_{\{\Theta_{0\eps}<1\}} \ell(\Theta_{0\eps}) \le c_1$ for all $\eps\in (0,1)$,
  recalling that $\Theta_{0\eps} \ge \eps$ in $\Om$ for all $\eps\in (0,1)$ we obtain that
  \bas
	- \io \leps(\Theta_{0\eps}) 
	\le - \int_{\{\Theta_{0\eps}<1\}} \ell(\Theta_{0\eps})
	+ \eps \int_{\{\Theta_{0\eps}<1\}} \ln \frac{1}{\Theta_{0\eps}}
	\le c_2:= c_1 + \frac{|\Om|}{e}
	\qquad \mbox{for all } \eps\in (0,1),
  \eas
  because $\xi\ln\frac{1}{\xi} \le \frac{1}{e}$ for all $\xi>0$.
  An application of (\ref{11.6}) to $t_0:=0$ therefore shows that
  \bas
	D \int_0^t \io \frac{|\na\Teps|^2}{\Teps^2}
	+ \kD \int_0^t \io \frac{|\nas\veps|^2}{\Teps}
	+ \int_0^t \io \frac{\geps}{\Teps}
	\le c_2 + \io \Keps(\Teps(\cdot,t))
	\qquad \mbox{for all } t>0,
  \eas
  whence the existence of $(C(T))_{T>0}\subset (0,\infty)$ fulfilling (\ref{11.2})-(\ref{11.4}) becomes a consequence
  of Lemma \ref{lem6}.
\qed
In order to derive an estimate for $\ln\Teps$ from (\ref{11.2}) through a Poincar\'e type inequality,
we first use (\ref{11.1}) to infer the following quantitative positivity property of the $\Teps$ 
from our overall assumption on finiteness of $\io \ell(\Theta_0)$.
\begin{lem}\label{lem18}
  Suppose that $n\ge 2$ and $\int_0^\infty \kappa(\sig)d\sig = \infty$.
  Then there exist $b>0$, $\epss\in (0,1)$ and $(C(T))_{T>0} \subset (0,\infty)$ such that whenever $T>0$,
  \be{18.2}
	\Big| \big\{ \Teps(\cdot,t) > b \big\} \Big|
	\ge \frac{1}{C(T)}
	\qquad \mbox{for all $t\in (0,T)$ and } \eps\in (0,\epss),
  \ee
  and that
  \be{18.3}
	\sup_{T>0} C(T)<\infty
	\quad \mbox{if (\ref{fg_decay}) holds.}
  \ee
\end{lem}
\proof
  We first write $\wh{\ell}^{(b)}(\xi):=\int_b^\xi \frac{\kappa(\sig)}{\sig} d\sig \in \R\cup \{+\infty\}$ for $\xi>0$
  and $b\ge 0$, and invoke Beppo Levi's theorem to see that $\wh{\ell}^{(b)}(\xi) \nearrow \wh{\ell}^{(0)}(\xi)$ for all $ \xi>0$ and,
  thus, 
  $\io \wh{\ell}^{(b)}(\Theta_0) \nearrow \io \wh{\ell}^{(0)}(\Theta_0) \in (0,\infty]$ as $b\searrow 0$, whence we can fix
  $b\in (0,1)$ such that with $\wh{\ell}:=\wh{\ell}^{(b)}$, the number $c_1:=\io \wh{\ell}(\Theta_0)$ is positive.\abs
  We next let
  \be{18.6}
	\lleps(\xi):=\int_b^\xi \frac{\keps(\sig)}{\sig} d\sig,
	\qquad \xi>0, \ \eps\in (0,1),
  \ee
  and note that according to (\ref{keps}), for $\eps\in (0,1)$ and $\xi\in (0,b)$ we have
  \bas
	|\lleps(\xi)| = \int_\xi^b \frac{\keps(\sig)}{\sig} d\sig
	\le \int_\xi^b \frac{\kappa(\sig)+\eps}{\sig} d\sig
	= |\wh{\ell}(\xi)| + \eps\ln\frac{b}{\xi},
  \eas
  while if $\eps\in (0,1)$ and $\xi\ge b$, then
  \bas
	|\lleps(\xi)| \le \frac{1}{b} \int_b^\xi \Big\{ \kappa(\sig) + \frac{1}{(\sig+1)^2} \Big\} d\sig
	\le \frac{1}{b} K(\xi) + \frac{1}{b}.
  \eas
  Since $\Theta_{0\eps} \ge \eps$ in $\Om$ for all $\eps\in (0,1)$, we therefore see that
  \bas
	|\lleps(\Theta_{0\eps})|
	\le h_\eps:=|\wh{\ell}(\Theta_{0\eps})| + \eps\ln\frac{b}{\eps} + \frac{1}{b} K(\Theta_{0\eps}) + \frac{1}{b}
	\quad \mbox{in $\Om$ \qquad for all } \eps\in (0,1),
  \eas
  and we observe that here $h_\eps\to |\wh{\ell}(\Theta_0)| + \frac{1}{b}K(\Theta_0) + \frac{1}{b}$ 
  in $L^1(\Om)$ as $\eps\searrow 0$ by Lemma \ref{lem99}, because $\wh{\ell}-\ell\equiv \int_b^1 \frac{\kappa(\sig)}{\sig}d\sig$
  on $(0,\infty)$.
  Since, for the same reason, Lemma \ref{lem99} can moreover readily be seen to assert that
  $\lleps(\Theta_{0\eps}) \to \wh{\ell}(\Theta_0)$
  a.e.~in $\Om$ as $\eps\searrow 0$, we may thus employ the dominated convergence theorem to infer that
  $\io \lleps(\Theta_{0\eps}) \to \io \wh{\ell}(\Theta_0)$ as $\eps\searrow 0$, so that by positivity of $c_1$ we can pick
  $\epss\in (0,1)$ such that
  \be{18.92}
	\io \lleps(\Theta_{0\eps}) \ge \frac{c_1}{2}
	\qquad \mbox{for all } \eps\in (0,\epss).
  \ee
  To make appropriate use of this, we apply Lemma \ref{lem6} in choosing $(c_2(T))_{T>0} \subset (0,\infty)$ such that if $T>0$, then
  \be{18.4}
	\io \Keps(\Teps) \le c_2(T)
	\qquad \mbox{for all $t\in (0,T)$ and } \eps\in (0,1),
  \ee
  and that
  \be{18.5}
	\sup_{T>0} c_2(T)<\infty
	\quad \mbox{if (\ref{fg_decay}) holds,}
  \ee
  whereas from Lemma \ref{lem11} we obtain that
  \be{18.7}
	\io \lleps(\Teps(\cdot,t)) \ge \io \lleps(\cdot,\Theta_{0\eps})
	\qquad \mbox{for all $t>0$ and } \eps\in (0,1),
  \ee
  because with $(\leps)_{\eps\in (0,1)}$ as in (\ref{leps}), we clearly have 
  $\lleps-\leps\equiv \int_b^1 \frac{\keps(\sig)}{\sig}d\sig$ on $(0,\infty)$ for all $\eps\in (0,1)$.
  Since $\lleps(\xi) \le 0$ for all $\xi\le b$ and $\eps\in (0,1)$, from (\ref{18.92}) we thus infer that whenever $a>b$,
  \bea{18.9}
	\frac{c_1}{2} \le \io \lleps(\Teps)
	&=& \int_{\{\Teps\le b\}} \lleps(\Teps)
	+ \int_{\{b<\Teps<a\}} \lleps(\Teps)
	+ \int_{\{\Teps\ge a\}} \lleps(\Teps) \nn\\
	&\le& \int_{\{b<\Teps<a\}} \lleps(\Teps)
	+ \int_{\{\Teps\ge a\}} \lleps(\Teps) \nn\\
	&\le& \lleps(a) \cdot \big| \{ \Teps>b\} \big|
	+ \int_{\{\Teps\ge a\}} \lleps(\Teps)
	\qquad \mbox{for all $t>0$ and } \eps\in (0,\epss).
  \eea
  In order to suitably select $a$ here, given $T>0$ we let
  \be{18.10}
	\eta(T):=\frac{c_1}{4c_2(T)},
  \ee
  and firstly choose $\xi_0(T)>b$ such that
  \be{18.11}
	\frac{1}{\xi_0(T)} \le \frac{\eta(T)}{2}.
  \ee
  Then abbreviating
  \be{18.12}
	c_3(T):=\int_b^{\xi_0(T)} \frac{\kappa(\sig)}{\sig} d\sig
	+ \int_b^\infty \frac{d\sig}{\sig(\sig+1)^2},
  \ee
  due to our assumption on divergence of $\int_0^\infty \kappa$ 
  we can fix $a(T)>\xi_0(T)$ in such a way that the function in (\ref{K}) satisfies
  \be{18.13}
	K(\xi)\ge 2
	\quad \mbox{and} \quad
	K(\xi)\ge \frac{4c_3(T)}{\eta(T)}
	\qquad \mbox{for all } \xi\ge a(T).
  \ee
  Since
  \bas
	\lleps(\xi_0(T))
	= \int_b^{\xi_0(T)} \frac{\keps(\sig)}{\sig} d\sig
	\le \int_b^{\xi_0(T)} \frac{\kappa(\sig)}{\sig} d\sig
	+ \eps \int_b^{\xi_0(T)} \frac{d\sig}{\sig(\sig+1)^2} 
	\le c_3(T)
	\qquad \mbox{for all } \eps\in (0,1)
  \eas
  by (\ref{18.12}), since clearly
  \bas
	\lleps(\xi)-\lleps(\xi_0(T))
	= \int_{\xi_0(T)}^\xi \frac{\keps(\sig)}{\sig} d\sig
	\le \frac{1}{\xi_0(T)} \int_{\xi_0(T)}^\xi \keps(\sig) d\sig
	\le \frac{1}{\xi_0(T)} \int_0^\xi \keps(\sig) d\sig
	= \frac{1}{\xi_0(T)} \Keps(\xi)
  \eas
  for all $\xi>\xi_0(T)$ and $\eps\in (0,1)$, and since
  \bas
	\Keps(\xi)
	\ge \int_0^\xi \kappa(\sig)d\sig - \eps \int_0^\xi \frac{d\sig}{(\sig+1)^2} 
	\ge K(\xi) - \eps
	\ge K(\xi)-1 
	\ge \frac{1}{2} K(\xi)
	\quad \mbox{for all $\xi\ge a(T)$ and } \eps\in (0,1)
  \eas
  according to the first inequality in (\ref{18.13}), we can now combine the second restriction in (\ref{18.13}) with (\ref{18.11})
  to see that if $T>0$, then
  for all $\xi\ge a(T)$ and $\eps\in (0,1)$,
  \bas
	\frac{\lleps(\xi)}{\Keps(\xi)}
	\le \frac{c_3(T) + \frac{1}{\xi_0(T)} \Keps(\xi)}{\Keps(\xi)}
	= \frac{c_3(T)}{\Keps(\xi)}
	+ \frac{1}{\xi_0(T)}
	\le \frac{2c_3(T)}{K(\xi)}
	+ \frac{1}{\xi_0(T)}
	\le \frac{\eta(T)}{2} + \frac{\eta(T)}{2}=\eta(T).
  \eas
  In view of (\ref{18.10}) and (\ref{18.4}), this means that for each $T>0$, on the right-hand side of (\ref{18.9}) we have
  \bas
	\int_{\{\Teps\ge a(T)\}} \lleps(\Teps)
	\le \eta(T) \int_{\{\Teps\ge a(T)\}} \Keps(\Teps)
	\le \eta(T) \io \Keps(\Teps)
	\le \eta(T) \cdot c_2(T)
	\le \frac{c_1}{4}
  \eas
  for all $t\in (0,T)$ and $\eps\in (0,\epss)$, whence from (\ref{18.9}) it follows that
  \bas
	\lleps(a(T)) \cdot \big| \{\Teps>b\}\big| \ge \frac{c_1}{4}
	\qquad \mbox{for all $t\in (0,T)$ and $\eps\in (0,\epss)$}.
  \eas
  Observing that
  \bas
	\lleps(a(T)) = \int_b^{a(T)} \frac{\keps(\sig)}{\sig} d\sig
	\le \int_0^{a(T)} \frac{\kappa(\sig)+1}{\sig} d\sig
	\qquad \mbox{for all $\eps\in (0,\epss)$},
  \eas
  according to (\ref{18.5}) we may thus conclude as intended.
\qed
The following conclusion will not only complete our proofs of Theorem \ref{theo37} and Theorem \ref{theo38},
but will also serve as a useful basis for our large time analysis by implying a uniform integrability of the
$\leps(\Teps)$ (see Lemma \ref{lem41}).
\begin{lem}\label{lem19}
  If $n\ge 2$ and (\ref{16.1}) holds, then 
  there exists $(C(T))_{T>0} \subset (0,\infty)$ such that 
  \be{19.1}
	\int_t^{t+1} \io \ln^2(\Teps) \le C(T)
	\qquad \mbox{for all $t\in (0,T)$, $T>0$ and } \eps\in (0,\epss),
  \ee
  and that
  \be{19.2}
	\sup_{T>0} C(T)<\infty
	\quad \mbox{if (\ref{fg_decay}) holds,}
  \ee
  where $\epss$ is as in Lemma \ref{lem18}.
\end{lem}
\proof
  According to a Poincar\'e-type inequality (see \cite[Cor 9.1.4]{jost_pde} 
  for convex and, e.g., \cite[Lemma 9.1]{lankeit_win_NoDEA} for
  more general domains), for each $\del>0$ there exists $c_1(\del)>0$ with the property that
  \be{19.3}
	\io \vp^2 \le c_1(\del) \io |\na\vp|^2
	\qquad \mbox{for all $\vp\in W^{1,2}(\Om)$ fulfilling } \big| \{\vp=0\}\big| \ge \del,
  \ee
  while Lemma \ref{lem18}, Lemma \ref{lem6} and Lemma \ref{lem11} provide $c_2>0$ and 	
  $(c_i(T))_{T>0} \subset (0,\infty)$, $i\in \{3,4,5\}$, such that $\sup_{T>0} \{ c_3(T)+c_4(T)+c_5(T)\}<\infty$ if 
  (\ref{fg_decay}) holds, and that if $T>0$, then
  \be{19.4}
	\big| \{\Teps>c_2\}\big|
	\ge c_3(T)
	\qquad \mbox{for all $t\in (0,T+1)$ and $\eps\in (0,\epss)$}
  \ee
  as well as
  \be{19.5}
	\io \Keps(\Teps) \le c_4(T)
	\quad \mbox{and} \quad
	\int_t^{t+1} \io \frac{|\na\Teps|^2}{\Teps^2} \le c_5(T)
	\qquad \mbox{for all $t\in (0,T+1)$ and $\eps\in (0,\epss)$}.
  \ee
  We moreover note that if we let $\uK$ be as in Lemma \ref{lem16}, applicable here due to our assumption on $\kappa$,
  then since $\ln\xi\le \xi$ for all $\xi>0$, from (\ref{16.4}) and the positivity and continuity of $\uK$ we obtain $c_6>0$
  such that $\uK(\xi)\ge c_6\sqrt{\xi}$ for all $ \xi\ge c_2$, and that thus
  \be{19.6}
	\ln^2 \frac{\xi}{c_2} \le \frac{16}{e^2} \cdot \sqrt{\frac{\xi}{c_2}}
	\le c_7 \uK(\xi)
	\qquad \mbox{for all } \xi\ge c_2
  \ee
  with $c_7:=\frac{16}{e^2 \sqrt{c_2} c_6}$,
  because $\ln\xi\le\frac{4}{e} \cdot \sqrt[4]{\xi}$ for all $\xi\ge 1$.
  Finally choosing $c_8>0$ such that in line with Lemma \ref{lem16} we have
  \be{19.66}
	\uK \le c_8 \Keps + c_8
	\quad \mbox{on } [0,\infty)
	\qquad \mbox{for all } \eps\in (0,1),
  \ee
  for fixed $T>0$ we apply (\ref{19.3}) to $\vp:=\big\{ \ln \frac{c_2}{\Teps} \big\}_+$ for $\eps\in (0,1)$ to see that
  thanks to (\ref{19.5}),
  \bas
	\int_t^{t+1} \int_{\{\Teps<c_2\}} \ln^2 \frac{c_2}{\Teps}
	&\le& c_1(c_3(T)) \int_t^{t+1} \int_{\{\Teps<c_2\}} \Big|\na\ln\frac{c_2}{\Teps}\Big|^2 \\
	&\le& c_1(c_3(T)) \int_t^{t+1} \io \frac{|\na\Teps|^2}{\Teps^2} \\
	&\le& c_1(c_3(T)) \cdot c_5(T)
	\qquad \mbox{for all $t\in (0,T)$ and } \eps\in (0,\epss).
  \eas
  As using (\ref{19.6}) along with (\ref{19.66}) and (\ref{19.5}) we furthermore find that
  \bas
	\int_t^{t+1} \int_{\{\Teps\ge c_2\}} \ln^2 \frac{c_2}{\Teps}
	&=& \int_t^{t+1} \int_{\{\Teps\ge c_2\}} \ln^2 \frac{\Teps}{c_2} \\
	&\le& c_7 \int_t^{t+1} \io \uK(\Teps) \nn\\
	&\le& c_7 c_8 \int_t^{t+1} \io \Keps(\Teps)
	+ c_7 c_8 |\Om| \\
	&\le& c_7 c_8 c_4(T)
	+ c_7 c_8 |\Om| 
	\qquad \mbox{for all $t\in (0,T)$ and } \eps\in (0,1),
  \eas
  we therefore obtain (\ref{19.1}) with some $(C(T))_{T>0} \subset (0,\infty)$ fulfilling (\ref{19.2}).
\qed
Our goals concerning global solvability of (\ref{0}) have thus been accomplished:\abs
\proofc of Theorem \ref{theo37}. \quad
  The statement on global solvability follows from Lemma \ref{lem24}, Lemma \ref{lem33} and Lemma \ref{lem35}.
  Positivity of $\Theta$ a.e.~in $\Om\times (0,\infty)$ can be seen by observing that
  (\ref{24.6}) together with Lemma \ref{lem19} and Fatou's lemma particularly implies that 
  $\ln \Theta \in L^1_{loc}(\bom\times [0,\infty))$.
\qed
\proofc of Theorem \ref{theo38}. \quad
  The claim is a direct consequence of Lemma \ref{lem244}, Lemma \ref{lem33}, Lemma \ref{lem36} and, again, Lemma \ref{lem19}.
\qed

\mysection{Stabilization of $\Theta$. Proof of Theorem \ref{theo55}}\label{sect_theta}
\subsection{A uniform smallness property of the approximate entropy production}
A major challenge related to our proof of Theorem \ref{theo55} will consist in making appropriate use of Lemma \ref{lem11}
to derive expedient information in the limit of vanishing $\eps$.
In fact, assuming (\ref{fg_decay}) one could well infer from (\ref{11.2})-(\ref{11.4}) in a straightforward manner that
$\int_t^{t+1} \io \big\{ \frac{|\na\Theta|^2}{\Theta^2} + \frac{|\nas u_t|^2}{\Theta} + \frac{g}{\Theta} \big\} \to 0$ as $t\to\infty$;
In view of possibly lacking
strong $L^1$-compactness properties of $\frac{|\na\Teps|^2}{\Teps^2}$ and $\frac{|\nas\veps|^2}{\Teps}$, however,
it seems unclear how far this can be used to infer smallness of 
$\int_t^{t+1} \io \big\{ \frac{|\na\Teps|^2}{\Teps^2} + \frac{|\nas \veps|^2}{\Teps} + \frac{\geps}{\Teps} \big\}$
for large but fixed $t>0$ and suitably small $\eps\in (\eps_j)_{j\in\N}$;
as our approach toward Theorem \ref{theo55} aims at avoiding a direct use of the sparse knowledge on
properties of the limits $\Theta, u$ and $\mu$ provided by Theorem \ref{theo37} and Theorem \ref{theo38}, deriving
such uniform smallness properties of the entropy production rate in (\ref{11.6}) will form the core intention of the next lemmata.\abs
Our first step, naturally not yet requiring any assumption on large time decay of $f$ and $g$, 
will return to Lemma \ref{lem19} to obtain the following.
\begin{lem}\label{lem41}
  Assume that either (\ref{131.1}) or (\ref{38.1}) holds.
  Then
  \be{41.1}
	\leps(\Teps) \to \ell(\Theta)
	\quad \mbox{in } L^1_{loc}(\bom\times [0,\infty))
	\qquad \mbox{as } \eps=\eps_j\searrow 0,
  \ee
  where $\ell$ is as in (\ref{ell}).
\end{lem}
\proof
  Let $T>0$. Recalling Lemma \ref{lem6}, we can then fix $c_1=c_1(T)>0$ such that
  \be{41.2}
	\io \Keps(\Teps) \le c_1
	\qquad \mbox{for all $t\in (0,T)$ and } \eps\in (0,1),
  \ee
  while Lemma \ref{lem19} provides $c_2=c_2(T)>0$ fulfilling
  \be{41.3}
	\int_0^T \io \ln^2 \Teps \le c_2
	\qquad \mbox{for all } \eps\in (0,\epss),
  \ee
  where $\epss$ is as in Lemma \ref{lem18}.
  Given $\eta>0$, we next fix $\xi_0=\xi_0(\eta,T)>1$ such that
  \be{41.4}
	\frac{1}{\xi_0} \cdot c_1 T \le \frac{\eta}{4},
  \ee
  and writing $c_3\equiv c_3(T):=\|\kappa\|_{L^\infty((0,\xi_0))} +1$ we can thereupon choose $\del=\del(\eta,T)>0$ such that
  \be{41.5}
	\sqrt{c_2} c_3 \sqrt{\del} \le \frac{\eta}{2}
  \ee
  and
  \be{41.6}
	c_3 \xi_0 \del \le \frac{\eta}{4}.
  \ee
  Now if $E\subset \Om\times (0,T)$ is measurable and such that $|E|<\del$, then in the identity
  \be{41.7}
	\int\int_E \big|\leps(\Teps)\big|
	= \int\int_{E\cap\{\Teps<\xi_0\}} \big|\leps(\Teps)\big|
	+ \int\int_{E\cap\{\Teps\ge\xi_0\}} \big|\leps(\Teps)\big|,
	\qquad \eps\in (0,1),
  \ee
  we can go back to (\ref{leps}) to see that since
  \bas
	\big|\leps(\xi)\big|
	= \bigg| \int_1^\xi \frac{\keps(\sig)}{\sig} d\sig\bigg|
	&\le& c_3 \bigg| \int_1^\xi \frac{d\sig}{\sig} \bigg|
	= c_3 |\ln\xi|
	\qquad \mbox{for all $\xi\in (0,\xi_0)$ and } \eps\in (0,1)
  \eas
  thanks to (\ref{keps}), due to the Cauchy-Schwarz inequality, (\ref{41.3}) and (\ref{41.5}) we have
  \bea{41.8}
	\int\int_{E\cap\{\Teps<\xi_0\}} \big|\leps(\Teps)\big|
	&\le& c_3 \int\int_E |\ln\Teps| \nn\\
	&\le& c_3 \cdot \bigg\{ \int_0^T \io \ln^2 \Teps \bigg\}^\frac{1}{2} \cdot \sqrt{|E|} \nn\\
	&\le& c_3 \sqrt{c_2} \sqrt{|E|} \nn\\
	&\le& \frac{\eta}{2}
	\qquad \mbox{for all } \eps\in (0,\epss).
  \eea
  At large values, however, we can use that, again by (\ref{keps}),
  \bas
	\big|\leps(\xi)\big|
	&=& \int_1^{\xi_0} \frac{\keps(\sig)}{\sig} d\sig
	+ \int_{\xi_0}^\xi \frac{\keps(\sig)}{\sig} d\sig \nn\\
	&\le& c_3 \xi_0 + \frac{1}{\xi_0} \int_{\xi_0}^\xi \keps(\sig) d\sig \nn\\
	&\le& c_3 \xi_0 + \frac{1}{\xi_0} \cdot \Keps(\xi)
	\qquad \mbox{for all $\xi\ge\xi_0$ and } \eps\in (0,1),
  \eas
  so that we may draw on (\ref{41.2}) in estimating
  \bea{41.10}
	\int\int_{E\cap\{\Teps\ge\xi_0\}} \big|\leps(\Teps)\big|
	&\le& c_3 \xi_0 |E|
	+ \frac{1}{\xi_0} \int\int_E \Keps(\Teps) \nn\\
	&\le& c_3 \xi_0 |E|
	+ \frac{1}{\xi_0} \cdot c_1 T \nn\\
	&\le& \frac{\eta}{4} + \frac{\eta}{4}=\frac{\eta}{2}
	\qquad \mbox{for all } \eps\in (0,1)
  \eea
  because of (\ref{41.6}) and (\ref{41.4}).
  As $\eta>0$ was arbitrary here, from (\ref{41.7})-(\ref{41.10}) we infer that the family $(\leps(\Teps))_{\eps\in (0,\epss)}$ 
  is uniformly integrable over $\Om\times (0,T)$, and that due to the pointwise convergence properties in (\ref{24.6})
  and (\ref{244.6}), the Vitali convergence theorem hence asserts that $\leps(\Teps)\to \ell(\Theta)$ in
  $L^1(\Om\times (0,T))$ as $\eps=\eps_j\searrow 0$, because clearly $\leps\to\ell$ in $C^0_{loc}((0,\infty))$
  as $\eps\searrow 0$ by (\ref{leps}) and (\ref{keps}).
\qed
A direct implication of this and (\ref{11.1}) reads as follows.
\begin{lem}\label{lem42}
  If either (\ref{131.1}) or (\ref{38.1}) holds,
  then there exists a null set $N_\star\subset (0,\infty)$ such that
  $\ell(\cdot,t)\in L^1(\Om)$ for all $t\in (0,\infty)\sm N_\star$, that
  \be{42.1}
	\io \leps(\Teps(\cdot,t)) \to \io \ell(\Theta(\cdot,t))
	\quad \mbox{for all $t\in (0,\infty)\sm N_\star$}
	\qquad \mbox{as } \eps=\eps_j\searrow 0,
  \ee
  and that
  \be{42.2}
	\io \ell(\Theta(\cdot,t)) \ge \io \ell(\Theta(\cdot,t_0))
	\qquad \mbox{for all $t_0\in (0,\infty)\sm N_\star$ and } t\in (t_0,\infty)\sm N_\star.
  \ee
\end{lem}
\proof
  This is an immediate consequence of Lemma \ref{lem41} when combined with Lemma \ref{lem11} and the Fubini-Tonelli theorem.
\qed
We now additionally impose the decay hypothesis in (\ref{fg_decay}), and then firstly obtain uniform bounds for the quantities
on the left-hand side of (\ref{11.6}):
\begin{lem}\label{lem43}
  Assume that either (\ref{131.1}) or (\ref{38.1}) holds,
  and that (\ref{fg_decay}) is satisfied.
  Then there exists $C>0$ such that
  \be{43.1}
	\io \leps(\Teps(\cdot,t)) \le C
	\qquad \mbox{for all $t>0$ and } \eps\in (0,1).
  \ee
\end{lem}
\proof
  According to (\ref{leps}), for each $\eps\in (0,1)$ we have $\leps(\xi)\le 0$ for all $\xi\in (0,1]$ and 
  $\leps(\xi) \le \int_1^\xi \keps(\sig)d\sig \le \Keps(\xi)$ for all $\xi>1$, so that
  \bas
	\io \leps(\Teps) \le \int_{\{\Teps>1\}} \leps(\Teps) 
	\le \io \Keps(\Teps)
	\qquad \mbox{for all $t>0$ and } \eps\in (0,1).
  \eas
  We thus obtain (\ref{43.1}) from (\ref{6.3}).
\qed
In particular, this establishes a rigorous analogue of the monotone convergence property anticipated in (\ref{ent}):
\begin{lem}\label{lem44}
  If either (\ref{131.1}) or (\ref{38.1}) holds, and if additionally (\ref{fg_decay}) is fulfilled,
  then there exists $L\in\R$ such that with
  $N_\star$ as in Lemma \ref{lem42},
  \be{44.1}
	\frac{1}{|\Om|} \io \ell(\Theta(\cdot,t)) \nearrow L
	\qquad \mbox{as } (0,\infty)\sm N_\star \ni t \to \infty.
  \ee
\end{lem}
\proof
  We only need to note that due to (\ref{42.2}), the function defined by $h(t):=\frac{1}{|\Om|} \io \ell(\Theta(\cdot,t))$,
  $t\in (0,\infty)\sm N_\star$, is nondecreasing, and that Lemma \ref{lem43} together with (\ref{42.1}) asserts boundedness of $h$
  from above.
\qed
We can thereby draw our intended main conclusion concerning asymptotic smallness of entropy production in the 
approximate problems:
\begin{lem}\label{lem45}
  Suppose that either (\ref{131.1}) or (\ref{38.1}) holds, and that and (\ref{fg_decay}) are satisfied.
  Then for each $\eta>0$ there exists $T_1(\eta)>0$ with the property that whenever $t>T_1(\eta)$, we can find 
  $\epsss(\eta,t)\in (0,1)$ such that
  \be{45.1}
	\int_t^{t+1} \io \frac{|\na\Teps|^2}{\Teps^2} \le \eta
	\qquad \mbox{for all } \eps\in (\eps_j)_{j\in\N} \subset (0,\epsss(\eta,t))
  \ee
  and
  \be{45.2}
	\int_t^{t+1} \io \frac{|\nas\veps|^2}{\Teps} \le \eta
	\qquad \mbox{for all } \eps\in (\eps_j)_{j\in\N} \subset (0,\epsss(\eta,t))
  \ee
  as well as
  \be{45.3}
	\int_t^{t+1} \io \frac{\geps}{\Teps} \le \eta
	\qquad \mbox{for all } \eps\in (\eps_j)_{j\in\N} \subset (0,\epsss(\eta,t)).
  \ee
\end{lem}
\proof
  We take $\eta_1=\eta_1(\eta)>0$ such that $\frac{\eta_1}{D} \le \eta$, $\frac{\eta_1}{\kD} \le \eta$ and $\eta_1 \le 1$,
  and using Lemma \ref{lem44} we can then find $T_1=T_1(\eta)>0$ such that with $L$ as found there,
  \be{45.4}
	\io \ell(\Theta(\cdot,t))
	\ge |\Om|\cdot L - \frac{\eta_1}{2}
	\qquad \mbox{for all } t\in (T_1,\infty)\sm N_\star.
  \ee
  For fixed $t>T_1$, we pick any $t_0=t_0(t)\in (T_1,t)\sm N_\star$ and $t_1=t_1(t)\in (t+1,\infty)\sm N_\star$, and can then employ
  Lemma \ref{lem42} to choose $\epsss=\epsss(\eta,t)\in (0,1)$ such that
  \bas
	\io \leps(\Teps(\cdot,t_1))
	\le \io \ell(\Theta(\cdot,t_1)) + \frac{\eta_1}{4}
	\quad \mbox{and} \quad
	\io \leps(\Teps(\cdot,t_0)) \ge \io \ell(\Theta(\cdot,t_0)) - \frac{\eta_1}{4}
  \eas
  for all $\eps\in (\eps_j)_{j\in\N} \subset (0,\epsss)$.
  An application of (\ref{11.6}) then shows that since $\io \ell(\Theta(\cdot,t_1)) \le |\Om|\cdot L$ by (\ref{44.1}),
  \bas
	& & \hs{-30mm}
	D \int_{t_0}^{t_1} \io \frac{|\na\Teps|^2}{\Teps^2}
	+ \kD \int_{t_0}^{t_1} \io \frac{|\nas\veps|^2}{\Teps}
	+ \int_{t_0}^{t_1} \io \frac{\geps}{\Teps} \nn\\
	&\le& \io \leps(\Teps(\cdot,t_1)) - \io \leps(\Teps(\cdot,t_0)) \\
	&\le& \bigg\{ \io \ell(\Theta(\cdot,t_1)) + \frac{\eta_1}{4} \bigg\}
	- \bigg\{ \io \ell(\Theta(\cdot,t_0)) - \frac{\eta_1}{4} \bigg\} \\
	&\le& |\Om|\cdot L
	- \io \ell(\Theta(\cdot,t_0)) + \frac{\eta_1}{2}
	\qquad \mbox{for all } \eps\in (\eps_j)_{j\in\N} \subset (0,\epsss),
  \eas
  whence using (\ref{45.4}) and the fact that $t_0<t<t+1<t_1$ we find that
  \bas
	D \int_{t_0}^{t_1} \io \frac{|\na\Teps|^2}{\Teps^2}
	+ \kD \int_{t_0}^{t_1} \io \frac{|\nas\veps|^2}{\Teps}
	+ \int_{t_0}^{t_1} \io \frac{\geps}{\Teps} 
	\le \frac{\eta_1}{2} + \frac{\eta_1}{2} =\eta_1
	\qquad \mbox{for all } \eps\in (\eps_j)_{j\in\N} \subset (0,\epsss).
  \eas
  In view of our choice of $\eta_1$, this implies (\ref{45.1}), (\ref{45.2}) and (\ref{45.3}).
\qed
\subsection{Proof of Theorem \ref{theo55}}
\newcommand{\rM}{\rho^{(M)}}
\newcommand{\lM}{\ell^{(M)}}
\newcommand{\lMe}{\leps^{(M)}}
\newcommand{\I}{{\mathcal{I}}}
\newcommand{\ZMk}{Z_k^{(M)}}
\newcommand{\ZMke}{Z_{k\eps}^{(M)}}
\newcommand{\LM}{L^{(M)}}
In order to make Lemma \ref{lem45} applicable so as to assert, regardless of possibly rapid growth of $\kappa$, some 
favorable compactness properties of $\phi(\Theta)$ and approximate variants for some $\phi:(0,\infty)\to\R$, 
in the following lemma we introduce certain relatives of $\ell$ and the $\leps$ which have undergone a
cut-off procedure in their derivatives.
\begin{lem}\label{lem47}
  Assume (\ref{kappa_reg}), and for $M>1$ let $\rM\in C_0^\infty([0,\infty))$ be such that $\rM\equiv 1$ on $[0,M]$,
  $\rM\equiv 0$ on $[M+1,\infty)$ and $0\ge (\rM)' \ge -2$, and set
  \be{lM}
	\lM(\xi):=\int_1^\xi \frac{\rM(\sig)\kappa(\sig)}{\sig} d\sig,
	\qquad \xi>0,
  \ee
  as well as
  \be{lMe}
	\lMe(\xi):=\int_1^\xi \frac{\rM(\sig)\keps(\sig)}{\sig} d\sig,
	\qquad \xi>0, \ \eps\in (0,1).
  \ee
  Then $\lM\in C^1((0,\infty))$ with $(\lM)'(\xi)=\frac{\rM(\xi)\kappa(\xi)}{\xi}$ for all $\xi>0$ and 
  $\lMe\in C^1((0,\infty))$ with $(\lMe)'(\xi)=\frac{\rM(\xi)\keps(\xi)}{\xi}$ for all $\xi>0$ and $\eps\in (0,1)$,
  and we have
  \be{47.1}
	\ell(\xi)-\frac{1}{M}\cdot K(\xi) \le \lM(\xi) \le \ell(\xi)
	\qquad \mbox{for all $\xi>0$ and } M>1.
  \ee
\end{lem}
\proof
  Since for each $M>1$ the properties of $\rM$ particularly ensure that for all $\xi>0$ we have
  \bas
	0 \le \ell(\xi)-\lM(\xi)
	= \int_1^\xi \frac{(1-\rM(\sig))\kappa(\sig)}{\sig} d\sig
	\le \int_M^\xi \frac{\kappa(\sig)}{\sig} d\sig 
	\le \frac{1}{M} \int_M^\xi \kappa(\sig) d\sig
	\le \frac{1}{M} K(\xi)
  \eas
  due to (\ref{K}), all statements are evident.
\qed
From lemma \ref{lem45} we can then indeed deduce the following.
\begin{lem}\label{lem48}
  Assume either (\ref{131.1}) or (\ref{38.1}), and that additionally (\ref{fg_decay}) holds, and let $p>n$.
  Then for any $\eta>0$ and $M>1$ there exists $T(\eta,M)>0$ such that 
  to each $t>T(\eta,M)$ there corresponds some $\epssss(\eta,M,t)\in (0,1)$ with the property that
  \be{48.1}
	\int_t^{t+1} \io \big| \na\lMe(\Teps)\big|^2 \le \eta
	\qquad \mbox{for all } \eps\in (\eps_j)_{j\in\N}\cap (0,\epssss(\eta,M,t))
  \ee
  and
  \be{48.2}
	\int_t^{t+1} \Big\| \pa_t \lMe(\Teps(\cdot,s)) \Big\|_{(W^{1,p}(\Om))^\star} ds \le \eta
	\qquad \mbox{for all } \eps\in (\eps_j)_{j\in\N}\cap (0,\epssss(\eta,M,t)).
  \ee
\end{lem}
\proof
  We abbreviate $c_1\equiv c_1(M):=\|\kappa\|_{L^\infty((0,M+1))} +1$ and let $c_2>0$ be such that in line with our assumption
  that $p>n$, any $\vp\in C^1(\bom)$ with $\|\vp\|_{W^{1,p}(\Om)} \le 1$ satisfies 
  $\|\vp\|_{L^\infty(\Om)} + \|\na\vp\|_{L^2(\Om)} \le c_2$.
  For fixed $\eta>0$, we then let $\eta_1\equiv \eta_1(\eta,M)>0$ be small enough such that $c_1 \eta_1\le \eta$ and
  \be{48.4}
	\big\{ c_2 D + 2c_2 D\cdot (M+1) + 2c_2\big\} \cdot \eta_1
	+ \big\{ c_2 D + c_2\sqrt{M+1}\sqrt{|\Om|} |\B|\big\} \cdot \sqrt{\eta_1} \le \eta,
  \ee 
  and in line with Lemma \ref{lem45} we now pick $T=T(\eta,M)>1$ and $(\epssss(\eta,M,t))_{t>T} \subset (0,1)$
  such that whenever $t>T$ and $\eps\in (\eps_j)_{j\in\N} \cap (0,\epssss(\eta,M,t))$, 
  \be{48.6}
	\int_t^{t+1} \io \frac{|\na\Teps|^2}{\Teps^2} \le \eta_1
  \ee
  and
  \be{48.7}
	\int_t^{t+1} \io \frac{\lan\D:\nas\veps,\nas\veps\ran}{\Teps}
	+ \int_t^{t+1} \io \frac{|\nas\veps|^2}{\Teps} 
	\le \eta_1
  \ee
  as well as
  \be{48.8}
	\int_t^{t+1} \io \frac{\geps}{\Teps} 
	\le \eta_1.
  \ee
  Since
  \bas
	(\lMe)'(\xi) = \frac{\rM(\xi)\keps(\xi)}{\xi}
	\qquad \mbox{for all $\xi>0$ and } \eps\in (0,1)
  \eas
  by (\ref{lM}), and since thus (\ref{keps}) along with our restrictions on $\rM$ 
  from Lemma \ref{lem47} warrants that for all $\eps\in (0,1)$ we have
  \be{48.10}
	(\lMe)' \equiv 0 \mbox{ on } [M+1,\infty)
	\quad \mbox{and} \quad
	\big| (\lMe)'(\xi) \big| \le \frac{\keps(\xi)}{\xi} \le \frac{c_1}{\xi}
	\mbox{ for all } \xi\in (0,M+1),
  \ee 
  for each $t>T$ we can rely on (\ref{48.6}) and the inequality $c_1\eta_1\le\eta$ to estimate
  \bea{48.11}
	\int_t^{t+1} \io \big| \na \lMe(\Teps) \big|^2
	&=& \int_t^{t+1} \io \big|(\lMe)'(\Teps)\big|^2 |\na\Teps|^2 \nn\\
	&\le& c_1 \int_t^{t+1} \io \frac{|\na\Teps|^2}{\Teps^2} \nn\\
	&\le& \eta
	\qquad \mbox{for all } \eps\in (\eps_j)_{j\in\N} \cap (0,\epssss(\eta,M,t)).
  \eea
  We next go back to (\ref{0eps}) and draw on or choice of $c_2$ to see that given $\vp\in C^1(\bom)$ satisfying
  $\|\vp\|_{W^{1,p}(\Om)} \le 1$, thanks to the Cauchy-Schwarz inequality, and again due to (\ref{48.10}), we have
  \bas
	\bigg| \io \pa_t \lMe(\Teps) \cdot\vp \bigg|
	&=& \bigg| \io \frac{\rM(\Teps) \keps(\Teps) \Theta_{\eps t}}{\Teps} \cdot\vp\bigg| \\
	&=& \bigg| \io \frac{\rM(\Teps)}{\Teps} \cdot \Big\{
	D\Del\Teps + \lan\D:\nas\veps,\nas\veps\ran - \Teps\lan\B,\nas\veps\ran + \geps \Big\}\cdot\vp \bigg| \\
	&=& \bigg|
		D \io \frac{\rM(\Teps) |\na\Teps|^2}{\Teps^2} \vp
	- D \io \frac{(\rM)'(\Teps) |\na\Teps|^2}{\Teps} \vp \nn\\
	& & - D\io \frac{\rM(\Teps)}{\Teps} \na\Teps\cdot\na\vp 
	+ \io \frac{\rM(\Teps) \lan\D:\nas\veps,\nas\veps\ran}{\Teps} \vp \nn\\
	& & \hs{6mm}
	- \io \rM(\Teps) \lan\B,\nas\veps\ran\vp
	+ \io \frac{\rM(\Teps)\geps}{\Teps} \vp \bigg| \\
	&\le& c_2 D \io \frac{|\na\Teps|^2}{\Teps^2}
	+ 2c_2 D \int_{\{\Teps<M+1\}} \frac{|\na\Teps|^2}{\Teps}
	+ c_2 D \cdot\bigg\{ \io \frac{|\na\Teps|^2}{\Teps^2}\bigg\}^\frac{1}{2} \\
	& & + c_2 \io \frac{\lan\D:\nas\veps,\nas\veps\ran}{\Teps}
	+ c_2 |\B| \cdot \bigg\{ \io \frac{|\nas\veps|^2}{\Teps} \bigg\}^\frac{1}{2} \cdot
		\bigg\{ \int_{\{\Teps<M+1\}} \Teps\bigg\}^\frac{1}{2} \nn\\
	& & + c_2 \io \frac{\geps}{\Teps}
	\qquad \mbox{for all $t>0$ and } \eps\in (0,1),
  \eas
  because $|\rM|\le 1$ and $\big| (\rM)'\big| \le 2$.
  Estimating
  \bas
	\int_{\{\Teps<M+1\}} \frac{|\na\Teps|^2}{\Teps}
	\le (M+1) \io \frac{|\na\Teps|^2}{\Teps^2}
	\qquad \mbox{and} \qquad
	\int_{\{\Teps<M+1\}} \Teps \le (M+1)|\Om|
  \eas
  for $t>0$ and $\eps\in (0,1)$, once more by means of the Cauchy-Schwarz inequality we thus infer from (\ref{48.6})-(\ref{48.8})
  and (\ref{48.4}) that whenever $t>T$,
  \bas
	& & \hs{-20mm}
	\int_t^{t+1} \Big\| \pa_t \lMe(\Teps(\cdot,s)) \Big\|_{(W^{1,p}(\Om))^\star} ds  \\
	&\le& \big\{ c_2 D + 2c_2 D \cdot (M+1)\big\} \cdot \int_t^{t+1} \io \frac{|\na\Teps|^2}{\Teps^2} 
	+ c_2 D\cdot\bigg\{ \int_t^{t+1} \io \frac{|\na\Teps|^2}{\Teps^2} \bigg\}^\frac{1}{2} \nn\\
	& & + c_2 \int_t^{t+1} \io \frac{\lan\D:\nas\veps,\nas\veps\ran}{\Teps}
	+ c_2 \sqrt{M+1} \sqrt{|\Om|} |\B| \cdot \bigg\{ \int_t^{t+1} \io \frac{|\nas\veps|^2}{\Teps} \bigg\}^\frac{1}{2} \\
	& & + c_2 \int_t^{t+1} \io \frac{\geps}{\Teps} \nn\\
	&\le& \big\{ c_2 D + 2c_2 D\cdot (M+1)\big\} \cdot \eta_1 + c_2 D\sqrt{\eta_1} \\
	& & + c_2 \eta_1 + c_2\sqrt{M+1} \sqrt{|\Om|} |\B| \sqrt{\eta_1} + c_2\eta_1 \nn\\
	&\le& \eta
	\qquad \mbox{for all } \eps\in (\eps_j)_{j\in\N} \cap (0,\epssss(\eta,M,t)).
  \eas
  Together with (\ref{48.11}), this yields the claim.
\qed
On the basis of compactness properties implied by the previous lemma together with Lemma \ref{lem41}, 
we can prepare a characterization of $\omega$-limits in the trajectories $(\lM(\Theta(\cdot,t)))_{t>0}$:
\begin{lem}\label{lem49}
  Let either (\ref{131.1}) or (\ref{38.1}) be satisfied, let (\ref{fg_decay}) be fulfilled,
  and fix $(t_k)_{k\in\N} \subset (0,\infty)$ such that $t_k\to\infty$ as $k\to\infty$.
  Then for each $M>1$, letting
  \be{ZMk}
	\ZMk(x,s):=\lM\big(\Theta(x,t_k+s)\big),
	\qquad x\in\Om, \ s\in (0,1), \ k\in\N,
  \ee
  and
  \be{ZMke}
	\ZMke(x,s):=\lMe\big(\Teps(x,t_k+s)\big),
	\qquad x\in\Om, \ s\in (0,1), \ k\in\N,
  \ee
  we have $\ZMk \in L^1(\Om\times (0,1))$ and $\na\ZMk\in L^2(\Om\times (0,1);\R^n)$ for all $k\in\N$ and $M>1$,
  \be{49.1}
	\ZMke \to \ZMk
	\quad \mbox{in } L^1(\Om\times (0,1))
	\quad \mbox{as } \eps=\eps_j\searrow 0
	\qquad \mbox{for all $k\in\N$ and } M>1,
  \ee
  and
  \be{49.11}
	\na\ZMke \wto \na\ZMk
	\quad \mbox{in } L^2(\Om\times (0,1);\R^n)
	\quad \mbox{as } \eps=\eps_j\searrow 0
	\qquad \mbox{for all $k\in\N$ and } M>1,
  \ee
  and, moreover,
  \be{49.2}
	\big(\ZMk\big)_{k\in\N}
	\mbox{ is relatively compact in }
	L^1(\Om\times (0,1))
	\quad \mbox{for all } M>1.
  \ee
\end{lem}
\proof
  Since (\ref{lMe}), (\ref{lM}) and (\ref{keps}) imply that for each $M>1$ we have $\lMe\to \lM$ uniformly in $(0,\infty)$
  as $\eps\searrow 0$, from (\ref{24.6}), (\ref{244.6}), (\ref{ZMke}) and (\ref{ZMk}) we obtain that for fixed $k\in\N$ and $M>1$,
  $\ZMke\to \ZMk$ a.e.~in $\Om\times (0,1)$ 
  as $\eps=\eps_j\searrow 0$.
  Since (\ref{lMe}) and (\ref{lM}) furthermore guarantee that
  \bas
	\big|\lMe(\xi)\big| \le |\leps(\xi)|
	\qquad \mbox{for all $\xi>0$, $\eps\in (0,1)$ and } M>1,
  \eas
  and since thus Lemma \ref{lem41} and the Dunford-Pettis theorem warrant uniform integrability of $(Z^{(M)}_{k \eps_j})_{j\in\N}$,
  the Vitali convergence theorem ensures that indeed (\ref{49.1}) holds.\abs
  We next fix any $p>n$ and apply Lemma \ref{lem48} to $\eta:=1$ to find $(\wh{\eps}(k))_{k\in\N} \subset (0,1)$ such that
  \bea{49.3}
	& & \hs{-20mm}
	\int_0^1 \io \big|\na\ZMke|^2 \le 1
	\quad \mbox{and} \quad
	\int_0^1 \big\| \pa_t \ZMke(\cdot,s)\big\|_{(W^{1,p}(\Om))^\star} ds \le 1 \nn\\
	& & \hs{20mm}
	\qquad \mbox{for all $\eps\in (\eps_j)_{j\in\N} \cap (0,\wh{\eps}(k))$ and any } k\in\N,
  \eea
  and note that in conjunction with (\ref{49.1}), the compactness property resulting from the first estimate herein warrants validity
  of (\ref{49.2}).\abs
  Moreover, again using (\ref{lMe}) and (\ref{keps}) we see that with $c_1\equiv c_1(M):=\|\kappa\|_{L^\infty((0,M+1))}+1$,
  whenever $\eps\in (0,1)$ we have
  \bas
	|\lMe(\xi)| \le c_1 |\ln\xi|
	\mbox{ for all } \xi\in (0,M+1)
	\quad \mbox{and} \quad
	|\lMe(\xi)| \le (M+1) c_1
	\qquad \mbox{for all } \xi\ge M+1,
  \eas
  so that
  \bas
	\int_0^1 \big|\ZMke|^2
	\le \int_{t_k}^{t_k+1} \io \big\{ c_1 |\ln\Teps| + (M+1) c_1 \big\}^2
	\qquad \mbox{for all $\eps\in (0,1)$ and } k\in\N,
  \eas
  whence employing Lemma \ref{lem19} and an Aubin-Lions lemma we conclude from (\ref{49.3}) that if we let
  $\epsss$ be as in Lemma \ref{lem18}, then the set
  \bas
	S^{(M)}
	:=\Big\{ \ZMke \ \Big| \ k\in\N, \ \eps\in (\eps_j)_{j\in\N} \cap (0,\wh{\eps}(k)) \cap (0,\epss) \Big\}
  \eas
  is relatively compact in $L^1(\Om\times (0,1))$. 
  Since (\ref{49.1}) asserts that with respect to the norm topology in $L^1(\Om\times (0,1))$ we have
  $\{ \ZMk \ | \ k\in\N\} \subset \ov{S^{(M)}}$, however, this establishes (\ref{49.2}).
\qed
In concluding that only spatially constant profiles can appear as large time limits, we will make use of the following
essentially well-known statament from real analysis, a proof of which we may leave to the reader.
\begin{lem}\label{lem50}
  Let $T>0$, and let $(h_j)_{j\in\N} \subset L^2((0,T);W^{1,2}(\Om))$ and $h\in L^1(\Om\times (0,T))$ be such that
  $h_j\wto h$ in $L^1(\Om\times (0,T))$ and $\na h_j\to 0$ in $L^2(\Om\times (0,T);\R^n)$ as $j\to\infty$.
  Then $h\in L^2((0,T);W^{1,2}(\Om))$ and $\na h\equiv 0$ a.e.~in $\Om\times (0,T)$.
\end{lem}
A combination of Lemma \ref{lem48} with Lemma \ref{lem49} thus yields the following nucleus of our claim concerning
large time stabilization of $\Theta$ toward a uniquely determined positive constant.
\begin{lem}\label{lem52}
  Assume either (\ref{131.1}) or (\ref{38.1}), as well as (\ref{fg_decay}).
  Then there exists $\Gamma>0$ with the following property.
  If $M>1$, and if $(t_k)_{k\in\N} \subset (0,\infty)$ is such that $t_k\to\infty$ as $k\to\infty$,
  then there exist $(k_l)_{l\in\N}\subset\N$ and
  \be{52.01}
	\LM \in \Big[ L-\frac{\Gamma}{M} , L\Big]
  \ee
  such that $k_l\to\infty$ and
  \be{52.1}
	Z_{k_l}^{(M)} \to \LM
	\quad \mbox{in } L^1(\Om\times (0,1))
  \ee
  as $l\to\infty$.
\end{lem}
\proof
  According to Lemma \ref{lem6}, (\ref{keps}) and (\ref{24.6}) and (\ref{244.6}), we can find $c_1>0$ such that
  \be{52.2}
	\io K(\Theta) \le c_1
	\qquad \mbox{for a.e.~} t>0,
  \ee
  and invoking Lemma \ref{lem49} we can fix $(k_l)_{l\in\N}\subset\N$ and $Z_\infty^{(M)} \in L^1(\Om\times (0,1))$ such that
  $k_l\to\infty$ and
  \be{52.3}
	Z_{k_l}^{(M)} \to Z_\infty^{(M)}
	\qquad \mbox{in $L^1(\Om\times (0,1))$ and a.e.~in $\Om\times (0,1)$}
  \ee
  as $l\to\infty$.
  To firstly confirm that
  \be{52.5}
	\na Z_\infty^{(M)} =0
	\qquad \mbox{a.e.~in } \Om\times (0,1),
  \ee
  we note that given $\eta>0$ we may infer from Lemma \ref{lem48} that there exist $l^{(1)}=l^{(1)}(\eta)\in\N$ and
  $(\eps^{(1)}(l))_{l>l^{(1)}} \subset (0,1)$ such that for each fixed $l>l^{(1)}$ we have
  \bas
	\int_{t_{k_l}}^{t_{k_l}+1} \io \big| \na \lMe(\Teps)\big|^2 \le \eta
	\qquad \mbox{for all } \eps\in (\eps_j)_{j\in\N} \cap (0,\eps^{(1)}(l)),
  \eas
  meaning that
  \bas
	\int_0^1 \io \big| \na Z_{k_l \eps}^{(M)}\big|^2 \le \eta
	\qquad \mbox{for all } \eps\in (\eps_j)_{j\in\N} \cap (0,\eps^{(1)}(l)).
  \eas
  In view of (\ref{49.11}) and an argument based on lower semicontinuity, this implies that
  \bas
	\int_0^1 \io |\na Z_{k_l}|^2 \le \eta
	\qquad \mbox{for all } l>l^{(1)},
  \eas
  and that thus
  \bas
	\na Z_{k_l} \to 0
	\quad \mbox{in } L^2(\Om\times (0,1);\R^n)
	\qquad \mbox{as } l\to\infty,
  \eas
  so that (\ref{52.5}) is a consequence of Lemma \ref{lem50} and (\ref{52.3}).\abs
  Next, picking an arbitrary $p>n$ we infer from Lemma \ref{lem49} that for each $\eta>0$ there exist $l^{(2)}=l^{(2)}(\eta)\in\N$
  and $(\eps^{(2)}(l))_{l>l^{(2)}} \subset (0,1)$ such that if $l>l^{(2)}$, then
  \be{52.6}
	\int_{t_{k_l}}^{t_{k_l}+1} \Big\| \pa_t \lMe\big(\Teps(\cdot,s)\big)\Big\|_{(W^{1,p}(\Om))^\star} ds \le \eta
	\qquad \mbox{for all } \eps\in (\eps_j)_{j\in\N} \cap (0,\eps^{(2)}(l)),
  \ee
  and prepare an appropriate conclusion from this by combining (\ref{49.1}) and (\ref{52.3}) with the Fubini-Tonelli theorem
  to fix a null set $N_1\subset (0,1)$ such that
  \be{52.7}
	Z_{k_l \eps}^{(M)}(\cdot,s) \to Z_{k_l}^{(M)} (\cdot,s)
	\ \mbox{ in } L^1(\Om)
	\quad \mbox{as } \eps=\eps_j\searrow 0
	\qquad \mbox{for all $s\in (0,1))\sm N_1$ and any } l\in\N,
  \ee
  and that
  \be{52.8}
	Z_{k_l}^{(M)}(\cdot,s) \to Z_\infty^{(M)}(\cdot,s)
	\ \mbox{ in } L^1(\Om)
	\quad \mbox{as } l\to\infty
	\qquad \mbox{for all } s\in (0,1)\sm N_1.
  \ee
  Now whenever $s_1\in (0,1)\sm N_1$ and $s_2\in (s_1,1)\sm N_1$, from (\ref{52.6}) we obtain that
  \bas
	\Big\| Z_{k_l \eps}^{(M)}(\cdot,s_2) - Z_{k_l \eps}^{(M)}(\cdot,s_1) \Big\|_{(W^{1,p}(\Om))^\star}
	&=& \Big\| \lMe\big(\Teps(\cdot,t_{k_l}+s_2)\big) - \lMe\big(\Teps(\cdot,t_{k_l}+s_1)\big) \Big\|_{(W^{1,p}(\Om))^\star} \\
	&\le& \int_{t_{k_l}+s_1}^{t_{k_l}+s_2} \Big\| \pa_t \lMe\big(\Teps(\cdot,s)\big) \Big\|_{(W^{1,p}(\Om))^\star} ds \\
	&\le& \eta
	\qquad \mbox{for all $\eps\in (\eps_j)_{j\in\N} \cap (0,\eps^{(2)}(l))$ and } l>l^{(2)},
  \eas
  which thanks to (\ref{52.7}) and the continuity of the embedding $L^1(\Om) \hra (W^{1,p}(\Om))^\star$ implies that
  \bas
	\Big\| Z_{k_l}^{(M)}(\cdot,s_2) - Z_{k_l}^{(M)}(\cdot,s_1) \Big\|_{(W^{1,p}(\Om))^\star} 
	\le \eta
	\qquad \mbox{for all } l>l^{(2)},
  \eas
  while from (\ref{52.8}) we similarly infer that
  \bas
	\Big\| Z_\infty^{(M)}(\cdot,s_2)-Z_\infty^{(M)}(\cdot,s_1)\Big\|_{(W^{1,p}(\Om))^\star} 
	\le \eta.
  \eas
  Taking $\eta\searrow 0$ here shows that $Z_\infty^{(M)}(\cdot,s_2)=Z_\infty^{(M)}(\cdot,s_1)$ a.e.~in $\Om$ for any such $s_1$ and
  $s_2$, and that hence in view of (\ref{52.5}) there must exist $\LM\in\R$ such that $Z_\infty^{(M)} =\LM$ a.e.~in $\Om\times (0,1)$,
  whence (\ref{52.1}) follows from (\ref{52.3}).\abs
  The characterization in (\ref{52.01}), finally, can be achieved by utilizing (\ref{47.1}):
  Due to the right inequality therein, namely, letting $N_\star$ be as in Lemma \ref{lem42} and writing
  $N_2:=N_1 \bigcup_{l\in\N} \big\{ s\in (0,1) \ \big| \ t_{k_l}+s \in N_\star\big\}$ we have
  \bas
	\io Z_{k_l}^{(M)}(\cdot,s)
	= \io \lM\big(\Theta(\cdot,t_{k_l}+s)\big)
	\le \io \ell \big(\Theta(\cdot,t_{k_l}+s)\big)
	\qquad \mbox{for all $s\in (0,1)\sm N_2$ and } l\in\N,
  \eas
  so that from (\ref{52.8}), (\ref{52.1}) and Lemma \ref{lem44} it follows that
  \be{52.77}
	\LM \cdot |\Om|
	= \io Z_\infty^{(M)}(\cdot,s) 
	\le L |\Om|
	\qquad \mbox{for all }s\in (0,1)\sm N_2.
  \ee
  Moreover, using the left inequality in (\ref{47.1}) in estimating
  \bas
	(L-\LM) \cdot |\Om|
	&\le& \int_{t_{k_l}}^{t_{k_l}+1} \io \big(L-\ell(\Theta)\big)
	+ \int_{t_{k_l}}^{t_{k_l}+1} \io \big(\ell(\Theta)-\lM(\Theta)\big) \\
	& & + \int_{t_{k_l}}^{t_{k_l}+1} \io \big| \lM(\Theta)-\LM\big| \\
	&\le& 
	|\Om|\cdot \rm{ess} \sup_{\hs{-3mm} t>t_{k_l}} \bigg\{ L - \frac{1}{|\Om|} \io \ell\big(\Theta(\cdot,t)\big) \bigg\}
	+ \frac{1}{M} \int_{t_{k_l}}^{t_{k_l}+1} \io K(\Theta) \\
	& & + \int_0^1 \io \big|Z_{k_l}^{(M)} - Z_\infty^{(M)} \big|
	\qquad \mbox{for all } l\in\N,
  \eas
  we infer upon letting $l\to\infty$ that thanks to Lemma \ref{lem44}, (\ref{52.3}) and (\ref{52.2}),
  \bas
	(L-\LM) \cdot |\Om|
	\le \frac{1}{M}\cdot c_1.
  \eas
  Combined with (\ref{52.77}), this verifies (\ref{52.01}) with $\Gamma:=\frac{c_1}{|\Om|}$.
\qed
In fact, a fairly simple argument enables us to turn this into a corresponding statement that does no
longer involve any of the auxiliary numbers $M$ appearing above, formulated here in terms of the original variables $x$ and $t$.
\begin{lem}\label{lem53}
  Assume either (\ref{131.1}) or (\ref{38.1}), and suppose that (\ref{fg_decay}) holds.
  Then
  \be{53.1}
	\int_t^{t+1} \io \big|\ell(\Theta)-L\big|
	\to 0
	\qquad \mbox{as } t\to\infty.
  \ee
\end{lem}
\proof
  If the claim was false, then there would exist $c_1>0$ and $(t_k)_{k\in\N}\subset (0,\infty)$ such that $t_k\to\infty$
  as $k\to\infty$ and
  \bas
	\int_{t_k}^{t_k+1} \io \big|\ell(\Theta)-L\big| \ge c_1
	\qquad \mbox{for all } k\in\N,
  \eas  
  whence for
  \bas
	Z_k(x,s):=\ell\big(\Theta(x,t_k+s)\big),
	\qquad x\in\Om, \ s\in (0,1), \ k\in\N,
  \eas
  we would have
  \be{53.2}
	\int_0^1 \io |Z_k-L| \ge c_1
	\qquad \mbox{for all } k\in\N.
  \ee
  Taking $c_2>0$ such that in line with Lemma \ref{lem6}, (\ref{24.6}) and (\ref{244.6}) 
  we have $\io K(\Theta) \le c_2$ for a.e.~$t>0$, we could then fix $M>1$ large enough such that
  \be{53.4}
	\frac{c_2}{M} \le \frac{c_1}{6}
	\qquad \mbox{and} \qquad
	\frac{\Gamma |\Om|}{M} \le \frac{c_1}{6},
  \ee
  where $\Gamma$ is the constant found in Lemma \ref{lem52},
  and by means of said lemma we could then find $k_0\in\N$ such that
  \be{53.3}
	\int_0^1 \io \big| Z_{k_0}^{(M)} - \LM\big| \le \frac{c_1}{6}.
  \ee
  Since our choice of $c_2$ ensures that, again thanks to (\ref{47.1}),
  \bas
	\int_0^1 \io \big| Z_k-Z_k^{(M)}\big|
	= \int_{t_k}^{t_k+1} \io \big|\ell(\Theta)-\lM(\Theta)\big| 
 	\le \frac{1}{M} \int_{t_k}^{t_k+1} \io K(\Theta) 
	\le \frac{c_2}{M}
	\qquad \mbox{for all } k\in\N,
  \eas
  from (\ref{53.3}) and (\ref{53.4}) we would therefore obtain, however, that
  \bas
	\int_0^1 \io \big| Z_{k_0}-L\big|
	&\le& \int_0^1 \io \big| Z_{k_0} - Z_{k_0}^{(M)} \big|
	+ \int_0^1 \io \big| Z_{k_0}^{(M)}- \LM \big|
	+ \int_0^1 \io \big| \LM-L \big| \\
	&\le& \frac{c_2}{M}
	+ \frac{c_1}{6}
	+ \big| \LM-L\big| \cdot |\Om| \\
	&\le& \frac{c_2}{M}
	+ \frac{c_1}{6}
	+ \frac{\Gamma |\Om|}{M} \\
	&\le& 
	\frac{c_1}{6} + \frac{c_1}{6} + \frac{c_1}{6} = \frac{c_1}{2},
  \eas
  which is incompatible with (\ref{53.2}).
\qed
Let us finally prepare a strengthening of the topological framework by simple interpolation:
\begin{lem}\label{lem54}
  Let either (\ref{131.1}) or (\ref{38.1}) be satisfied, assume (\ref{fg_decay}),
  and suppose that $p>0$ is such that
  \be{54.1}
	\big(\Theta^p(\cdot,\cdot+t)\big)_{t>0}
	\mbox{ is uniformly integrable over } \Om\times (0,1)
  \ee
  Then
  \be{54.2}
	\int_t^{t+1} \io |\Theta-\Tinf|^p \to 0
	\qquad \mbox{as } t\to\infty,
  \ee
  where
  \be{54.3}
	\Tinf:=\ell^{-1}(L)>0.
  \ee
\end{lem}
\proof	
  If (\ref{54.2}) did not hold, then we could find $(t_k)_{k\in\N}\subset (0,\infty)$ such that $t_k\to\infty$ as $k\to\infty$, 
  and that
  \bas
	h_k(x,s):=\Theta(x,t_k+s),
	\qquad x\in\Om, \ s\in (0,1), \ k\in\N,
  \eas
  would satisfy
  \be{54.4}
	\int_0^1 \io |h_k-\Tinf|^p \ge c_1
	\qquad \mbox{for all } k\in\N.
  \ee
  But Lemma \ref{lem53} ensures that $\int_0^1 \io |\ell(h_k)-L| \to 0$ as $k\to\infty$,
  whence there would exist $(k_l)_{l\in\N} \subset\N$ such that $k_l\to\infty$ and $\ell(h_{k_l})- L \to 0$ a.e.~in $\Om\times (0,1)$
  as $l\to\infty$.
  Therefore, $h_{k_l} \to \ell^{-1}(L)=\Tinf$ and hence $|h_{k_l}-\Tinf|^p \to 0$ a.e.~in $\Om\times (0,1)$ as $l\to\infty$,
  so that since (\ref{54.1}) entails that $\big( |h_k-\Tinf|^p \big)_{k\in\N}$ is uniformly integrable over $\Om\times (0,1)$,
  the Vitali convergence theorem would lead to the conclusion that
  \bas
	\int_0^1 \io |h_{k_l}-\Tinf|^p \to 0
	\qquad \mbox{as } l\to\infty,
  \eas
  which would contradict (\ref{54.4}).
\qed
Our main result on large time behavior of the thermal part in (\ref{0}) has thereby become fairnyl evident:\abs
\proofc of Theorem \ref{theo55}. \quad
  We let $\Tinf$ be as in Lemma \ref{lem54}, and note that if (\ref{131.1}) holds, then Lemma \ref{lem13} together 
  with (\ref{24.6})
  yields $c_1>0$ such that $\io \Theta \le c_1$ for a.e.~$t>0$.
  In particular, this ensures that (\ref{54.1}) is satisfied for each $p\in (0,1)$, so that (\ref{55.1}) results from
  Lemma \ref{lem54}.\\
  If, alternatively, (\ref{38.1}) is valid, then Lemma \ref{lem15} and Lemma \ref{lem17} in conjunction with (\ref{244.6})
  assert (\ref{54.1}) with $p=1$, meaning that also (\ref{55.2}) is a consequence of Lemma \ref{lem54}.
\qed
\mysection{Large time decay of $u_t$. Proof of Proposition \ref{prop56}}\label{sect_ut}
Arguing toward large time decay of $u_t$ can as well be related to the uniform smallness properties asserted by Lemma \ref{lem45}:
\begin{lem}\label{lem58}
  Assume either (\ref{131.1}) or (\ref{38.1}), as well as (\ref{fg_decay}).
  Then there exists $C>0$ such that 
  \be{58.1}
	\int_t^{t+1} \io |\veps| 
	\le C \cdot \bigg\{ \int_t^{t+1} \io \frac{|\nas\veps|^2}{\Teps} \bigg\}^\frac{1}{2}
	\qquad \mbox{for all $t>0$ and } \eps\in (0,1).
  \ee
\end{lem}
\proof
  According to Lemma \ref{lem13} and Lemma \ref{lem17}, the hypotheses in (\ref{131.1}), (\ref{38.1}) and (\ref{fg_decay}) 
  ensure the existence of $c_1>0$ such that
  \bas
	\int_t^{t+1} \io \Teps \le c_1
	\qquad \mbox{for all $t>0$ and } \eps\in (0,1).
  \eas
  Choosing $c_2>0$ such that in line with a Poincar\'e-Korn inequality (\cite[Theorem 1 ii)]{solombrino}) we have
  \bas
	\|\vp\|_{L^1(\Om)} \le c_2\|\nas\vp\|_{L^1(\Om)}
	\qquad \mbox{for all $\vp\in C^\infty(\bom;\R^n)$ fulfilling $\vp=0$ on $\pO$,}
  \eas
  by means of the Cauchy-Schwarz inequality we thus obtain that
  \bas
	\int_t^{t+1} \io |\veps|
	\le c_2 \cdot \bigg\{ \int_t^{t+1} \io \frac{|\nas\veps|^2}{\Teps} \bigg\}^\frac{1}{2} 
		\cdot \bigg\{ \int_t^{t+1} \io \Teps \bigg\}^\frac{1}{2}
	\le c_2 \sqrt{c_1} \cdot \bigg\{ \int_t^{t+1} \io \frac{|\nas\veps|^2}{\Teps} \bigg\}^\frac{1}{2} 
  \eas
  for all $t>0$ and $\eps\in (0,1)$.
\qed
We thereby already obtain the claimed decay property:\abs
\proofc of Proposition \ref{prop56}. \quad
  We take $c_1>0$ such that in accordance with Lemma \ref{lem58} we have
  \bas
	\int_t^{t+1} \io |\veps| \le c_1 \cdot \bigg\{ \int_t^{t+1} \io \frac{|\nas\veps|^2}{\Teps} \bigg\}^\frac{1}{2}
	\qquad \mbox{for all $t>0$ and } \eps\in (0,1),
  \eas
  and given $\eta>0$ we may then rely on Lemma \ref{lem45} to fix $T(\eta)>0$ and 
  $(\eps^\star(\eta,t))_{t>T(\eta)} \subset (0,1)$
  in such a way that if $t>T(\eta)$, then
  \bas
	c_1 \cdot \bigg\{ \int_t^{t+1} \io \frac{|\nas\veps|^2}{\Teps} \bigg\}^\frac{1}{2}
	\le \eta
	\qquad \mbox{for all } \eps\in (\eps_j)_{j\in\N} \cap (0,\eps^\star(\eta,t)).
  \eas
  For any such $t$, we thus have
  \bas
	\int_t^{t+1} \io |\veps| \le \eta
	\qquad \mbox{for all } \eps\in (\eps_j)_{j\in\N} \cap (0,\eps^\star(\eta,t)),
  \eas
  so that since 
  Lemma \ref{lem24} and Lemma \ref{lem244}
  particularly assert that $\veps\wto u_t \in L^1((t,t+1);L^1(\Om;\R^n))$ as
  $\eps=\eps_j\searrow 0$, by lower semicontinuity of $L^1$ norms with respect to weak convergence this implies that
  $\int_t^{t+1} \io |u_t| \le \eta$ for all $t>T(\eta)$. As $\eta$ was arbitrary, this proves the claim.
\qed
\mysection{Decay of $u$. Proof of Theorem \ref{theo60}}\label{sect_u}
To complete our large time analysis of (\ref{0}), in this section we address the setting from Theorem \ref{theo38}
and may then draw on the $L^1$ stabilization property of $\Theta$, as asserted by Theorem \ref{theo55} in this case,
to achieve the major part of our reasoning toward decay of $u$ by characterizing corresponding $\omega$-limits as follows.
\begin{lem}\label{lem59}
  Assume (\ref{38.1}) and (\ref{fg_decay}).
  Then there exists a null set
  $N\subset (0,\infty)$
  with the property that whenever $(t_k)_{k\in\N} \subset (0,\infty)\sm N$ and $u_\infty\in L^2(\Om;\R^n)$ are such that
  $t_k\to\infty$ and
  \be{59.1}
	u(\cdot,t_k) \to u_\infty
	\qquad \mbox{in } L^2(\Om;\R^n)
  \ee
  as $k\to\infty$, it follows that
  \be{59.2}
	u_\infty =0
	\qquad \mbox{a.e.~in } \Om.
  \ee
\end{lem}
\proof
  According to Lemma \ref{lem244}
  and the Fubini-Tonelli theorem when combined with Lemma \ref{lem6}, it follows that for each $t>0$ we have
  \be{59.3}
	\ueps(\cdot,t) \to u(\cdot,t)
	\quad \mbox{in } L^2(\Om;\R^n)
	\qquad \mbox{as } \eps=\eps_j\searrow 0,
  \ee
  and that there exist $c_1>0$ as well as a null set $N\subset (0,\infty)$
  \be{59.4}
	u(\cdot,t) \in W_0^{1,2}(\Om;\R^n)
	\quad \mbox{with} \quad
	\|u(\cdot,t)\|_{W^{1,2}(\Om)} \le c_1
	\qquad \mbox{for all } t\in (0,\infty)\sm N.
  \ee
  Assuming that $(t_k)_{k\in\N} \subset (0,\infty)\sm N$ and $u_\infty\in L^2(\Om;\R^n)$ satisfy $t_k\to\infty$ and (\ref{59.1})
  as $k\to\infty$, we first observe that due to (\ref{59.4}) we have
  \be{59.5}
	u_\infty \in W_0^{1,2}(\Om).
  \ee
  To see that actually $u_\infty$ must be trivial, we fix a nonnegative $\zeta\in C_0^\infty(\R)$ such that
  $\supp \zeta \subset (0,1)$ and $\int_{\R} \zeta=1$, and taking any $\vp\in C_0^\infty(\Om;\R^n)$ we use the first equation 
  in (\ref{0eps}) in verifying that
  \bea{59.6}
	& & \hs{-14mm}
	J_{1\eps}(k) + J_{2\eps}(k)
	:= - \int_0^\infty \io \zeta'(t-t_k) \veps(x,t)\cdot\vp(x) dxdt
	+ \eps \int_0^\infty \io \zeta(t-t_k) \veps(x,t) \cdot \Del^{2m} \vp(x,t) dxdt \nn\\
	&=& - \int_0^\infty \io \zeta(t-t_k) \lan \D:\nas\veps(x,t),\na\vp(x)\ran dxdt \nn\\
	& & - \int_0^\infty \io \zeta(t-t_k) \lan \C:\nas\ueps(x,t),\na\vp(x)\ran dxdt \nn\\
	& & + \int_0^\infty \io \zeta(t-t_k) \Teps(x,t) \lan \B,\na\vp(x)\ran dxdt \nn\\
	& & + \int_0^\infty \io \zeta(t-t_k) \feps(x,t) \cdot\vp(x) dxdt \nn\\
	&=:& J_{3\eps}(k)+...+J_{6\eps}(k)
	\qquad \mbox{for all $\eps\in (0,1)$ and } k\in\N.
  \eea
  Here, (\ref{244.4}) implies that for each $k\in\N$ we have
  \be{59.7}
	J_{1\eps}(k) \to J_1(k):= - \int_0^\infty \io \zeta'(t-t_k) u_t(x,t)\cdot \vp(x) dxdt
  \ee
  and
  \be{59.8}
	J_{2\eps}(k) \to 0
  \ee
  as $\eps=\eps_j\searrow 0$, whereas by (\ref{244.44}), (\ref{244.6}) and (\ref{fgec}),
  \be{59.9}
	J_{3\eps}(k) \to J_3(k):= - \int_0^\infty \io \zeta(t-t_k) \lan \D:\nas u_t(x,t),\na\vp(x)\ran dxdt
  \ee
  and
  \be{59.10}
	J_{5\eps}(k)\to J_5(k):=
	\int_0^\infty \io \zeta(t-t_k) \Theta(x,t) \lan \B,\na\vp(x)\ran dxdt
  \ee
  as well as
  \be{59.11}
	J_{6\eps}(k) \to J_6(k)
	:= \int_0^\infty \io \zeta(t-t_k) f(x,t)\cdot\vp(x) dxdt
  \ee
  for all $k\in\N$ as $\eps=\eps_j\searrow 0$.\abs
  In order to similarly treat the second summand on the right of (\ref{59.6}), we first note that due to Lemma \ref{lem58}
  and Lemma \ref{lem45}, given any $\eta>0$ we can find $k_0(\eta)\in\N$ such that for all $k>k_0(\eta)$ there exists
  $\eps^\star(\eta,k)\in (0,1)$
  fulfilling 
  \be{59.12}
	\int_{t_k}^{t_k+1} \io |\veps| 
	\le
	\frac{\eta}{2c_2}
	\qquad \mbox{for each } \eps\in (\eps_j)_{j\in\N} \cap (0,\eps^\star(\eta,k)),
  \ee
  where we have set
  \be{59.122}
	c_2:=\frac{1}{2} \cdot \bigg\{ \sum_{\mathbf{i},\mathbf{j},\mathbf{k},\mathbf{l}=1}^n |\C_{\mathbf{ijkl}}| \bigg\}
	\cdot c_3 \cdot 
	\max_{(\mathbf{i},\mathbf{j},\mathbf{k},\mathbf{l})\in\{1,...,n\}^4} 
	\Big\{ \|\pa_{\mathbf{lj}} \vp_{\mathbf{i}}\|_{L^\infty(\Om)} + \|\pa_{\mathbf{kj}} \vp_{\mathbf{i}}\|_{L^\infty(\Om)} \Big\}
	\qquad \mbox{with} \qquad
	c_3:=\|\zeta\|_{L^\infty(\R)}.
  \ee
  In the identity 
  \bea{59.13}
	J_{4\eps}(k)
	&=& - \frac{1}{2} \sum_{\mathbf{i},\mathbf{j},\mathbf{k},\mathbf{l}=1}^n \C_{\mathbf{ijlk}} \int_0^\infty \io \zeta(t-t_k) 
		\cdot \Big\{ \pa_{\mathbf{l}}(\ueps)_{\mathbf{k}}(x,t) + \pa_{\mathbf{k}} (\ueps)_{\mathbf{l}}(x,t) \Big\}
		\cdot \pa_{\mathbf{j}} \vp_{\mathbf{i}}(x) dxdt \nn\\
	&=& \frac{1}{2} \sum_{\mathbf{i},\mathbf{j},\mathbf{k},\mathbf{l}=1}^n \C_{\mathbf{ijlk}} \int_0^\infty \io \zeta(t-t_k)
		\cdot \Big\{ (\ueps)_{\mathbf{k}}(x,t) \pa_{\mathbf{lj}} \vp_{\mathbf{i}}(x)
		+ (\ueps)_{\mathbf{l}}(x,t) \pa_{\mathbf{kj}} \vp_{\mathbf{i}}(x) \Big\} dxdt \nn\\
	&=& \frac{1}{2} \sum_{\mathbf{i},\mathbf{j},\mathbf{k},\mathbf{l}=1}^n \C_{\mathbf{ijkl}} \int_0^\infty \io \zeta(t-t_k)
		\cdot \Big\{ (\ueps)_{\mathbf{k}}(x,t_k) \pa_{\mathbf{lj}} \vp_{\mathbf{i}}(x)
		+ (\ueps)_{\mathbf{l}}(x,t_k) \pa_{\mathbf{kj}} \vp_{\mathbf{i}}(x) \Big\} dxdt \nn\\
	& & + \frac{1}{2} \sum_{\mathbf{i},\mathbf{j},\mathbf{k},\mathbf{l}=1}^n \C_{\mathbf{ijkl}} \int_0^\infty \io \zeta(t-t_k)
		\cdot \bigg\{ \int_{t_k}^t (\veps)_{\mathbf{k}}(x,s) ds \bigg\} \cdot \pa_{\mathbf{lj}} \vp_{\mathbf{i}}(x) dxdt \nn\\
	& & + \frac{1}{2} \sum_{\mathbf{i},\mathbf{j},\mathbf{k},\mathbf{l}=1}^n \C_{\mathbf{ijkl}} \int_0^\infty \io \zeta(t-t_k)
		\cdot \bigg\{ \int_{t_k}^t (\veps)_{\mathbf{l}}(x,s) ds \bigg\} \cdot \pa_{\mathbf{kj}} \vp_{\mathbf{i}}(x) dxdt,
  \eea
  as obtained for each $k\in\N$ and $\eps\in (0,1)$ upon an integration by parts due to the fact that $u_{\eps t}=\veps$
  for all $\eps\in (0,1)$, we can use this inequality (\ref{59.12}) to estimate the second to last contribution according to
  \bas
	& & \hs{-20mm}
	\bigg| \frac{1}{2} \sum_{\mathbf{i},\mathbf{j},\mathbf{k},\mathbf{l}=1}^n \C_{\mathbf{ijkl}} \int_0^\infty \io \zeta(t-t_k)
		\cdot \bigg\{ \int_{t_k}^t (\veps)_{\mathbf{k}}(x,s) ds \bigg\} \cdot \pa_{\mathbf{lj}} \vp_{\mathbf{i}}(x) dxdt
	\bigg| \\
	&\le& c_2 \int_{t_k}^{t_k+1} \io \int_{t_k}^t |\veps(x,s)| dsdxdt \\
	&\le& c_2 \int_{t_k}^{t_k+1} \io |\veps| \\
	&\le& \frac{\eta}{2}
	\qquad \mbox{for all $k>k_0(\eta)$ and } \eps\in (\eps_j)_{j\in\N} \cap (0,\eps^\star(\eta,k)).
  \eas
  After a similar handling of the rightmost summand in (\ref{59.13}), using (\ref{244.5}) and (\ref{59.3}) together with
  our restriction that $t_k\not\in N$ we thus infer that for
  \be{59.133}
	\hs{-4mm}
	J_4(k):=
	\frac{1}{2} \sum_{\mathbf{i},\mathbf{j},\mathbf{k},\mathbf{l}=1}^n \C_{\mathbf{ijlk}} \int_0^\infty \io \zeta(t-t_k) \cdot
		\Big\{ u_{\mathbf{k}}(x,t_k) \pa_{\mathbf{lj}}\vp_{\mathbf{i}}(x) 
			+ u_{\mathbf{l}}(x,t_k) \pa_{\mathbf{kj}}\vp_{\mathbf{i}}(x)\Big\} dxdt,
	\quad k\in\N,
  \ee
  we have
  \bas
	\limsup_{\eps=\eps_j\searrow 0} \big| J_{4\eps}(k)-J_4(k)\big|
	\le \frac{\eta}{2} + \frac{\eta}{2} = \eta
	\qquad \mbox{for all } k>k_0(\eta).
  \eas
  When combined with (\ref{59.6})-(\ref{59.11}), this shows that
  \bea{59.14}
	\big|J_4(k)\big|
	&\le& \limsup_{\eps=\eps_j\searrow 0} \big| J_{4\eps}(k)-J_4(k)\big| 
	+ \limsup_{\eps=\eps_j\searrow 0} \big|J_{4\eps}(k)\big| \nn\\
	&\le& \eta +
	\limsup_{\eps=\eps_j\searrow w} \big| J_{1\eps}(k)+J_{2\eps}(k)-J_{3\eps}(k)-J_{5\eps}(k)-J_{6\eps}(k)\big| \nn\\
	&\le& \eta +
	\bigg| \int_0^\infty \io \zeta'(t-t_k) u_t(\cdot,t)\cdot \vp(x) dxdt \bigg| \nn\\
	& & + \bigg| \int_0^\infty \io \zeta(t-t_k) \lan \D:\nas u_t(x,t),\na\vp(x)\ran dxdt \bigg| \nn\\
	& & + \bigg| \int_0^\infty \io \zeta(t-t_k) \Theta(x,t) \lan \B,\na\vp(x)\ran dxdt \bigg| \nn\\
	& & + \bigg| \int_0^\infty \io \zeta(t-t_k) f(x,t) \cdot\vp(x) dxdt \bigg|
	\qquad \mbox{for all } k>k_0(\eta).
  \eea
  Here, Proposition \ref{prop56} ensures that
  \bea{59.15}
	\bigg| \int_0^\infty \io \zeta'(t-t_k) u_t(\cdot,t)\cdot \vp(x) dxdt \bigg|
	&\le& \|\zeta'\|_{L^\infty(\R)} \|\vp\|_{L^\infty(\Om)} \int_{t_k}^{t_k+1} \io |u_t| \nn\\
	&\to& 0
	\qquad \mbox{as } k\to\infty,
  \eea
  while after an integration by parts in the style of that in (\ref{59.13}), the same token applies so as to guarantee that,
  again by (\ref{59.122}),
  \bea{59.16}
	& & \hs{-12mm}
	\bigg| \int_0^\infty \io \zeta(t-t_k) \lan \D:\nas u_t(x,t),\na\vp(x)\ran dxdt \bigg| \nn\\
	&=& \bigg| \frac{1}{2} \sum_{\mathbf{i},\mathbf{j},\mathbf{k},\mathbf{l}=1}^n \D_{\mathbf{ijlk}} 
	\int_0^\infty \io \zeta(t-t_k) \cdot
	\Big\{ (u_t)_{\mathbf{k}}(x,t) \pa_{\mathbf{lj}} \vp_{\mathbf{i}}(x) 
	+ (u_t)_{\mathbf{l}}(x,t) \pa_{\mathbf{kj}} \vp_{\mathbf{i}}(x) \Big\} dxdt \bigg| \nn\\
	&\le& 2 \cdot 
	\frac{1}{2} \cdot \bigg\{ \sum_{\mathbf{i},\mathbf{j},\mathbf{k},\mathbf{l}=1}^n |\D_{\mathbf{ijkl}}| \bigg\}
	\cdot c_3 \cdot 
	\max_{(\mathbf{i},\mathbf{j},\mathbf{k},\mathbf{l})\in\{1,...,n\}^4} 
	\Big\{ \|\pa_{\mathbf{lj}} \vp_{\mathbf{i}}\|_{L^\infty(\Om)} + \|\pa_{\mathbf{kj}} \vp_{\mathbf{i}}\|_{L^\infty(\Om)} \Big\}
	\cdot \int_{t_k}^{t_k+1} \io |u_t| \nn\\
	&\to& 0
	\qquad \mbox{as } k\to\infty.
  \eea
  In the crucial second to last integral in (\ref{59.14}), we use that under the present assumptions we may rely on part ii) of
  Theorem \ref{theo55} to make sure that with $\Tinf$ as provided there we have 
  \bas
	\int_{t_k}^{t_k+1} |\Theta-\Tinf| \to 0
	\qquad \mbox{as } k\to\infty, 
  \eas 
  and that thus, since an integration by parts shows that
  \bas
	\io \lan \B,\na\vp\ran =0
  \eas
  due to the fact that $\vp=0$ on $\pO$, it follows that also
  \bea{59.17}
	& & \hs{-20mm}
	\bigg| \int_0^\infty \io \zeta(t-t_k) \Theta(x,t) \lan \B,\na\vp(x)\ran dxdt \bigg| \nn\\
	&=& \bigg| \Tinf \cdot \bigg\{ \int_0^\infty \zeta(t-t_k) dt \bigg\} \cdot \io \lan \B,\na\vp(x) \ran dx \nn\\
	& & \hs{6mm}
	+ \int_0^\infty \io \zeta(t-t_k) \cdot \big\{ \Theta(x,t)-\Tinf\big\} \lan \B,\na\vp(x)\ran dxdt \bigg| \nn\\
	&\le& c_3 |\B| \cdot \|\na\vp\|_{L^\infty(\Om)} \int_{t_k}^{t_k+1} \io |\Theta-\Tinf| \nn\\
	&\to& 0
	\qquad \mbox{as } k\to\infty.
  \eea
  Since, finally, (\ref{fg_decay}) particularly warrants that
  \bas
	\bigg| \int_0^\infty \io \zeta(t-t_k) f(x,t) \cdot\vp(x) dxdt \bigg|
	&\le& c_3 \|\vp\|_{L^2(\Om)} \int_{t_k}^{t_k+1} \|f(\cdot,t)\|_{L^2(\Om)} dt \nn\\
	&\to& 0
	\qquad \mbox{as } k\to\infty,
  \eas
  collecting (\ref{59.14})-(\ref{59.17}) we infer that
  \bas
	\limsup_{k\to\infty} \big|J_4(k)\big| \le \eta,
  \eas  
  and that therefore, as $\eta>0$ was arbitrary,
  \be{59.18}
	J_4(k) \to 0
	\qquad \mbox{as } k\to\infty.
  \ee
  On the other hand, in view of our hypothesis in (\ref{59.1}) we know from (\ref{59.133}) that due to the identity 
  $\int_{\R} \zeta=1$,
  \bas
	J_4(k)
	&=& \frac{1}{2} \sum_{\mathbf{i},\mathbf{j},\mathbf{k},\mathbf{l}=1}^n \C_{\mathbf{ijlk}} 
	\cdot \bigg\{ \int_0^\infty \zeta(t-t_k) dt \bigg\} \cdot \io 
	\Big\{ u_{\mathbf{k}}(x,t_k) \pa_{\mathbf{lj}}\vp_{\mathbf{i}}(x) + u_{\mathbf{l}}(x,t_k) \pa_{\mathbf{kj}}\vp_{\mathbf{i}}(x)
	\Big\} dx \nn\\
	&\to& \frac{1}{2} \sum_{\mathbf{i},\mathbf{j},\mathbf{k},\mathbf{l}=1}^n \C_{\mathbf{ijlk}} 
	\cdot \io
	\Big\{ (u_\infty)_{\mathbf{k}}(x) \pa_{\mathbf{lj}}\vp_{\mathbf{i}}(x) 
	+ (u_\infty)_{\mathbf{l}}(x) \pa_{\mathbf{kj}}\vp_{\mathbf{i}}(x)
	\Big\} dx
  \eas
  as $k\to\infty$.
  But since (\ref{59.5}) allows us to integrate by parts here without any appearance of boundary integrals, in combination
  with (\ref{59.18}) this means that
  \bas
	\io \lan\C:\nas u_\infty,\na\vp\ran =0
	\qquad \mbox{for all } \vp\in C_0^\infty(\Om;\R^n).
  \eas
  As $\lan\C:\nas\psi^{(1)},\na\psi^{(2)}\ran  = \lan\C:\nas\psi^{(1)},\nas\psi^{(2)}\ran$ a.e.~in $\Om$
  whenever $\psi^{(i)}\in W^{1,2}(\Om;\R^n)$ for $i\in\{1,2\}$, by density of $C_0^\infty(\Om;\R^n)$ in $W_0^{1,2}(\Om;\R^n)$
  this implies that $\io \lan\D;\nas u_\infty,\nas u_\infty\ran =0$. 
  According to (\ref{DC_lower}), we conclude that $\io |\nas u_\infty|^2=0$ and that thus (\ref{59.2}) holds.
\qed
Our asymptotic analysis can thus be completed:\abs
\proofc of Theorem \ref{theo60}.\quad
  Since $u$ is continuous on $[0,\infty)$ as an $L^2(\Om;\R^n)$-valued function by Theorem \ref{theo38},
  and since Lemma \ref{lem6} together with Lemma \ref{lem244} ensures the existenc of a null set $N\subset (0,\infty)$ such that
  $(u(\cdot,t))_{t\in (0,\infty)\sm N}$ is bounded in $W^{1,2}(\Om;\R^n)$ and hence relatively compact in $L^2(\Om;\R^n)$,
  the claim can readily be derived from the characterization of corresponding $\omega$-limits achieved in Lemma \ref{lem59},
  and from the fact that complemets of null sets are dense in $(0,\infty)$.
\qed

\newpage

\mysection{Appendix}
\subsection{Consistency of the solution concepts}
As announced near Definition \ref{dw}, let us here make sure that generalized solutions, in the sense
specified there, in fact must be classical if they are additionally required to enjoy suitable further regularity properties.
We note that the following statement in this direction imposes fairly mild hypotheses only, inter alia not requiring quantities
such as $\Theta_t$ or $\Del\Theta$ to be continuous in all of $\bom\times [0,\infty)$.
\begin{prop}\label{prop_dw}
  Let $D>0$, 
  $\D=(\D_{\mathbf{ijkl}})_{(\mathbf{i}, \mathbf{j}, \mathbf{k},\mathbf{l}) 
		\in \{1,...,n\}^4} \in \R^{n\times n \times n \times n}$ and
  $\C=(\C_{\mathbf{ijkl}})_{(\mathbf{i}, \mathbf{j}, \mathbf{k},\mathbf{l}) 
		\in \{1,...,n\}^4} \in \R^{n\times n \times n \times n}$ 
  satisfy (\ref{DCsym}), let
  $\B=(\B_{\mathbf{ij}})_{(\mathbf{i},\mathbf{j}) \in \{1,...,n\}^2} \in \R^{n\times n}_{sym}$,
  $\kappa\in C^0([0,\infty))$,
  $f\in C^0(\bom\times [0,\infty);\R^n)$ and $g\in C^0(\bom\times [0,\infty))$,
  $u_0\in C^1(\bom;\R^n)$, $u_{0t} \in C^0(\bom;\R^n)$ and $\Theta_0\in C^0(\bom)$ be such that
  $u_0=0$ on $\pO$ and $\Theta_0\ge 0$, and suppose that
  \be{p1}
	\lbal
	u \in C^0(\bom\times [0,\infty);\R^n)
	\cap C^2(\bom\times (0,\infty);\R^n)
	\cap C^0([0,\infty);W_0^{1,2}(\Om;\R^n)) 
	\qquad \mbox{and} \\[1mm]
	\Theta\in C^0(\bom\times [0,\infty)) \cap C^{2,1}(\bom\times (0,\infty))
	\ear
  \ee
  are such that
  \be{p2}
	u_t \in C^{1,0}(\bom\times [0,\infty);\R^n) \cap C^{2,1}(\bom\times (0,\infty);\R^n),
  \ee
  that $\Theta\ge 0$ in $\Om\times (0,\infty)$, and that $(u,\Theta)$ forms a generalized solution of (\ref{0}) in the sense
  of Definition \ref{dw}.
  Then $(u,\Theta)$ solves (\ref{0}) in the classical sense.
\end{prop}
\proof
  Letting $a$ and $\mu$ be as in Definition \ref{dw}, according to the regularity properties in (\ref{p2}) and (\ref{p1}) as
  well as the evident fact that $u=u_t=0$ on $\pO\times (0,\infty)$ by (\ref{p1}),
  given $\vp\in C_0^\infty(\Om\times [0,\infty);\R^n)$
  we may integrate by parts in (\ref{wu}) and use 
  that $\lan \D:\nas u_t,\na\vp \ran=\lan \D:\nas u_t,\nas\vp \ran$ and
  $\lan \C:\nas u,\na\vp\ran = \lan \C:\nas u,\nas\vp\ran$ as well as
  $\lan \B,\na\vp\ran=\lan \B,\nas\vp\ran$ in $\Om\times (0,\infty)$
  to see that
  \bea{p22}
	& & \hs{-30mm}
	- \int_0^\infty \io u_t \cdot\vp_t 
	- \io \big\{ u(\cdot,0)-u_0\big\} \cdot \vp_t(\cdot,0)
	- \io u_{0t} \cdot\vp(\cdot,0) \nn\\
	&=& - \int_0^\infty \io \lan \D:\nas u_t,\nas\vp \ran
	- \int_0^\infty \io \lan \C:\nas u,\nas\vp\ran \nn\\
	& & + \int_0^\infty \io \Theta \lan \B,\nas\vp\ran
	+ \int_0^\infty \int_{\bom} \lan \B,\nas\vp\ran d\mu
	+ \int_0^\infty \io f\cdot\vp,
  \eea
  which by an approximation argument extends so as to hold actually for all compactly supported 
  $\vp\in C^1(\bom\times [0,\infty);\R^n)$ fulfilling $\vp=0$ on $\pO\times (0,\infty)$.
  Given $\psi\in C_0^\infty(\Om;\R^n)$, for each $\del>0$ we may hence choose
  $\vp(x,t):=\frac{1}{2\del} (\del-t)_+^2 \psi(x)$, $(x,t)\in\bom\times [0,\infty)$, to obtain the identity
  \be{p3}
	\frac{1}{\del} \int_0^\del \io (\del-t) u_t\cdot\psi
	+ \io \big\{ u(\cdot,0)-u_0\big\} \cdot\psi
	- \frac{\del}{2} \io u_{0t} \cdot\psi 
	= \frac{1}{2\del} \int_0^\del (\del-t)^2 h(t) dt,
  \ee
  where 
  $h(t):=\io \big\{ - \lan \D:\nas u_t,\nas\psi\ran -  \lan C:\nas u,\nas\psi\ran + \Theta \lan \B,\nas \psi\ran + f\cdot \psi \big\}
  + \int_{\bom} \lan \B,\nas\psi\ran d\mu(t)$, $t>0$,
  satisfies 
  $h\in L^\infty_{loc}([0,\infty))$ 
  thanks to (\ref{p2}), (\ref{p1}) and (\ref{w6}).
  Since moreover $u_t\cdot\psi$ is bounded in $\Om\times (0,1)$ by (\ref{p2}), we can therefore find $c_1>0$ and $c_2>0$ such that
  \be{p4}
	\bigg| \frac{1}{2\del} \int_0^\del (\del-t)^2 h(t) dt \bigg| \le c_1 \del^2
	\quad \mbox{and} \quad
	\bigg| \frac{1}{\del} \int_0^\del \io (\del-t) u_t\cdot\psi \bigg| \le c_2\del
	\qquad \mbox{for all } \del\in (0,1),
  \ee
  whence taking $\del\searrow 0$ in (\ref{p3}) shows that $\io \big\{ u(\cdot,0)-u_0 \big\} \cdot\psi=0$ and thus
  \be{p5}
	u(\cdot,0)=u_0
	\qquad \mbox{in } \Om,
  \ee
  as $\psi\in C_0^\infty(\Om;\R^n)$ was arbitrary.\abs
  We may thereupon divide (\ref{p3}) by $\del$ and again rely on the first inequality in (\ref{p4}) to see that since
  \bas
	& & \hs{-30mm}
	\bigg| \frac{1}{\del^2} \int_0^\del (\del-t) u_t\cdot\psi
	- \frac{1}{2} \io u_t(\cdot,0)\cdot\psi \bigg| \\
	&=& \bigg| \frac{1}{\del^2} \int_0^\del \io (\del-t) \big\{ u_t(x,t) - u_t(x,0)\big\}\cdot\psi(x) dxdt \bigg| \\
	&\le& \bigg\{ \sup_{(x,t)\in \Om\times (0,\del)} \big| u_t(x,t)-u_t(x,0)\big| \bigg\} 
		\cdot \|\psi\|_{L^\infty(\Om)} \cdot \frac{1}{\del^2} \int_0^\del (\del-t) dt \\[2mm]
	&\to& 0
	\qquad \mbox{as } \del\searrow 0
  \eas
  by continuity of $u_t$ in $\Om\times \{0\}$, we also have
  $\frac{1}{2} \io \big\{ u_t(\cdot,0)-u_{0t}\big\}\cdot \psi=0$ for any such $\psi$, and hence
  \be{p6}
	u_t(\cdot,0)=u_{0t}
	\qquad \mbox{in } \Om.
  \ee
  For $t_0>0, t_1>t_0$ and $\del\in (0,\frac{t_0}{2})$, we next let $\zeta\in W^{1,\infty}((0,\infty))$ be the piecewise 
  linear continuous function on $[0,\infty)$ satisfying $\zeta\equiv 0$ on $[0,t_0-\del]$, $\zeta'\equiv \frac{1}{\del}$
  on $(t_0-\del,t_0)$, $\zeta\equiv 1$ on $[t_0,t_1]$, $\zeta'\equiv -\frac{1}{\del}$ on $(t_1,t_1+\del)$ and
  $\zeta\equiv 0$ on $[t_1+\del,\infty)$.
  Taking $(\zeta_j)_{j\in\N} \subset C_0^\infty((0,\infty))$ such that $\supp \zeta_j \subset (t_0-2\del,t_1+2\del)$
  for all $j\in\N$ and $\zeta_j \wsto \zeta$ in $W^{1,\infty}((0,\infty))$ as $j\to\infty$, on applying (\ref{p22}) to
  $\vp(x,t):=\zeta_j(t) u_t(x,t)$, $(x,t)\in\bom\times [0,\infty)$, and letting $j\to\infty$, we obtain that
  \bea{p7}
	& & \hs{-20mm}
	- \int_{t_0-\del}^{t_1+\del} \io u_t \cdot (\zeta u_{tt} + \zeta' u_t) \nn\\
	&=& - \int_{t_0-\del}^{t_1+\del} \io \zeta \lan \D:\nas u_t,\nas u_t\ran
	- \int_{t_0-\del}^{t_1+\del} \io  \zeta \lan \C:\nas u, \nas u_t \ran \nn\\
	& & + \int_{t_0-\del}^{t_1+\del} \io  \zeta \Theta \lan \B,\nas u_t\ran
	+ \int_{t_0-\del}^{t_1+\del} \int_{\bom} \zeta \lan \B,\nas u_t\ran d\mu
	+ \int_{t_0-\del}^{t_1+\del} \io \zeta f\cdot u_t,
  \eea
  where by the chain rule and an integration by parts,
  \bas
	- \int_{t_0-\del}^{t_1+\del} \io u_t \cdot (\zeta u_{tt} + \zeta' u_t)
	&=& - \frac{1}{2} \int_{t_0-\del}^{t_1+\del} \io  \zeta \cdot (|u_t|^2)_t
	- \int_{t_0-\del}^{t_1+\del} \io \zeta' |u_t|^2
	= - \frac{1}{2} \int_{t_0-\del}^{t_1+\del} \io \zeta' |u_t|^2 \\
	&=& \frac{1}{2\del} \int_{t_1}^{t_1+\del} \io |u_t|^2
	- \frac{1}{2\del} \int_{t_0-\del}^{t_0} |u_t|^2 \\
	&\to& \frac{1}{2} \io |u_t|^2(\cdot,t_1) - \frac{1}{2} \io |u_t|^2(\cdot,t_0)
	\qquad \mbox{as } \del\searrow 0.
  \eas
  Likewise,
  \bas
	- \int_{t_0-\del}^{t_1+\del} \io \zeta \lan \C:\nas u,\nas u_t\ran
	\to - \frac{1}{2} \io \lan \C:\nas u(\cdot,t_1),\nas u(\cdot,t_1)\ran
	+ \frac{1}{2} \io \lan \C:\nas u(\cdot,t_0),\nas u(\cdot,t_0)\ran
  \eas
  as $\del\searrow 0$, so that letting $\del\searrow 0$ and then $t_0\searrow 0$ in (\ref{p7}), 
  thanks to (\ref{p6}), (\ref{p5}) and the
  $W^{1,2}\times L^2$-valued continuity of $0\le t\mapsto (u,u_t)(\cdot,t)$, as implied by (\ref{p1}) and (\ref{p2}), 
  we conclude that for all $t>0$,
  \bea{p8}
	& & \hs{-20mm}
	\frac{1}{2} \io |u_t(\cdot,t)|^2
	+ \frac{1}{2} \io \lan \C:\nas u(\cdot,t),\nas u(\cdot,t)\ran \nn\\
	&=& \frac{1}{2} \io |u_{0t}|^2
	+ \frac{1}{2} \io \lan \C:\nas u_0,\nas u_0\ran \nn\\
	& & - \int_0^t \io \lan \D:\nas u_t,\nas u_t \ran
	+ \int_0^t \io \Theta\lan \B:\nas u_t\ran
	+ \int_0^t \int_{\bom} \lan \B,\nas u_t\ran d\mu 
	+ \int_0^t \io f\cdot u_t.
  \eea
  We next exploit (\ref{wt}) in a standard manner (cf.~\cite[Lemma 2.1]{win_SIMA2015} for details in a closely related situation)
  to see that whenever $\phi\in C^\infty([0,\infty))$ is such that $\phi'\in C_0^\infty([0,\infty))$,
  $\phi'\ge 0$ and $\phi''\le 0$, we have
  \bas
	\pa_t K^{(\phi)}(\Theta)
	\ge D\phi'(\Theta)\Del\Theta + \phi'(\Theta) \lan \D:\nas u_t,\nas u_t\ran - \Theta \phi'(\Theta) \lan \B,\nas u_t\ran
		+ \phi'(\Theta) g
	\qquad \mbox{in } \Om\times (0,\infty)
  \eas
  and $D\phi'(\Theta)\frac{\partial\Theta}{\partial\nu}\ge 0$ on $\pO\times (0,\infty)$
  as well as $K^{(\phi)}(\Theta(\cdot,0)) \ge K^{(\phi)}(\Theta_0)$ in $\Om$, where $K^{(\phi)}$ is as in (\ref{Kphi}).
  Here, given $M>0$ we choose any nonnegative and nonincreasing $\chi\in C_0^\infty([0,\infty))$ such that $\chi\equiv 1$ on $[0,M]$,
  and set $\phi(\xi):=\int_0^\xi \chi(\sig)d\sig$, $\xi\ge 0$, to thus infer that since $(K^{(\phi)})'\equiv \kappa \phi'$,
  \bas
	\lball
	\kappa(\Theta) \Theta_t \ge D\Del\Theta
	+ \lan \D:\nas u_t,\nas u_t\ran
	- \Theta \lan \B,\nas u_t\ran
	+ g
	\qquad & \mbox{in } \big(\Om\times (0,\infty)\big) \cap \{\Theta<M\}, \\[1mm]
	\frac{\pa\Theta}{\pa\nu} \ge 0
	\qquad & \mbox{on } \big(\pO\times (0,\infty)\big) \cap \{\Theta<M\}, \\[1mm]
	\Theta(\cdot,0)\ge \Theta_0
	\qquad & \mbox{in } \Om\cap \{\Theta<M\},
	\ear
  \eas
  and that thus, since $M$ is arbitrary here, letting
  \be{p91}
	h_1:=\kappa(\Theta)\Theta_t - D\Del\Theta - \lan\D:\nas u_t,\nas u_t\ran + \Theta \lan \B,\nas u_t\ran - g
  \ee
  as well as
  \be{p92}
	h_2:=D\frac{\pa\Theta}{\pa\nu}
	\qquad \mbox{and} \qquad
	h_3:=K(\Theta(\cdot,0))-K(\Theta_0),
  \ee
  with the nondecreasing $K$ as in (\ref{K}),
  defines nonnegative functions $h_1\in C^0(\bom\times (0,\infty)), h_2\in C^0(\pO\times (0,\infty))$ and $h_3\in C^0(\bom)$.
  In particular, an integration shows that for all $t_0>0$ and any $t>t_0$ we have
  \bas
	\io K(\Theta(\cdot,t)) - \io K(\Theta(\cdot,t_0))
	&=& D \int_{t_0}^t \io \Del\Theta
	+ \int_{t_0}^t \io \lan \D:\nas u_t,\nas u_t\ran
	- \int_{t_0}^t \io \Theta \lan \B,\nas u_t\ran \\
	& & + \int_{t_0}^t \io g
	+ \int_{t_0}^t \io h_1 \\
	&=& \int_{t_0}^t \int_{\pO} h_2
	+ \int_{t_0}^t \io \lan \D:\nas u_t,\nas u_t\ran
	- \int_{t_0}^t \io \Theta \lan \B,\nas u_t\ran  \\
	& & + \int_{t_0}^t \io g
	+ \int_{t_0}^t \io h_1 
  \eas
  and hence, by continuity of $\Theta$ in $\bom\times \{0\}$, and by the nonnegativity of $h_1$ and $h_2$ as well as Beppo Levi's
  theorem,
  \bas
	\io K(\Theta(\cdot,t)) 
	&=& \io K(\Theta_0)
	+ \int_0^t \io \lan \D:\nas u_t,\nas u_t\ran
	- \int_0^t \io \Theta \lan \B,\nas u_t\ran  \\
	& & + \int_0^t \io g
	+ \int_0^t \io h_1 
	+ \int_0^t \int_{\pO} h_2
	+ \io h_3
	\qquad \mbox{for all } t>0.
  \eas
  In conjunction with (\ref{p8}), this shows that if we take $\F$ and $\F_0$ from (\ref{F}) and (\ref{F0}), then
  \bas
	\F(t) 
	&=& \F_0
	+ \int_0^t \io f\cdot u_t
	+ \int_0^t \io g \nn\\
	& & + \int_0^t \int_{\bom} \lan \B,\nas u_t\ran d\mu 
	+ \int_0^t \io h_1 + \int_0^t \int_{\pO} h_2
	+ \io h_3
	\qquad \mbox{for all } t>0,
  \eas
  while, on the other hand, (\ref{wF}) says that
  \bas
	\F(t) + a \int_{\bom} d\mu(t) \le \F_0 + \int_0^t \io f\cdot u_t + \int_0^t \io g
	\qquad \mbox{for a.e.~} t>0.
  \eas
  Noting that for each $T>0$ the number $c_3(T):=|\B| \cdot \|\nas u_t\|_{L^\infty(\Om\times (0,T))}$ is finite due to (\ref{p2}),
  by combining these relations and using the nonnegativity of the measure $\mu$ we obtain that
  \bas
	a \int_{\bom} d\mu(t)
	+ \int_0^t \io h_1 + \int_0^t \int_{\pO} h_2 + \io h_3
	\le c_3(T) \int_0^t \int_{\bom} d\mu(s) ds
	\qquad \mbox{for a.e.~} t\in (0,T).
  \eas
  Again since $\mu$ is nonnegative, and since 
  $(0,T)\ni t \mapsto \int_{\bom} d\mu(t)$ lies in $L^\infty((0,T))$ by (\ref{w6}),
  we may draw on a Gr\"onwall lemma to infer that $\mu(t)=0$ for a.e.~$t\in (0,T)$, and that consequently also
  $h_1\equiv 0$ in $\Om\times (0,T)$, $h_2\equiv 0$ on $\pO\times (0,T)$ and $h_3\equiv 0$ in $\Om$.
  As $T>0$ was arbitrary, from (\ref{p91}) and (\ref{p92}) we conclude that the second sub-problem of (\ref{0}) is solved
  classically, while using that $\mu=0$ we may thereupon integrate by parts in (\ref{wu}) and recall (\ref{p5}) and (\ref{p6})
  to see that also the first part of (\ref{0}) is satisfied in the pointwise sense.
\qed
\subsection{Local existence. Proof of Lemma \ref{lem1}}
Inter alia since the second equation in (\ref{0eps}) does not involve diffusion, 
our statement on local existence and extensibility in the approximate problems claimed in Lemma \ref{lem1} 
seems not covered by standard parabolic literature.
Let us therefore, finally, design a fixed point framework capable of providing the intended conclusion 
in an essentially self-contained manner:\abs
\proofc of Lemma \ref{lem1}.\quad
  We let $A=A(\eps)$ denote the self-adjoint realization of $\eps\Del^{2m}$ in $L^2(\Om;\R^n)$, with domain given by
  $D(A):=W^{4m,2}(\Om) \cap W_0^{2m,2}(\Om)$, and then obtain that the corresponding fractional power $A^\frac{1}{2}$
  has the property that $\|A^\frac{1}{2}(\cdot)\|_{L^2(\Om)}$ defines a norm equivalent to $\|\cdot\|_{W^{2m,2}(\Om)}$
  on $D(A^\frac{1}{2})=W_0^{2m,2}(\Om)$.
  Since $m>\frac{n+2}{4}$ and $m\ge 1$, and since thus $W^{2m,2}(\Om) \hra W^{1,\infty}(\Om)$ and $W^{2m,2}(\Om) \hra W^{2,2}(\Om)$,
  we can accordingly fix positive constants $c_i=c_i(\eps)>0$, $i\in\{1,...,4\}$, such that
  \be{1.2}
	\big\| \lan \B,\nas\vp\ran \big\|_{L^\infty(\Om)}
	\le c_1 \|A^\frac{1}{2} \vp\|_{L^2(\Om)}
	\qquad \mbox{for all } \vp\in D(A^\frac{1}{2})
  \ee
  and
  \be{1.3}
	\big\| \lan \D:\nas\vp,\nas\vp\ran\big\|_{L^\infty(\Om)}
	\le c_2 \|A^\frac{1}{2} \vp\|_{L^2(\Om)}^2
	\qquad \mbox{for all } \vp\in D(A^\frac{1}{2})
  \ee
  as well as
  \be{1.4}
	\big\| \div (\D:\nas \vp)\big\|_{L^2(\Om)}
	\le c_3 \|A^\frac{1}{2} \vp\|_{L^2(\Om)}
	\qquad \mbox{for all } \vp\in D(A^\frac{1}{2})
  \ee
  and
  \be{1.5}
	\big\| \div (\C:\nas \vp)\big\|_{L^2(\Om)}
	\le c_4 \|A^\frac{1}{2} \vp\|_{L^2(\Om)}
	\qquad \mbox{for all } \vp\in D(A^\frac{1}{2}),
  \ee
  and we moreover choose $c_5>0$ such that
  \be{1.6}
	\big\| \div (\vp\B)\big\|_{L^2(\Om)}
	\le c_5 \|\na\vp\|_{L^2(\Om)}
	\qquad \mbox{for all } \vp\in W^{1,2}(\Om).
  \ee
  Next, known smoothing properties of the associated analytic semigroup $(e^{-tA})_{t\ge 0}$ ensure that
  \be{1.7}
	\|A^\al e^{-tA} \vp\|_{L^2(\Om)} \le \|A^\al \vp\|_{L^2(\Om)}
	\qquad \mbox{for all $t>0, \al\ge 0$ and } \vp\in D(A^\al),
  \ee
  and that with some $c_6=c_6(\eps)>0$,
  \be{1.8}
	\|A^\frac{1}{2} e^{-tA} \vp\|_{L^2(\Om)} \le c_6 t^{-\frac{1}{2}} \|\vp\|_{L^2(\Om)}
	\qquad \mbox{for all $t\in (0,1)$ and } \vp\in L^2(\Om;\R^n).
  \ee
  We furthermore let
  \be{1.91}
	R\equiv R(\eps):=\|A^\frac{1}{2} v_{0\eps}\|_{L^2(\Om)} + 1
  \ee
  and
  \be{1.9}
	c_7\equiv c_7(\eps):=\|A^\frac{1}{2} u_{0\eps}\|_{L^2(\Om)},
	\
	c_8\equiv c_8(\eps):=\sup_{t>0} \|\feps(\cdot,t)\|_{L^2(\Om)}
	\ \mbox{and} \
	c_9 \equiv c_9(\eps):=\sup_{t>0} \|\geps(\cdot,t)\|_{L^\infty(\Om)}
  \ee
  as well as
  \be{1.10}
	c_{10}\equiv c_{10}(\eps):=\|\Theta_{0\eps}\|_{L^\infty(\Om)}
  \ee
  and
  \be{1.100}
	c_{11}\equiv c_{11}(\eps):= \io |\na \Theta_{0\eps}|^2
	+ \frac{2}{D} \cdot \Big\{ \frac{2c_1^2 c_{10}^2 R^2 |\Om|}{\eps} + 
		\frac{(c_2 R^2 + c_9)^2 |\Om|}{2\eps} \Big\}
  \ee
  and
  \be{1.101}
	c_{12} \equiv c_{12}(\eps) := c_3 R + c_4\cdot (c_7+R) + c_5 c_{11} + c_8,
  \ee
  choose $\lam=\lam(\eps)>0$ large enough such that
  \be{1.111}
	\lam\ge \frac{2c_1 R}{\eps}
	\qquad \mbox{and} \qquad
	\lam\ge \frac{2(c_2 R^2 + c_9)}{c_{10} \eps},
  \ee
  and taking any
  \be{1.122}
	\gamma\in \Big(\frac{1}{2},1\Big)
	\qquad \mbox{and} \qquad
	\vt \in (0,1-\gamma)
  \ee
  we fix $T=T(\eps)\in (0,1]$ such that
  \be{1.134}
	e^{\lam T} \le 2
	\qquad \mbox{and} \qquad
	2c_6 c_{12} T^\frac{1}{2} \le 1.
  \ee
  Then in the Banach space  
  \be{1.15}
	X:=C^\vt([0,T];W_0^{2m,2}(\Om;\R^n)),
  \ee
  we consider that closed bounded convex set
  \be{1.16}
	S:=\Big\{ \vp \in X \ \Big| \ \|A^\frac{1}{2} \vp(\cdot,t)\|_{L^2(\Om)} \le R
	\mbox{ for all } t\in [0,T]\Big\}
  \ee
  and define a mapping $\Phi$ on $S$ as follows:
  Given $\ovv\in S$, we use that our assumption $m>\frac{n+2}{4}$ actually implies that $W^{2m,2}(\Om)$ continuously embeds even into
  $C^{1+\iota_1}(\bom)$ for some $\iota_1\in (0,1)$, in view of (\ref{1.15}) meaning that 
  $\ovv\in C^\vt([0,T];C^{1+\iota_1}(\bom;\R^n))$.
  In particular, this warrants that 
  \be{1.166}
	h_1:=-\lan \B,\nas \ovv\ran
	\qquad \mbox{and} \qquad
	h_2:=\lan \D:\nas\ovv,\nas\ovv\ran + \geps
  \ee
  both lie in $C^{\iota_2,\frac{\iota_2}{2}}(\bom\times [0,T])$ for some $\iota_2\in (0,1)$, so that since
  the inclusion $\na\Theta_{0\eps}\in C_0^\infty(\Om;\R^n)$ guarantees validity of the corresponding first-order compatibility
  condition, in view of the lower estimate $\keps\ge\eps$ and the fact that $h_2\ge 0$, it follows from standard parabolic theory
  (\cite{amann},  \cite{LSU}) that there exists $T_\star=T_\star(\ovv)\in (0,T]$ such that the problem
  \be{1.17}
	\lball
	\keps(\Theta) \Theta_t = D\Del\Theta + h_1(x,t) \Theta + h_2(x,t),
	\qquad & x\in\Om, \ t\in (0,T_\star), \\[1mm]
	\frac{\pa\Theta}{\pa\nu}=0,
	\qquad & x\in\pO, \ t\in (0,T_\star), \\[1mm]
	\Theta(x,0)=\Theta_0(x),
	\qquad & x\in\Om,
	\ear
  \ee
  admits a unique nonnegative classical solution $\Phi_1(\ovv):=\Theta\in C^{2,1}(\bom\times [0,T_\star))$ which is such that
  \be{1.18}
	\mbox{if $T_\star<T$, \quad then \quad}
	\limsup_{t\nearrow T_\star} \|\Theta(\cdot,t)\|_{L^\infty(\Om)} = \infty.
  \ee
  To see that actually $T_\star=T$ and
  \be{1.19}
	\Theta \le 2c_{10}
	\qquad \mbox{in } \Om\times (0,T),
  \ee
  we let
  \be{1.20}
	\ov{\Theta}(x,t):= c_{10} e^{\lam t},
	\qquad x\in\bom, \ t\in [0,T_\star),
  \ee
  and note that thanks to the first condition in (\ref{1.134}),
  \be{1.21}
	\ov{\Theta} \le 2c_{10}
	\qquad \mbox{in } \Om\times (0,T_\star).
  \ee
  Moreover, in line with the inequality $\keps\ge\eps$ and our selections in (\ref{1.166}), (\ref{1.2}), (\ref{1.3}), (\ref{1.9})
  and (\ref{1.16}) we have
  \be{1.22}
	|h_1| \le c_1 \|A^\frac{1}{2}\ovv\|_{L^2(\Om)}
	\le c_1 R
	\quad \mbox{and} \quad
	|h_2| \le c_2\|A^\frac{1}{2}\ovv\|_{L^2(\Om)}^2 + c_9
	\le c_2 R^2 + c_9
	\qquad \mbox{in } \Om\times (0,T_\star),
  \ee
  so that
  \bas
	\keps(\Theta) \ov{\Theta}_t - D\Del\ov{\Theta} - h_1\ov{\Theta} - h_2
	&=& \keps(\Theta) \cdot \lam c_{10} e^{\lam t} - h_1\cdot c_{10} e^{\lam t} - h_2 \\
	&\ge& \eps \cdot \lam c_{10} e^{\lam t} - c_1 R \cdot c_{10} e^{\lam t} - c_2 R^2 - c_9 \\
	&=& (\eps\lam-c_1 R) \cdot c_{10} e^{\lam t} - c_2 R^2 - c_9 \\
	&\ge& \frac{\eps\lam}{2} \cdot c_{10} e^{\lam t}
	- c_2 R^2 - c_9 \\
	&\ge& \frac{\eps\lam}{2} \cdot c_{10} - c_2 R^2 - c_9 \\[2mm]
	&\ge& 0
	\qquad \mbox{in } \Om\times (0,T_\star)
  \eas
  according to the restrictions in (\ref{1.111}).
  As $\frac{\pa\ov{\Theta}}{\pa\nu}=0$ on $\pO\times (0,T_\star)$ and $\ov{\Theta}(\cdot,0)=c_{10} \ge \Theta_{0\eps}$ in $\Om$
  by (\ref{1.10}), a comparison principle thus asserts that $\Theta\le\ov{\Theta}$ in $\Om\times (0,T_\star)$, whence 
  (\ref{1.21}) and (\ref{1.18}) guarantee that indeed $T_\star=T$, and that (\ref{1.19}) holds.\abs
  We next test (\ref{1.17}) in a straightforward manner against $\Theta_t$ to see by using Young's inequality that due to
  (\ref{1.21}), and again thanks to (\ref{1.22}),
  \bas
	\io \keps(\Theta)\Theta_t^2
	+ \frac{D}{2} \frac{d}{dt} \io |\na\Theta|^2
	&=& \io h_1 \Theta \Theta_t 
	+ \io h_2 \Theta_t \\
	&\le& \eps \io \Theta_t^2
	+ \frac{1}{2\eps} \io h_1^2 \Theta^2
	+ \frac{1}{2\eps} \io h_2^2 \\
	&\le& \eps \io \Theta_t^2
	+ \frac{1}{2\eps} \cdot (c_1 R)^2 \cdot (2c_{10})^2 \cdot |\Om|
	+ \frac{1}{2\eps} \cdot (c_2 R^2 + c_9)^2 \cdot |\Om|
  \eas
  for all $t\in (0,T)$,
  whence estimating $\io \keps(\Theta) \Theta_t^2 \ge \eps\io \Theta_t^2$ for $t\in (0,T)$ we infer upon integrating in time that
  in view of (\ref{1.100}),
  \be{1.23}
	\io |\na\Theta|^2
	\le c_{11}
	\qquad \mbox{for all } t\in (0,T),
  \ee
  because $T\le 1$.
  Observing that for
  \be{1.24}
	[\Phi_2(\ovv)](\cdot,t):=u(\cdot,t):=u_{0\eps} + \int_0^t \ovv(\cdot,s) ds,
	\qquad t\in [0,T],
  \ee
  we have 
  \be{1.244}
	\|A^\frac{1}{2} u(\cdot,t)\|_{L^2(\Om)}
	\le \|A^\frac{1}{2} u_{0\eps}\|_{L^2(\Om)}
	+ \int_0^t \|A^\frac{1}{2} \ovv(\cdot,s)\|_{L^2(\Om)} ds 
	\le c_7 + R
	\qquad \mbox{for all } t\in (0,T)
  \ee
  due to (\ref{1.9}), (\ref{1.16}) and the fact that $T\le 1$, from (\ref{1.23}) when combined with (\ref{1.4}), (\ref{1.5}), 
  (\ref{1.6}), (\ref{1.9}) and (\ref{1.101}) we conclude that for
  \bas
	h_3:=\div (\D:\nas\ovv) + \div(\C:\nas u) - \div (\Theta\B) + \feps
  \eas
  we have
  \bea{1.25}
	\|h_3\|_{L^2(\Om)}
	&\le& c_3 \|A^\frac{1}{2}\ovv\|_{L^2(\Om)}
	+ c_4\|A^\frac{1}{2} u\|_{L^2(\Om)}
	+ c_5\|\na\Theta\|_{L^2(\Om)}
	+ \|\feps\|_{L^2(\Om)} \nn\\
	&\le& c_3 R + c_4\cdot (c_7+R) + c_5 c_{11}^\frac{1}{2} + c_8 
	= c_{12}
	\qquad \mbox{for all } t\in (0,T).
  \eea
  Therefore, letting
  \be{1.26}
	[\Phi\ovv](\cdot,t)
	:= e^{-tA} v_{0\eps} + \int_0^t e^{-(t-s)A} h_3(\cdot,s) ds,
	\qquad t\in [0,T], 
  \ee
  defines an element of $V^0([0,T];W_0^{2m,2}(\Om;\R^n))$ which according to (\ref{1.7}), (\ref{1.8}) and (\ref{1.91}) as well as
  the fact that $A^\frac{1}{2}$ and $e^{-tA}$ commute on $D(A^\frac{1}{2})$ for all $t>0$, satisfies
  \bea{1.27}
	\big\| [A^\frac{1}{2} \Phi \ovv](\cdot,t)\big\|_{L^2(\Om)}
	&=& \bigg\| e^{-tA} A^\frac{1}{2} v_{0\eps}
	+ \int_0^t A^\frac{1}{2} e^{-(t-s)A} h_3(\cdot,s) ds \bigg\|_{L^2(\Om)} \nn\\
	&\le& \|A^\frac{1}{2} v_{0\eps}\|_{L^2(\Om)}
	+ c_6 \int_0^t (t-s)^{-\frac{1}{2}} \|h_3(\cdot,s)\|_{L^2(\Om)} ds \nn\\
	&\le& R-1 + c_6 c_{12} \int_0^t (t-s)^{-\frac{1}{2}} ds \nn\\
	&=& R-1 + 2c_6 c_{12} t^\frac{1}{2} \nn\\
	&\le& R
	\qquad \mbox{for all } t\in (0,T) 
  \eea
  as a consequence of our second restriction on $T$ in (\ref{1.134}).
  To verify that in addition to this we have
  \be{1.28}
	\Phi(\ovv) \in C^{1-\gamma}([0,T];D(A^\gamma)),
  \ee
  possibly with less explicit knowledge on corresponding norms, we follow a line of arguments well-established in regularity
  theories of mild solutions to abstract evolution problems (\cite[Part 2]{friedman}) and firstly pick $c_{13}=c_{13}(\eps)>0$
  and $c_{14}=c_{14}(\eps)>0$ such that for all $t>0$ and each $\vp\in L^2(\Om;\R^n)$ we have
  \be{1.29}
	\|A^{1+\gamma} e^{-tA} \vp\|_{L^2(\Om)}
	\le c_{13} t^{-1-\gamma} \|\vp\|_{L^2(\Om)}
	\qquad \mbox{and} \qquad
	\|A^\gamma e^{-tA} \vp\|_{L^2(\Om)}
	\le c_{14} t^{-\gamma} \|\vp\|_{L^2(\Om)},
  \ee
  and for fixed $t_0\in [0,T)$ and $t\in (t_0,T]$ we represent
  \bea{1.30}
	[A^\gamma \Phi(\ovv)](\cdot,t) - [A^\gamma \Phi (\ovv)](\cdot,t_0)
	&=& [e^{-tA}-e^{-t_0 A}] A^\gamma v_{0\eps}
	+ \int_0^{t_0} A^\gamma [e^{-(t-s)A}-e^{-(t_0-s)A}] h_3(\cdot,s) ds \nn\\
	& & + \int_{t_0}^t A^\gamma e^{-(t-s)A} h_3(\cdot,s) ds.
  \eea
  Noting that
  \bas
	\big\| [e^{-tA}-e^{-t_0 A}] A^\gamma v_{0\eps} \big\|_{L^2(\Om)}
	= \bigg\| - \int_{t_0}^t e^{-\sig A} A^{1+\gamma} v_{0\eps} d\sig \bigg\|_{L^2(\Om)}
	\le \|A^{1+\gamma} v_{0\eps}\|_{L^2(\Om)} \cdot (t-t_0)
  \eas
  by (\ref{1.7}), that
  \bas
	\bigg\| \int_0^{t_0} A^\gamma [e^{-(t-s)A}-e^{-(t_0-s)A}] h_3(\cdot,s) ds \bigg\|_{L^2(\Om)}
	&=& \bigg\| - \int_0^{t_0} \int_{t_0}^t A^{1+\gamma} e^{-(\sig-s)A} h_3(\cdot,s) d\sig ds \bigg\|_{L^2(\Om)} \\
	&\le& c_{13} c_{12} \int_0^{t_0} \int_{t_0}^t (\sig-s)^{-1-\gamma} d\sig ds \\
	&=& \frac{c_{13} c_{12}}{\gamma(1-\gamma)} \cdot \big\{ (t-t_0)^{1-\gamma} + t_0^{1-\gamma} - t^{1-\gamma} \big\} \\
	&\le& \frac{c_{13} c_{12}}{\gamma(1-\gamma)} (t-t_0)^{1-\gamma}
  \eas
  due to (\ref{1.29}), (\ref{1.25}) and the condition $\gamma\in (0,1)$, and that
  \bas
	\bigg\| \int_{t_0}^t A^\gamma e^{-(t-s)A} h_3(\cdot,s) ds \bigg\|_{L^2(\Om)}
	\le c_{14} c_{12} \int_{t_0}^t (t-s)^{-\gamma} ds
	= \frac{c_{14} c_{12}}{1-\gamma} (t-t_0)^{1-\gamma}
  \eas
  thanks to (\ref{1.29}), (\ref{1.25}) and the inequality $\gamma<1$, from (\ref{1.30}) we infer that, indeed,
  \bas
	\big\| [A^\gamma \Phi(\ovv)](\cdot,t) - [A^\gamma \Phi(\ovv](\cdot,t_0) \big\|_{L^2(\Om)}
	\le \Big\{ \|A^{1+\gamma} v_{0\eps}\|_{L^2(\Om)} + \frac{c_{13} c_{12}}{\gamma(1-\gamma)} + \frac{c_{14} c_{12}}{1-\gamma}
		\Big\} \cdot (t-t_0)^{1-\gamma}
  \eas
  for any such $t_0$ and $t$.
  Since $[\Phi(\ovv)](\cdot,0)= v_{0\eps} \in D(A^\gamma)$, this clearly entails (\ref{1.28}) and thus, when combined with (\ref{1.27}),
  that $\Phi$ in fact maps $S$ into itself, and that $\ov{\Phi (S)}$ actually is a compact subst of $X$ thanks to the
  Arzel\`a-Ascoli theorem, because (\ref{1.122}) asserts that $1-\gamma>\vt$, and that $D(A^\gamma)$ is compactly embedded into
  $W_0^{2m,2}(\Om;\R^n)$ due to the inequality $\gamma>\frac{1}{2}$ (\cite{henry}).\abs
  To finally see that $\Phi$ is continuous, we assume that $(\ovk)_{k\in\N} \subset S$ and $\ovv\in S$ are such that 
  $\ovk\to \ovv$ in $X$ as $k\to\infty$.
  Then again thanks to the continuity of the embedding $W^{2m,2}(\Om)\hra C^{1+\iota_1}(\bom)$, the functions 
  $h_{1k}:=-\lan \B,\nas \ovk\ran$ and $h_{2k}:=\lan\D:\nas\ovk,\nas\ovk\ran + \geps$ in the identity
  $\keps(\Theta_k) \Theta_{kt} = D\Del\Theta_k + h_{1k}(x,t) \Theta_k + h_{2k}(x,t)$, satisfied along with homogeneous
  Neumann boundary conditions and the initial condition $\Theta_k(\cdot,0)=\Theta_{0\eps}$ by $\Theta_k:=\Phi_1(\ovk)$
  for all $k\in\N$ by definition of $\Phi_1$, have the property that $(h_{1k})_{k\in\N}$ and $(h_{2k})_{k\in\N}$ are
  bounded in $C^{\iota_3,\frac{\iota_3}{2}}(\bom\times [0,T])$ with some $\iota_3\in (0,1)$, so that parabolic Schauder theory
  (\cite{LSU}) yields boundedness of $(\Theta_k)_{k\in\N}$ in $C^{2+\iota_4,1+\frac{\iota_4}{2}}(\bom\times [0,T])$
  with some $\iota_4\in (0,1)$.
  As solutions to (\ref{1.17}) are unique within the class of nonnegative functions from $C^{2,1}(\bom\times [0,T])$
  due to the maximum principle, thanks to the Arzel\`a-Ascoli theorem this implies that $\Theta_k \to \Theta=\Phi_1(\ovv)$
  in $C^{2,1}(\bom\times [0,T])$ as $k\to\infty$.  
  Along with an argument parallel to that in (\ref{1.244}), this particularly ensures that for 
  $h_{3k}:=\div(\D:\nas\ovk) + \div (\C:\nas \Phi_2(\ovk)) - \div(\Phi_1(\ovk)\B) + \feps$, $k\in\N$,
  we have $h_{3k} \to \wh{h}_3:=\div (\D:\nas\ovv) + \div (\C:\nas \Phi_2(\ovv)) - \div (\Phi_1(\ovv)\B) + \feps$
  in $L^\infty((0,T);L^2(\Om))$ as $k\to\infty$, and that thus, once more by (\ref{1.7}),
  \bas
	\big\| [\Phi (\ovk)](\cdot,t) - [\Phi(\ovv)](\cdot,t) \big\|_{L^2(\Om)}
	&=& \bigg\| \int_0^t e^{-(t-s)A} [h_{3k}(\cdot,s)-\wh{h}_3(\cdot,s)] ds \bigg\|_{L^2(\Om)} \\
	&\le& \int_0^t \|h_{3k}(\cdot,s)-\ovv{h}_3(\cdot,s)\|_{L^2(\Om)} ds
	\to 0
	\qquad \mbox{as } k\to\infty.
  \eas
  By compactness of $\ov{\Phi(S)}$ in $X$, this ensures that, indeed, $\Phi (\ovk) \to \Phi(\ovv)$ in $X$ as $k\to\infty$.\abs
  In conclusion, we may employ the Schauder fixed point theorem to find $\veps\in S$ such that $\Phi(\veps)=\veps$,
  and writing $\ueps:=\Phi_2(\veps)$ and $\Teps:=\Phi_1(\veps)$ we readily see by means of a straightforward bootstrap argument
  that the triple $(\veps,\ueps,\Teps)$ actually enjoys regularity features of the form in (\ref{1.1}) and solves (\ref{0eps})
  in the classical sense in $\Om\times (0,T)$.
  As the above choice of $T$ depends on the initial data exclusively through the quantities
  $\|v_{0\eps}\|_{W^{2m,2}(\Om)}, \|u_{0\eps}\|_{W^{2m,2}(\Om)}, \|\na\Theta_{0\eps}\|_{L^2(\Om)}$ and 
  $\|\Theta_{0\eps}\|_{L^\infty(\Om)}$, a standard extension procedure finally yields $\tme\in (0,\infty]$ and a classical
  solution of (\ref{0eps}) in $\Om\times (0,\tme)$ with the properties in (\ref{1.1}) and (\ref{ext}).
  Strict positivity of $\Teps$ in $\bom\times [0,\tme)$, finally, is a consequence of the strong maximum principle
  and the fact that $\Theta_{0\eps}>0$ in $\bom$.
\qed

\bigskip

\section*{Declarations}
{\bf Funding.} \quad
The author acknowledges support of the Deutsche Forschungsgemeinschaft (Project No. 444955436).\abs
{\bf Conflict of interest statement.} \quad
The author declares that he has no conflict of interest.
{\bf Data availability statement.} \quad
Data sharing is not applicable to this article as no datasets were
generated or analyzed during the current study.


\small

\begin{thebibliography}{99}
%
\bibitem{villani}
  \sc Alexandre, R., Villani, C.:  
  \it On the Boltzmann equation for long-range interactions. 
  \rm Comm. Pure Appl. Math. {\bf 55}, 30–70 (2002)
\bibitem{amann}
  \sc Amann, H.:
  \it Nonhomogeneous linear and quasilinear elliptic and parabolic boundary value problems.
  \rm In: Function spaces, differential operators and nonlinear analysis (Friedrichroda, 1992), Volume 133 of 
  {\em Teubner-Texte Math.},  Teubner, Stuttgart, 9-126 (1993)
\bibitem{friedrich2022}
  \sc Badal, R., Friedrich, M., Kru\v{c}ik, M.:
  \it Nonlinear and linearized models in thermoviscoelasticity.
  \rm Arch. Ration. Mech. Anal. {\bf 247}, 5 (2023)
\bibitem{friedrich2024}
  \sc Badal, R., Friedrich, M., Kru\v{c}ik, M., Machill, L.:
  \rm Positive temperature in nonlinear thermoviscoelasticity and the derivation of linearized models.
  \rm J. Math. Pures Appl. {\bf 202}, 103751 (2025)
\bibitem{bartels_roubicek}
  \sc Bartels, S., Roub\'{\i}\v{c}ek, T.:
  \it Thermoviscoplasticity at small strains.
  \rm ZAMM Z. Angew. Math. Mech. {\bf 88}, 735-754 (2008)
\bibitem{cieslak_MAAN}
  \sc Bies, P.:, Cie\'slak, T.:
  \it Time-asymptotics of a heated string.
  \rm Math. Ann. {\bf 391}, 5941-5964 (2025)
\bibitem{cieslak_fuest_lankeit}
  \sc Bies, P.M., Cie\'slak, T., Fuest, M., Lankeit, J., Muha, B., Trifunovic, S.:
  \it Existence, uniqueness, and long-time asymptotic behavior of regular solutions in multidimensional thermoelasticity.
  \rm Preprint, {\tt arxiv2507.20794v2}
\bibitem{blanchard_guibe}
  \sc Blanchard, D., Guib\'e, O.:
  \it Existence of a solution for a nonlinear system in thermoviscoelasticity.
  \rm Adv. Differential Equations {\bf 5}, 1221-1252 (2000)
\bibitem{bonetti_bonfanti}
  \sc Bonetti, E., Bonfanti, G.:
  \it Existence and uniqueness of the solution to a 3D thermoviscoelastic system.
  \rm Electron. J. Differential Equations 2003, 50 (2003)
\bibitem{chen_hoffmann}
  \sc Chen, Z., Hoffmann, K.-H.:
  \it On a one-dimensional nonlinear thermoviscoelastic model for structural phase transitions in shape memory alloys. 
  \rm J. Differential Equations {\bf 112}, 325-350 (1994)
\bibitem{cieslak_CVPD}
  \sc Cie\'slak, T., Muha, B, Trifunovi\'c, S.A.:
  \it Global weak solutions in nonlinear 3D thermoelasticity. 
  \rm Calc. Var. Partial Differential Equations {\bf 63}, 26 (2024)
\bibitem{claes_lankeit_win}
  \sc Claes, L., Lankeit, J., Winkler, M.: 
  \it A model for heat generation by acoustic waves in piezoelectric materials: Global large-data solutions.
  \rm Math.~Models Methods Appl.~Sci., to appear
\bibitem{dafermos}
  \sc Dafermos, C.M.:
  \it Global smooth solutions to the initial boundary value problem for the equations of one-dimensional thermoviscoelasticity.
  \rm SIAM J. Math. Anal. {\bf 13}, 397-408 (1982)
\bibitem{dafermos_hsiao_smooth}
  \sc Dafermos, C.M., Hsiao, L.:
  \it Global smooth thermomechanical processes in one-dimensional nonlinear thermoviscoelasticity.
  \rm Nonlinear Anal. {\bf 6}, 435-454 (1982)
\bibitem{debye}
  \sc Debye, P.:
  \it Zur Theorie der spezifischen W\"armen.
  \rm Annalen der Physik {\bf 344}, 789-839 (1912)
\bibitem{solombrino}
  \sc Di Fratta, G., Solombrino, F.:
  \it Korn and Poincar\'e-Korn inequalities: a different perspective.
  \rm Proc. Am. Math. Soc. {\bf 153}, 143-159 (2025)
\bibitem{diperna_lions}
  \sc Di Perna, R.-J., Lions, P.-L.: \it
  On the Cauchy problem for Boltzmann equations: Global existence and weak stability.
  \rm Ann. Math. {\bf 130}, 321-366 (1989)
\bibitem{drotziger}
  \sc Drotziger, F., et al.:
  \it Glassy anomalies in the heat capacity of an ordered $2-$bromobenzophenone single crystal.
  \rm Physical Review Letters {\bf 120}, 215502 (2018)
\bibitem{feireisl}
  \sc Feireisl, E., Novotny, A.: 
  \it Singular Limits in Thermodynamics of Viscous Fluids. 
  \rm Advances in Mathematical Fluid Mechanics. Birkh\"auser, Basel (2009)
\bibitem{friedman}
   \sc Friedman, A.: \it Partial Differential Equations.
  \rm Holt, Rinehart \& Winston, New York, 1969
\bibitem{jost_pde}
  \sc Jost, J.: \it Partial differential equations. 
  \rm Graduate Texts in Mathematics, Vol. 214. Springer, New York, second edition, 2007
\bibitem{gawinecki}
  \sc Gawinecki, J.A.:
  \it Global existence of solutions for non-small data to non-linear spherically symmetric thermoviscoelasticity.
  \rm Math. Methods Appl. Sci. {\bf 26}, 907-936 (2003)
\bibitem{gaw_zaj_TMNA}
  \sc Gawinecki, J.A., Zajaczkowski, W.M.:
  \it Global existence of solutions to the nonlinear thermoviscoelasticity system with small data.
  \rm Topol. Methods Nonlinear Anal. {\bf 39}, 263-284 (2012)
\bibitem{gaw_zaj_CPAA}
  \sc Gawinecki, J.A., Zajaczkowski, W.M.:
  \it Global regular solutions to two-dimensional thermoviscoelasticity.
  \rm Commun. Pure Appl. Anal. {\bf 15}, 1009-1028 (2016)
\bibitem{griffel}
  \sc Griffel, M., Skochdopole, R.E., Spedding, F.H.:
  \it The Heat Capacity of Gadolinium from 15 to $355^\circ K$.
  \rm Phys. Rev. {\bf 93}, 657 (1954)
\bibitem{guo_zhu}
  \sc Guo, B., Zhu, P.:
  \it Global existence of smooth solution to nonlinear thermoviscoelastic system with clamped boundary conditions in solid-like
  materials.
  \rm Commun. Math. Phys. {\bf 203}, 365-383 (1999)
\bibitem{henry}
  \sc Henry, D.: \it Geometric Theory of Semilinear Parabolic Equations.
  \rm Lecture Notes in Mathematics. 840. Springer, Berlin-Heidelberg-New York, 1981
\bibitem{hsiao_luo}
  \sc Hsiao, L., Luo, T.:
  \it Large-time behavior of solutions to the equations of one-dimensional nonlinear thermoviscoelasticity.
  \rm Q. Appl. Math. {\bf 56}, 201-219 (1998)
\bibitem{jiang_QAM1993}
  \sc Jiang, S.:
  \it Global large solutions to initial boundary value problems in one- dimensional nonlinear thermoviscoelasticity. 
  \rm Q. Appl. Math. {\bf 51}, 731-744 (1993)
\bibitem{kim}
  \sc Kim, J.U.: \it Global existence of solutions of the equations of one-dimensional thermoviscoelasticity
  with initial data in $BV$ and $L^1$.
  \rm Ann. Scuola Norm. Sup. Pisa Cl. Sci. {\bf 10}, 357-429 (1983)
\bibitem{LSU}
  \sc Ladyzenskaja, O. A., Solonnikov, V. A., Ural'ceva, N. N.:
  \it Linear and Quasi-Linear Equations of Parabolic Type.
  \rm Amer. Math. Soc. Transl., Vol. {\bf 23}, Providence, RI, 1968
\bibitem{lankeit_win_NoDEA}
  \sc Lankeit, J., Winkler, M.:
  \it A generalized solution concept for the Keller-Segel system with logarithmic sensitivity:
  global solvability for large nonradial data. 
  \rm Nonlinear Differential Eq. Appl. NoDEA {\bf 24}, 49 (2017)
\bibitem{lions}
  \sc Lions, J.L.: 
  \it Quelques m\'ethodes de r\'esolution des probl\`emes aux limites non lin\'eaires. 
  \rm Etudes mathematiques. Dunod/Gauthier-Villars, Paris, XX, 1969
\bibitem{mielke_roubicek}
  \sc Mielke, A., Roubi\v{c}ek, T.:
  \it Thermoviscoelasticity in Kelvin-Voigt rheology at large strains.
  \rm Arch. Ration. Mech. Anal. {\bf 238}, 1-45 (2020)
\bibitem{owczarek_wielgos}
  \sc Owczarek, S., Wielgos, K.:
  \it On a thermo-visco-elastic model with nonlinear damping forces and $L^1$ temperature data. 
  \rm Math. Methods Appl. Sci. {\bf 46}, 9966-9999 (2023)
\bibitem{paoli_petrov}
  \sc Paoli, L., Petrov, A.:
  \it  Global existence result for thermoviscoelastic problems with hysteresis.
  \rm Nonlinear Anal. Real World Appl. {\bf 13}, 524-542 (2012)
\bibitem{paoli_petrov_gamm}
  \sc Paoli, L., Petrov, A.:
  \it Thermodynamics of multiphase problems in viscoelasticity.
  \rm GAMM-Mitt. {\bf 35}, 75-90 (2012)
\bibitem{pawlow2000}
  \sc Pawlow, I.:
  \it Three-dimensional model of thermomechanical evolution of shape memory material.
  \rm Control and Cybernetics {\bf 29}, 341-365 (2000)
\bibitem{pawlow_zajaczkowski_SIMA}
  \sc Pawlow, I., Zajaczkowski, W.M.:
  \it Global regular solutions to a Kelvin-Voigt type thermoviscoelastic system.
  \rm SIAM J. Math. Anal. {\bf 45}, 1997-2045 (2013)
\bibitem{paw_zaj2017}
  \sc Pawlow, I., Zajaczkowski, W.M.:
  \it Global regular solutions to three-dimensional thermo-visco-elasticity with nonlinear temperature-dependent specific heat.
  \rm Commun. Pure Appl. Anal. {\bf 16}, 1331-1371 (2017)
\bibitem{racke_zheng}
  \sc Racke, R., Zheng, S.:
  \it Global existence and asymptotic behavior in nonlinear thermoviscoelasticity.
  \rm J. Differential Equations {\bf 134}, 46-67 (1997)
\bibitem{rossi_roubicek}
  \sc Rossi, R., Roubi\v{c}ek, T.:
  \it Adhesive contact delaminating at mixed mode, its thermodynamics and analysis.
  \rm Interfaces Free Bound. {\bf 15}, 1-137 (2013)
\bibitem{roubicek}
  \sc Roubi\v{c}ek, T.:
  \it Thermo-visco-elasticity at small strains with $L^1$-data.
  \rm Quart. Appl. Math. {\bf 67}, 47-71 (2009)
\bibitem{roubicek_SIMA}
  \sc Roubi\v{c}ek, T.:
  \it Thermodynamics of rate-independent processes in viscous solids at small strains.
  \rm SIAM J. Math. Anal. {\bf 42}, 256-297 (2010)
\bibitem{roubicek_nodea2013}
  \sc Roubi\v{c}ek, T.:
  \it Nonlinearly coupled thermo-visco-elasticity.
  \rm NoDEA Nonlinear Differential Equations Appl. {\bf 20}, 1243-1275 (2013)
\bibitem{roubicek_dcdss13}
  \sc Roubi\v{c}ek, T.:
  \it Thermodynamics of perfect plasticity.
  \rm Discrete Contin. Dyn. Syst. Ser. S {\bf 6}, 193-2014 (2013)
\bibitem{shibata}
  \sc Shibata, Y.:
  \it Global in time existence of small solutions of nonlinear thermoviscoelastic equations.
  \rm Math. Methods Appl. Sci. {\bf 18}, 871-895 (1995)
\bibitem{troyanchuk}
  \sc Troyanchuk, I.O., et al.:
  \it Specific heat anomalies in $La_{1-x}Sr_{x}MnO_{3}$.
  \rm Physical Review B {\bf 71}, 224432 (2005)
\bibitem{wang_win_JLMS}
  \sc Wang, Y., Winkler, M.:
  \it An interpolation inequality involving $L\log L$ spaces and application to the characterization of blow-up behavior 
  in a two-dimensional Keller-Segel-Navier-Stokes system.
  \rm J.~London Math.~Soc. {\bf 109}, e12885 (2024)
\bibitem{watson}
  \sc Watson, S.J.:
  \it Unique global solvability for initial-boundary value problems in one-dimensional nonlinear thermoviscoelasticity.
  \rm Arch. Ration. Mech. Anal. {\bf 153}, 1-37 (2000)
\bibitem{win_SIMA2015}
  \sc Winkler, M.:
  \it Large-data global generalized solutions in a chemotaxis system with tensor-valued sensitivities. \quad
  \rm SIAM J.~Math.~Anal. {\bf 47}, 3092-3115 (2015)
\bibitem{win_SIMA2020}
  \sc Winkler, M.: 
  \it Small-mass solutions in the two-dimensional Keller-Segel system coupled to the Navier-Stokes equations.
  \rm SIAM J.~Math.~Anal. {\bf 52}, 2041-2080 (2020)
\bibitem{zhigun}
  \sc Zhigun, A.:
  \it Generalized global supersolutions with mass control for systems with taxis.
  \rm SIAM J. Math. Anal. {\bf 51}, 2425-2443 (2019)
%
\end{thebibliography}
\end{document}